\documentclass{amsart}
\usepackage{mathrsfs, amsmath, graphics}
\usepackage{amssymb, comment}

\usepackage{graphicx}

\usepackage{epsfig}

\usepackage{amsfonts}

\usepackage{amssymb,amscd}
\usepackage{tikz}
\usepackage{verbatim}
\usepackage{enumerate}
\usepackage{booktabs}
\usepackage[all]{xy}
\usepackage{subfigure}

\usepackage{caption}
\usepackage{subfigure}
\usepackage{marginnote}

\newtheorem{thm}{Theorem}[section]
\newtheorem{lemma}[thm]{Lemma}

\newtheorem{prop}[thm]{Proposition}

\theoremstyle{definition}
\newtheorem{defn}[thm]{Definition}

\newtheorem{question}[thm]{Question}

\theoremstyle{remark}
\newtheorem{remark}[thm]{Remark}
\usepackage{enumerate}
\numberwithin{equation}{section}

\theoremstyle{notation}

\usepackage{floatrow}
\newfloatcommand{capbtabbox}{table}[][\FBwidth]

\usepackage{lineno}

\usepackage{graphicx}
\usepackage{mathrsfs}
\usepackage{epsfig}
\usepackage{amsmath}
\usepackage{amsfonts}
\usepackage{amssymb}
\usepackage{amssymb,amscd}
\usepackage{tikz}
\usepackage{verbatim}
\usepackage{enumerate}
\usepackage{booktabs}
\usepackage[all]{xy}

\theoremstyle{definition}

\theoremstyle{remark}

\numberwithin{equation}{section}

\title{Complexification of an infinite volume Coxeter tetrahedron}

\author{Jiming Ma}
\address{School of Mathematical Sciences, Fudan University, Shanghai, 200433, P. R. China}
\email{majiming@fudan.edu.cn}



\date{Jan. 19, 2023}

\thanks{Jiming Ma was partially supported by  NSFC 12171092. \\
}

\begin{document}



\maketitle

\begin{abstract}

	Let $T$ be an infinite volume Coxeter tetrahedron in  three dimensional real hyperbolic space ${\bf H}^{3}_{\mathbb R}$ with two opposite right-angles and  the other angles are all zeros.  Let $G$ be the Coxeter group of $T$, so 
 $$G=\left\langle \iota_1, \iota_2, \iota_3, \iota_4 \Bigg| \begin{array} {c}    \iota_1^2=  \iota_2^2 = \iota_3^2=\iota_4^2=id,  \\ [3 pt]
(\iota_1 \iota_3)^{2}=(\iota_2 \iota_4)^{2}=id
	\end{array}\right\rangle$$
	as an  abstract  group.
	We  study type-preserving representations $\rho: G \rightarrow \mathbf{PU}(3,1)$,  where $\rho( \iota_{i})=I_{i}$ is a complex reflection fixing a  complex hyperbolic plane   in  three dimensional complex hyperbolic space ${\bf H}^{3}_{\mathbb C}$   for $1 \leq i \leq 4$. The
	   moduli space $\mathcal{M}$ of these representations is parameterized by $\theta \in  [\frac{5 \pi}{6}, \pi]$. In particular, $\theta=\frac{5 \pi}{6}$  and $\theta=\pi$ degenerate to  ${\bf H}^{2}_{\mathbb C}$-geometry and  ${\bf H}^{3}_{\mathbb R}$-geometry   respectively.
	Via Dirichlet  domains, we   show   $\rho=\rho_{\theta}$ is a discrete and faithful representation of the group $G$ for all $\theta \in [\frac{5 \pi}{6}, \pi]$. This is the first nontrivial  moduli space in   three dimensional complex  hyperbolic space that has been studied completely.


\end{abstract}


\section{Introduction}\label{sec:intro}

\subsection{Motivation}\label{subsection:motivation}

Hyperbolic $n$-space ${\bf H}^{n}_{\mathbb R}$ is the unique complete simply
connected Riemannian $n$-manifold with all sectional curvatures $-1$.  Complex hyperbolic $n$-space ${\bf H}^{n}_{\mathbb C}$ is the unique complete simply
connected K\"ahler  $n$-manifold with all holomorphic  sectional curvatures $-1$. But the Riemannian  sectional curvatures of a complex hyperbolic space  are no longer constant, which are
pinched between $-1$ and $-\frac{1}{4}$. This makes complex hyperbolic  geometry much more difficult to study.
The   holomorphic isometry group of ${\bf H}^{n}_{\mathbb C}$ is $\mathbf{PU}(n,1)$, the   orientation preserving isometry group of ${\bf H}^{n}_{\mathbb R}$ is $\mathbf{PO}(n,1)$. Besides  ${\bf H}^{n}_{\mathbb R}$ is a totally geodesic submanifold of ${\bf H}^{n}_{\mathbb C}$,  $\mathbf{PO}(n,1)$ is a natural subgroup of  $\mathbf{PU}(n,1)$.


Over the last sixty years the theory of Kleinian groups, that  is,   deformations of groups into $\mathbf{PO}(3,1)$,  has flourished because of its close connections with low dimensional topology and geometry. More precisely, pioneered by
Ahlfors and Bers in the 1960's, 
 Thurston formulated a conjectural classification scheme
for all hyperbolic 3-manifolds with finitely generated fundamental groups in the late 1970's.
The conjecture predicted that an infinite volume hyperbolic 3-manifold with finitely
generated fundamental group is uniquely determined by its topological type and its
end invariants. Thurston's conjecture is completed by a series of works of many mathematicians,   which is one of the most great breakthrough  in 3-manifolds theory.  See Minsky's ICM talk 	\cite{Minsky:2006} for related topics and the reference.

There are also some remarkable works on deformations of groups into $\mathbf{PU}(2,1)$.
Let $\Delta(p,q,r)$ be the abstract $(p,q,r)$ reflection triangle group with the presentation
$$\Delta(p,q,r)=\langle \sigma_1, \sigma_2, \sigma_3 | \sigma^2_1=\sigma^2_2=\sigma^2_3=(\sigma_2 \sigma_3)^p=(\sigma_3 \sigma_1)^q=(\sigma_1 \sigma_2)^r=id \rangle,$$
where $p,q,r$ are positive integers or $\infty$ satisfying $$\frac{1}{p}+\frac{1}{q}+\frac{1}{r}<1.$$ If $p,q$ or $r$ equals $\infty$, then
the corresponding relation does not appear.  The  ideal triangle group is the case that $p=q=r=\infty$.
A \emph{$(p,q,r)$ complex hyperbolic triangle group} is a representation $\rho$ of $\Delta(p,q,r)$ into $\mathbf{PU}(2,1)$
where the generators fix complex lines, we denote $\rho(\sigma_{i})$ by $I_{i}$.
 Goldman and Parker  initiated the study of the deformations of ideal triangle group into  $\mathbf{PU}(2,1)$ in \cite{GoPa}.
They  gave an interval  in the moduli space of complex hyperbolic ideal triangle groups, for points in this interval  the corresponding representations are discrete and faithful.
They conjectured that a complex hyperbolic ideal triangle group $\Gamma=\Delta_{\infty,\infty, \infty}=\langle I_1, I_2, I_3 \rangle$ is discrete and faithful if and only if $I_1 I_2 I_3$ is not elliptic. Schwartz proved  Goldman-Parker's conjecture in \cite{Schwartz:2001ann, schwartz:2006}. 
Richard Schwartz has also  conjectured the necessary and sufficient condition for a general complex hyperbolic  triangle group $\Delta_{p,q,r}=\langle I_1,I_2,I_3\rangle < \mathbf{PU}(2,1)$ to be a discrete and faithful  representation of $\Delta(p,q,r)$. See Schwartz's ICM talk \cite{Schwartz-icm} for related topics and the reference. Schwartz's conjecture has been proved in a few cases \cite{ParkerWX:2016, ParkerWill:2017}.



 From above we know one  way to study discrete subgroups of  $\mathbf{PU}(n,1)$ is the deformations of a well-understood  representation. From a  finitely presented abstract  group $G$, and a discrete faithful representation $\rho_{0}: G \rightarrow \mathbf{PU}(n,1)$,  we may deform $\rho_{0}$  to $\rho_{1}: G \rightarrow \mathbf{PU}(n,1)$ along a path. We are interested in   whether  $\rho_{1}$ is discrete and faithful. Moreover, even when  $\rho_{1}$ is not faithful, but it also has the chance to be discrete. This case is very interesting, since if we are lucky, we have the  opportunity to get  a complex hyperbolic lattice at  $\rho_{1}$ \cite{dpp:2016, dpp:2021}.

One of the most important questions in complex hyperbolic geometry is
 the existence  of (infinitely many commensurable classes of)  non-arithmetic complex hyperbolic lattices  \cite{Margulis, Fisher:2021}. This is notorious difficult comparing to its real hyperbolic counterpart	\cite{Gromov-PS:1992}. As results of forty years' hard work, people only found 22
 commensurable classes of  non-arithmetic complex hyperbolic lattices in $\mathbf{PU}(2,1)$ \cite{DeligneMostow:1986,dpp:2016, dpp:2021}, and
  2
 commensurable classes of  non-arithmetic complex hyperbolic lattices in $\mathbf{PU}(3,1)$ \cite{DeligneMostow:1986, Deraux:2020}. Both $\mathbf{PO}(3,1)$  and $\mathbf{PU}(2,1)$ are subgroups of $\mathbf{PU}(3,1)$. It is reasonable that deformations of some discrete groups in $\mathbf{PO}(3,1)$  into the larger group $\mathbf{PU}(3,1)$  may give   some discrete, but not faithful representations, which in turn   have the opportunity to
 give some new  ${\bf H}^3_{\mathbb C}$-lattices  as pioneered  by \cite{DeligneMostow:1986, dpp:2016, dpp:2021}.  The author remarks that comparing to results in  ${\bf H}^{2}_{\mathbb C}$-geometry  	\cite{Schwartz:2001, Schwartz:2001, DerauxF:2015}, any  discrete deformations of a group $G$ into $\mathbf{PU}(3,1)$ with some accidental parabolic or elliptic element is also highly interesting.  For instance, we have no non-trivial example of a 5-manifold $N$ which  admits uniformizable CR-structures now. 
  Here by nontrivial example, we mean  $N$ is neither  diffeomorphic to the $\mathbb{S}^3$-bundle over a ${\bf H}^{2}_{\mathbb R}$-manifold $F^2$, the trivial $\mathbb{S}^2$-bundle over a ${\bf H}^{3}_{\mathbb R}$-manifold $Y^3$, nor  diffeomorphic to a $\mathbb{S}^1$-bundle over a ${\bf H}^{2}_{\mathbb C}$-manifold  $X^4$.


In this paper, we  study   deformations of groups into $\mathbf{PU}(3,1)$ via Dirichlet domains,  which is much more difficult and richer than deformations of  groups into $\mathbf{PO}(3,1)$  and $\mathbf{PU}(2,1)$.  By this we mean:
\begin{itemize}
		\item It is well known that the space of discrete and faithful  representations of a group into $\mathbf{PO}(3,1)$ has fractal  boundary in general. For example,  the so called Riley slice has a beautiful fractal  boundary in $\mathbb{C}$ (see Page VIII of  \cite{ASWY:2019});
	\item  People tend to guess that  the space of discrete and faithful  representations of a group into $\mathbf{PU}(2,1)$ has piece-wise smooth boundary (at least when the deformation space has two dimension). For one of the tractable  cases, the so called complex Riley slice,  which is 2-dimensional, see \cite{ParkerWill:2017};
	\item Moreover, there are very few results on the space of discrete and faithful  representations of a group into $\widehat{\mathbf{PU}(2,1)}$. Here  $\widehat{\mathbf{PU}(2,1)}$ is the full isometry group of  ${\bf H}^2_{\mathbb C}$. To the author's knowledge, the only complete  classification result   is in \cite{Falbelparker:2003}. Where Falbel-Parker completed the study on the space of discrete and faithful  representations of $\mathbb{Z}_{2}*\mathbb{Z}_{3}$  into $\widehat{\mathbf{PU}(2,1)}$ (with one additional parabolic element), the moduli space is 1-dimensional.

\end{itemize}
So deformations of groups into $\mathbf{PU}(3,1)$   may have fractal boundaries in general (at least   when the deformation spaces have large dimensions), but we are very far from understanding them.






\subsection{Main result of the paper}  \label{subsec:mainresult}
In this article,   we  complexify  a Coxeter tetrahedron in the real hyperbolic space  ${\bf H}^{3}_{\mathbb R}$  into the complex hyperbolic space  ${\bf H}^{3}_{\mathbb C}$.

 Using the upper space model of  ${\bf H}^{3}_{\mathbb R}$, a Coxeter tetrahedron in  ${\bf H}^{3}_{\mathbb R}$ is determined by four round circles $C_i$ in $\mathbb{C}$ for $1 \leq i \leq 4$. Where each circle $C_i$ is the ideal boundary of a totally geodesic  ${\bf H}^{2 }_{\mathbb R}\hookrightarrow {\bf H}^{3}_{\mathbb R}$, and $\mathbb{C}\cup \infty$ is the ideal boundary of  ${\bf H}^{3}_{\mathbb R}$. In the right subfigure of Figure \ref{figure:tetrahedron}, there is a  configuration of four  round circles in $\mathbb{C}$:
 \begin{itemize}
 	\item the red, purple,  green and blue circles are $C_1$, $C_2$, $C_3$ and $C_4$ respectively;
 	
 	\item   $C_1$ and  $C_3$ intersect perpendicularly at  two points.  $C_2$ and  $C_4$ also  intersect perpendicularly at two points;
 	
 	\item  $C_i$ is tangent to $C_{i+1}$   for $1 \leq  i \leq 4$ mod $4$.
 	
 	\end{itemize}

This configuration of round circles in $\mathbb{C}$ can be obtained as follows. We first take $C_1$ and  $C_3$ with the same radius and  intersect perpendicularly at  two point, see the red and green circles in  right subfigure of Figure \ref{figure:tetrahedron}. We take original round circles $C_2$ and $C_4$ which are 
coincident, such that they are tangent to both 
$C_1$ and  $C_3$ (the original $C_2$ and $C_4$ are both the dashed circle in Figure \ref{figure:tetrahedron}). We view $C_2$ and $C_4$ have angle $\pi$ at this configuration. Then we make $C_2$ with bigger radius, and the center of it to the left of the 
original one. We also make $C_4$ with bigger but the same radius as $C_2$, and the center of it to the right of the 
original one. Moreover each of $C_2$ and $C_4$ is also  tangent to both 
$C_1$ and  $C_3$. When the radii of $C_2$ and $C_4$ diverge to the infinity, that is, the limiting case is that $C_2$ and $C_4$ are two vertical lines tangent to both $C_1$ and $C_3$,  the angle between $C_2$ and $C_4$ converges to zero. So at an
intermediate time, the angle between $C_2$ and $C_4$ is $\frac{\pi}{2}$. 
It is well-known  that there is a unique  configuration of round circles in $\mathbb{C}$ satisfies the above angle conditions up to $\mathbf{PSL}(2,\mathbb{C})$-action (equivalently up to  $\mathbf{PO}(3,1)$-action), see \cite{VinbergS:1993}.

\begin{figure}
	\begin{center}
		\begin{tikzpicture}
		\node at (0,0) {\includegraphics[width=12cm,height=5cm]{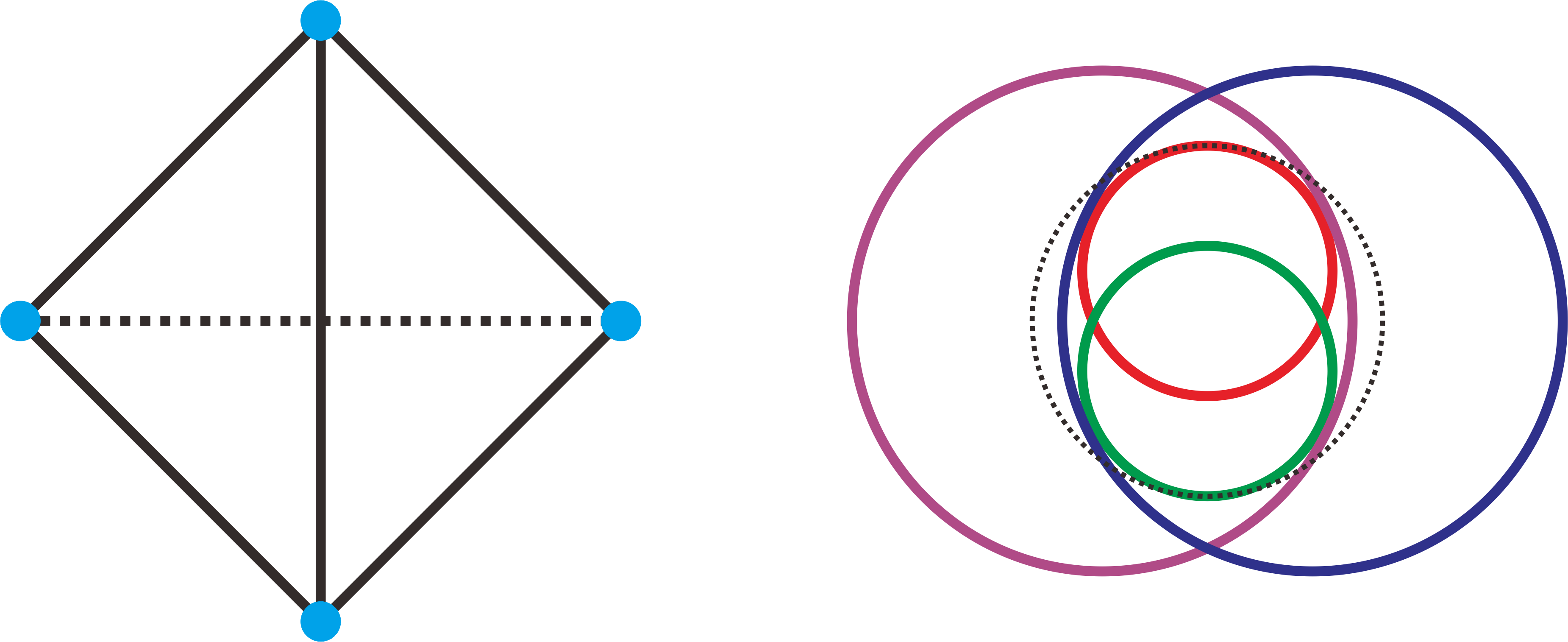}};
		\node at (-5.1,1.3){\LARGE $\frac{\pi}{2}$};
		\node at (-2.2,-1.8){\LARGE $\frac{\pi}{2}$};

		\node at (-2.1,1.5){\small $0$};
		\node at (-3.3,1.2){\small $0$};
		\node at (-2.6,-0.4){\small $0$};
		\node at (-5.0,-1.60){\small $0$};
		
		\node at (1.1,2.05){\small $C_2$};
		
		\node at (5.1,2.05){\small $C_4$};
		
		\node at (3.31,-0.55){\small $C_1$};
		\node at (3.31,0.55){\small $C_3$};
		
		\end{tikzpicture}
	\end{center}
	\caption{The infinite volume Coxeter tetrahedron  $T$ (left) in  ${\bf H}^{3}_{\mathbb R}$ and the ideal boundary its defining  hyperplane configuration (right).}
	\label{figure:tetrahedron}
\end{figure}
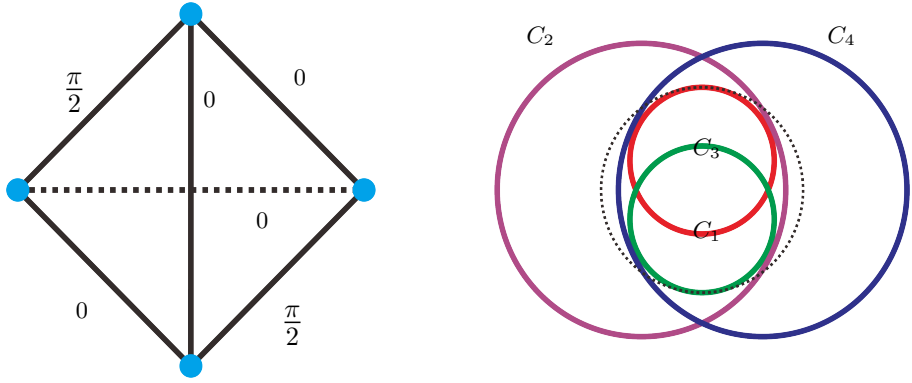

So let $T$ be the Coxeter tetrahedron in  ${\bf H}^{3}_{\mathbb R}$,  which is determined by four totally geodesic hyperbolic planes  in  ${\bf H}^{3}_{\mathbb R}$, such that the ideal boundaries of these  hyperbolic planes is the given   configuration of round circles  in right subfigure  of Figure \ref{figure:tetrahedron}. Then 
$T$ is an infinite volume Coxeter tetrahedron with two opposite angles $\frac{\pi}{2}$ and  the other angles are all zeros, see the left subfigure of Figure 	\ref{figure:tetrahedron}.

 Let $G$ be the Coxeter group of  $T$, that is, the reflection group with four generators the (real) reflections about the defining hyperbolic planes of $T$. So 
$$G=\left\langle \iota_1, \iota_2, \iota_3, \iota_4 \Bigg| \begin{array}  {c}   \iota_1^2=  \iota_2^2 = \iota_3^2=\iota_4^2=id,\\ [3 pt]
(\iota_1 \iota_3)^{2}=(\iota_2 \iota_4)^{2}=id
\end{array}\right\rangle$$
as an  abstract  group.
The group $G$ is  isomorphic to $(\mathbb{Z}_2 \oplus\mathbb{Z}_2) \ast(\mathbb{Z}_2 \oplus\mathbb{Z}_2)$ abstractly.  For technical reasons, we also need to consider a  subgroup of $G$, say $$K=\langle \iota_{1}\iota_{3},\iota_{2}\iota_{4},\iota_{1}\iota_{2} \rangle.$$ $K$ is an index  two subgroup of $G$, and $K$ is isomorphic to $\mathbb{Z}_2 \ast \mathbb{Z}_2 \ast \mathbb{Z}$.


In this paper we   study representations $\rho: G \rightarrow \mathbf{PU}(3,1)$,  such that  $\rho( \iota_{i})=I_{i}$ is a complex reflection fixing a  complex hyperbolic plane   in ${\bf H}^{3}_{\mathbb C}$   for $1 \leq i \leq 4$, with the condition   $I_{i}I_{i+1}$ is parabolic for $1 \leq i \leq 4$  mod $4$. This is a   natural complexification of $T$, since we replace the real reflections by complex reflections. The
moduli space $\mathcal{M}$ is parameterized by $\theta \in  [\frac{5 \pi}{6}, \pi]$. In particular, $\theta=\frac{5 \pi}{6}$  and $\theta=\pi$ degenerate to  ${\bf H}^{2}_{\mathbb C}$-geometry and  ${\bf H}^{3}_{\mathbb R}$-geometry   respectively. By this we mean the group $\rho_{\frac{5 \pi}{6}}(G)$ preserves  a totally geodesic ${\bf H}^{2}_{\mathbb C}\hookrightarrow {\bf H}^{3}_{\mathbb C}$  invariant, and $\rho_{\pi}(G)$  preserves   a totally geodesic   ${\bf H}^{3}_{\mathbb R} \hookrightarrow {\bf H}^{3}_{\mathbb C}$ invariant respectively. 
See Section \ref{sec:moduli} for more details.

Using the 
Dirichlet domains of the  $\rho_{\frac{5 \pi}{6}}(G)$-action on ${\bf H}^{2}_{\mathbb C}$ and the  $\rho_{\pi}(G)$-action on  ${\bf H}^{3}_{\mathbb R}$ as guides, the   main result of this paper  is

\begin{thm}\label{thm:complex3dim} $\rho_{\theta}$ is a discrete and faithful representation of the group $G$ into  $\mathbf{PU}(3,1)$ for any $\theta \in [\frac{5 \pi}{6}, \pi]$.
\end{thm}



In fact there is a hidden and lucky $\mathbb{Z}_4$-symmetry of each representation $\rho_{\theta}$. More precisely, there is $J \in  \mathbf{PU}(3,1)$ an order-4 regular elliptic element, such that $I_{i}=JI_{i-1}J^{-1}$  for $i=1,2,3,4$ mod 4.
We denote $A_{i}=I_{i}I_{i+1}$ for $1 \leq i \leq 4$.  Then 
$$\rho_{\theta}(K)=\left\langle A_1, A_2, A_3, A_4   
\right\rangle,$$ such that $A_{i}$ is parabolic and $A_{i}A_{i+1}$ has order two for $1 \leq i \leq 4$.
We consider the Dirichlet domain of $\rho_{\theta}(K)$ with the fixed point of $J$ as the  center. 
 The reason that we work on $\rho_{\theta}(K)$ instead of $\rho_{\theta}(G)$ is that there are subgroups of $G$ which are infinite dihedral  groups.
 Moreover	by Goldman-Parker \cite{Goldmanparker:1992},  the Dirichlet domain of an infinite dihedral group  tends to have infinite facets (depending on the chosen center of the Dirichlet domain). So it seems the Dirichlet domain of $\rho_{\theta}(G)$ is combinatorically much more difficult to study than that of $\rho_{\theta}(K)$ for general $\theta$.

Let
$$R=\{A^{\pm 1}_1,~~A^{\pm 1}_2,~~A^{\pm 1}_3,~~A^{\pm 1}_4, ~~A_1A_2, ~~A_2A_3\}$$
be a subset of $\rho_{\theta}(K)$ consisting of ten elements. We will show the partial Dirichlet  domain 
$D_{R}$ is in fact the Dirichlet domain  of $\rho_{\theta}(K)$, then we obtain Theorem \ref{thm:complex3dim}.




Since the proof of Theorem \ref{thm:complex3dim}  is much  involved, we prove it in following steps:
\begin{itemize}
	\item We consider firstly the case $\theta = \frac{5 \pi}{6}$, that is, we consider ${\bf H}^{2}_{\mathbb C}$-geometry first.
 The proof when 
 $\theta = \frac{5 \pi}{6}$ can also be viewed as a model of the proof of Theorem \ref{thm:complex3dim} for general $\theta$. We remark that for technical reasons  the proof for $\theta \in (\frac{5 \pi}{6}, \pi]$ in Section \ref{sec:complex3dim} does not hold for $\theta =\frac{5 \pi}{6}$, so we must consider $\theta = \frac{5 \pi}{6}$ separately. Moreover, even through the discreteness  and faithfulness of  $\rho_{\frac{5 \pi}{6}}(K)< \mathbf{PU}(2,1)$ are much easier than general $\rho_{\theta}(K)< \mathbf{PU}(3,1)$, but it is also highly nontrivial;
 
 \item  We then consider the case $\theta = \pi$, that is, we consider ${\bf H}^{3}_{\mathbb R}$-geometry. We remind the reader that in   ${\bf H}^{2}_{\mathbb C}$-geometry and   ${\bf H}^{3}_{\mathbb R}$-geometry of our groups, even through
 the combinatorics of Dirichlet domains for  $\rho_{\frac{5  \pi}{6}}(K)$ and $\rho_{ \pi}(K)$ are different (and they must be different since one has dimension four and the other has dimension three), but they are similar. Moreover,  the words which have contributions to   Dirichlet domains both are the set   $R$.  We also have that the intersection patterns of  Dirichlet domains of  $\rho_{\frac{5  \pi}{6}}(K)$ and $\rho_{ \pi}(K)$  are the same. This can be seen by comparing 	Figure \ref{figure:22infty2dimdirichletabstract} and Figure 	\ref{figure:22inftyrealdirichletabstract} in Sections \ref{sec:Dirich2dim}  and \ref{sec:Dirich3dimreal} respectively. This fact is also very lucky. The part about ${\bf H}^{3}_{\mathbb R}$-geometry  (Section \ref{sec:Dirich3dimreal}) can be omitted 
 logically by two reasons: it is simple, and    the proof in Section  \ref{sec:complex3dim} also covers  ${\bf H}^{3}_{\mathbb R}$-geometry of $\rho_{\pi}(K)$. But the reader are
 encouraged to read this part before Section  \ref{sec:complex3dim};
 
 \item  From above, it is very reasonable to guess that for general $\theta \in (\frac{5 \pi}{6}, \pi]$, the Dirichlet domain of $\rho_{\theta}(K)$ is also  given by the set   $R$. This is what we do in Section  \ref{sec:complex3dim}, that is, we prove ${\bf H}^{3}_{\mathbb C}$-geometry of  Theorem \ref{thm:complex3dim} for  general $\theta \in (\frac{5 \pi}{6}, \pi]$.

 \end{itemize}

 


In the procedure of the proof of  Theorem \ref{thm:complex3dim}, we also find a method that   parameterizing   the intersection of three co-equidistant bisectors in  ${\bf H}^{3}_{\mathbb C}$. Moreover, for a group in  Theorem \ref{thm:complex3dim}, by the $\mathbb{Z}_4$-symmetry, there are two isometric types of 5-facets, say $s_{12}$ and $s_{13}$; One isometric type of 4-facets, say $s_{12}\cap s_{14}$;   One isometric type of 3-facets, say $s_{12}\cap s_{13}\cap s_{14}$. In particular, there is no quadruple intersection of bisectors.   So we do not need the precisely combinatorial structure of these facets. What we really need is the above  mentioned facets are all non-empty. Which is enough for the  Poincar\'e polyhedron theorem in our (lucky) case. See Section \ref{sec:complex3dim} for more details.

But for  more general subgroups of  $\mathbf{PU}(3,1)$ in future study, quadruple intersections of bisectors are  unavoidably. We pose the following question, by which the author has difficult to show. But we believe it is  fundamental in  ${\bf H}^{3}_{\mathbb C}$-geometry. See Subsection \ref{subsection:TripleHermitiancrossproduct} for the notations and background.

\begin{question} \label{question:3-ball} Assume any lifts of  four points $q_0$, $q_{1}$, $q_{2}$ and $q_{3}$ in ${\bf H}^{3}_{\mathbb C}$   are linearly independent  vectors in $\mathbb{C}^{3,1}$, and the triple intersection $$B(q_0,q_1) \cap B(q_0,q_2) \cap B(q_0,q_3)$$ of  three bisectors  $B(q_0,q_i)$  in ${\bf H}^{3}_{\mathbb C}$ for $i=1,2,3$ is nonempty, then  the triple intersection is a 3-ball. 	
\end{question}

See Figure 	\ref{figure:B12B13B143dim}  for an example of  the boundary of the triple intersection of bisectors, which seems to be 2-sphere.

To our knowledge, Theorem \ref{thm:complex3dim}  is the first moduli space in   ${\bf H}^{3}_{\mathbb C}$-geometry that has been studied completely. There are in fact infinitely many Coxeter polytopes in  ${\bf H}^{3}_{\mathbb R}$. In this paper we only complexify the simplest one.  Specially, finite volume Coxeter polytopes in  ${\bf H}^{3}_{\mathbb R}$ merit further complexifying/deforming in  ${\bf H}^{3}_{\mathbb C}$-geometry to find ${\bf H}^{3}_{\mathbb C}$-lattices.  We hope this paper  may  attract more interest  on  this promising  direction.


 {\bf The paper is organized as follows.} In Section \ref{sec:background}  we give well known background
 material on complex hyperbolic geometry. In Section \ref{sec:gram}, we give the matrix representations of $G$ into $\mathbf{PU}(3,1)$ with complex reflection generators.
  Section \ref{sec:Dirich2dim} is devoted to the  Dirichlet domain of  $\rho_{\frac{5 \pi}{6}}(K) <\mathbf{PU}(2,1)$ acting on  ${\bf H}^{2}_{\mathbb C}$. The   Dirichlet domain of  $\rho_{\pi}(K) <\mathbf{PO}(3,1)$ acting on ${\bf H}^{3}_{\mathbb R}$ is studied  in Section \ref{sec:Dirich3dimreal} (but omitting some details). With the warming up in  Sections \ref{sec:Dirich2dim}  and \ref{sec:Dirich3dimreal}, we prove Theorem \ref{thm:complex3dim}  in Section  \ref{sec:complex3dim}.


\textbf{Acknowledgement}: The author  would like to thank his co-author Baohua Xie
\cite{MaX:2021}, the author learned a lot  from Baohua on complex hyperbolic geometry.

 \section{Background}\label{sec:background}

 The purpose of this section is to introduce briefly complex hyperbolic geometry. One can refer to Goldman's book \cite{Go} for more details.


\subsection{Complex hyperbolic space}  \label{subsec:chs}
Let ${\mathbb C}^{n,1}$  denote the vector space ${\mathbb C}^{n+1}$ equipped with the Hermitian
form  of signature $(n,1)$:
 $$\langle {\bf{z}}, {\bf{w}} \rangle=\bf{w}^* \cdot  H \cdot \bf{z},$$ where $ \bf{w}^*$ is the  Hermitian transpose  of  $\bf{w}$,
$$H=\begin{pmatrix}
Id_{n} & 0\\
0 & -1\\
\end{pmatrix}, $$ and $Id_{n}$ is the $n \times n$ identity matrix.
Then
the Hermitian form divides ${\mathbb C}^{n,1}$ into three parts $V_{-}, V_{0}$ and $V_{+}$. Which are
\begin{eqnarray*}
  V_{-} &=& \{{\bf z}\in {\mathbb C}^{n+1}-\{0\} : \langle {\bf z}, {\bf z} \rangle <0 \}, \\
  V_{0} &=& \{{\bf z}\in {\mathbb C}^{n+1}-\{0\} : \langle {\bf z}, {\bf z} \rangle =0 \}, \\
  V_{+} &=& \{{\bf z}\in {\mathbb C}^{n+1}-\{0\} : \langle {\bf z}, {\bf z} \rangle >0 \}.
\end{eqnarray*}

Let $$[~~]: {\mathbb C}^{n+1}-\{0\}\longrightarrow {\mathbb C}{\mathbf P}^{n}$$ be  the canonical projection onto the  complex projective space. 
Then the {\it complex hyperbolic space} ${\bf H}^{n}_{\mathbb C}$ is the image of $V_{-}$ in ${\mathbb C}{\mathbf P}^{n}$
by the  map  $[~~ ]$. The  {\it ideal boundary} of ${\bf H}^{n}_{\mathbb C}$, or {\it  boundary at infinity}, is  the image of $V_{0}$ in
 ${\mathbb C}{\mathbf P}^{n}$, we denote it by $\partial {\bf H}^{n}_{\mathbb C}$.  In this paper, we will denote by   $$\mathbf{q}=(z_1,z_2, \cdots, z_{n+1})^t$$ a vector in ${\mathbb C}^{n,1}$ (note that we use the boldface $\mathbf{q}$), and 
by $$q=[z_1,z_2, \cdots, z_{n+1}]^t$$ the corresponding point in  ${\mathbb C}{\mathbf P}^{n}$. Here the 
superscript  $``t"$  means the transpose of a vector. 
Consider  the natural biholomorphic embedding of ${\mathbb C}^{n}$ onto the affine patch of ${\mathbb C}{\mathbf P}^{n}$, that is 
\begin{equation}\label{projection}
\left(\begin{matrix} z_1 \\ z_2\\ \vdots \\ z_{n} \end{matrix}\right)
\longmapsto \left[\begin{matrix} z_1 \\
z_2 \\ \vdots \\ z_n \\ 1  \end{matrix}\right].
\end{equation}
Then ${\bf H}^{n}_{\mathbb C}$  is identified with the unit ball $\mathbb{B}^n$ in ${\mathbb C}^{n}$ by this embedding, and 
 $\partial{\bf H}^{n}_{\mathbb C}$ is identified with the unit sphere $\mathbb{S}^{2n-1}=\partial \mathbb{B}^n$  in ${\mathbb C}^{n}$.

There is a typical anti-holomorphic isometry $\iota$ of ${\bf H}^{n}_{\mathbb C}$. $\iota$ is given on the level of homogeneous coordinates by complex conjugate

\begin{equation}\label{antiholo}
\iota:\left[\begin{matrix} z_1 \\ z_2\\ \vdots \\ z_{n+1} \end{matrix}\right]
\longmapsto \left[\begin{matrix} \overline{z_1} \\
\overline{z_2} \\ \vdots \\\overline{z_{n+1}} \end{matrix}\right].
\end{equation}


\subsection{Totally geodesic submanifolds and complex reflections}
There are two kinds of totally geodesic submanifolds  in ${\bf H}^{n}_{\mathbb C}$:

\begin{itemize}

\item Given any point $x \in {\bf H}^{n}_{\mathbb C}$, and a complex linear subspace $F$ of dimension $k$ in the tangent space $T_{x}{\bf H}^{n}_{\mathbb C}$, there is a unique complete holomorphic totally geodesic     submanifold contains $x$ and is tangent  to $F$. Such a holomorphic submanifold is called a \emph{$\mathbb{C}^{k}$-plane}.  A $\mathbb{C}^{k}$-plane is the intersection of a complex $k$-dimensional projective subspace in  $\mathbb{C}{\bf P}^{n}$ with ${\bf H}^{n}_{\mathbb C}$, and it is holomorphic  isometric  to ${\bf H}^{k}_{\mathbb C}$.  A $\mathbb{C}^{1}$-plane is also called a \emph{complex geodesic}. The intersection of a $\mathbb{C}^{k}$-plane with $\partial {\bf H}^{n}_{\mathbb C}=\mathbb{S}^{2n-1}$ is a smoothly embedded sphere $\mathbb{S}^{2k-1}$, which is called a  \emph{$\mathbb{C}^{k}$-chain}.

\item  Corresponding to the compatible real structures on $\mathbb{C}^{n,1}$ are the real forms of ${\bf H}^{n}_{\mathbb C}$. That is, the maximal totally real totally geodesic sub-spaces of ${\bf H}^{n}_{\mathbb C}$, which have real dimension $n$.  A maximal totally real totally geodesic subspace of ${\bf H}^{n}_{\mathbb C}$  is the  fixed-point set  of an  anti-holomorphic isometry of ${\bf H}^{n}_{\mathbb C}$. We have gave an example of  anti-holomorphic isometry $\iota$     in (\ref{antiholo}) of Subsection  \ref{subsec:chs}.
For the usual real structure, this submanifold is the real hyperbolic $n$-space ${\bf H}^{n}_{\mathbb R}$ with curvature $-\frac{1}{4}$. Any  totally geodesic subspace of a maximal totally real totally geodesic subspace is a  totally real totally geodesic subspace,  which  is isometric to the real hyperbolic $k$-space ${\bf H}^{k}_{\mathbb R}$ for some $k$.

\end{itemize}

Since the Riemannian sectional curvatures  of the  complex hyperbolic space
are non-constant, there are no totally geodesic hyperplanes in  ${\bf H}^{n}_{\mathbb C}$ when $n \geq 2$.


Let $L$ be a $(n-1)$-dimensional complex plane in ${\bf H}^{n}_{\mathbb C}$,  a {\it polar vector} of $L$ is the unique vector (up to scaling) in ${\mathbb C}^{n,1}$ perpendicular to this complex plane with
respect to the Hermitian form. A polar vector of a $(n-1)$-dimensional complex plane belongs to  $V_{+}$ and each vector in $V_{+}$ corresponds to a $(n-1)$-dimensional complex plane.
Moreover, let $L$ be a $(n-1)$-dimensional complex plane with polar vector
${\bf n}\in V_{+}$,
then the {\it complex reflection} fixing $L$ with rotation angle $\theta$  is given by
\begin{equation} \label{equation:reflection}
 I_{\bf n, \theta}({\bf z}) = -{\bf z}+(1-\mathrm{e}^{\theta   \mathrm{i}})\frac{\langle {\bf z}, {\bf n} \rangle}{\langle {\bf n},{\bf n}\rangle}{\bf n}.
\end{equation} 

The complex plane $L$ is also called the \emph{mirror} of $I_{\bf n, \theta}$. In this paper, we only consider the case  $\theta= \pi$, then the complex reflection has order 2.


 \subsection{The isometries} The complex hyperbolic space is a K\"{a}hler manifold of constant holomorphic sectional curvature $-1$.
 We denote by $\mathbf{U}(n,1)$ the Lie group of the Hermitian form  $\langle \cdot,\cdot\rangle$ preserving complex linear
 transformations and by $\mathbf{PU}(n,1)$ the group modulo scalar matrices. The group of holomorphic
 isometries of ${\bf H}^{n}_{\mathbb C}$ is exactly $\mathbf{PU}(n,1)$. It is sometimes convenient to work with
 $\mathbf{SU}(n,1)$.
 The full isometry group of ${\bf H}^{n}_{\mathbb C}$ is
 $$\widehat{\mathbf{PU}(n,1)}=\langle \mathbf{PU}(n,1),\iota\rangle,$$
 where $\iota$ is the anti-holomorphic isometry in Subsection   \ref{subsec:chs}.

 
 Elements of $\mathbf{SU}(n,1)$ fall into three types, according to the number and types of  fixed points of the corresponding
 isometry. Namely,
 \begin{itemize}
 	\item
  an isometry is {\it loxodromic} if it has exactly two fixed points 
 on $\partial {\bf H}^{n}_{\mathbb C}$;
 	\item  an isometry is  {\it parabolic} if it has exactly one fixed point
 	on $\partial {\bf H}^{n}_{\mathbb C}$;
	\item   an isometry is {\it elliptic}  when it has (at least) one fixed point inside ${\bf H}^{n}_{\mathbb C}$. 
 \end{itemize}

An element $A\in \mathbf{SU}(n,1)$ is called {\it regular} whenever it has  distinct eigenvalues,
 an elliptic $A\in \mathbf{SU}(n,1)$ is called  {\it special elliptic} if
 it has a repeated eigenvalue. Complex reflection about a totally geodesic  ${\bf H}^{n-1}_{\mathbb C} \hookrightarrow {\bf H}^{n}_{\mathbb C}$ is an example of  special elliptic element in  $\mathbf{SU}(n,1)$.




\subsection{Bisectors, spinal spheres and  Dirichlet domain}
 Dirichlet domain (or Dirichlet polyhedron) is a fundamental tool to study a discrete subgroup in $\mathbf{SU}(n,1)$, which is defined in terms  of (infinitely many) bisectors.
 

\begin{defn} Given two distinct points $q_0$  and $q_1$ in ${\bf H}^{n}_{\mathbb C}$ with the same norm (e.g. one could
	take lifts $\mathbf{q_0},\mathbf{q_1}$ of them such that $\langle \mathbf{q_0},\mathbf{q_0}\rangle=\langle \mathbf{q_1},\mathbf{q_1}\rangle= -1$), the \emph{bisector} $B(q_0,q_1)$ is the projectivization  of the set of negative   vector $\mathbf{x}$ in ${\mathbb C}^{n, 1}$
	with
	$$|\langle \mathbf{x}, \mathbf{q_0}\rangle|=|\langle  \mathbf{x}, \mathbf{q_1}\rangle|.$$
\end{defn}

The  {\it spinal sphere} of the bisector $B(q_0,q_1)$ is the intersection of $\partial {\bf H}^{n}_{\mathbb C}$ with the closure of $B(q_0,q_1)$ in $\overline{{\bf H}^{n}_{\mathbb C}}= {\bf H}^{n}_{\mathbb C}\cup \partial { {\bf H}^{n}_{\mathbb C}}$. The bisector $B(q_0,q_1)$ is a topological $(2n-1)$-ball, and its spinal sphere is a $(2n-2)$-sphere.




A bisector $B(q_0,q_1)$ separates  ${\bf H}^{n}_{\mathbb C}$ into two half-spaces, one of them containing $q_0$. 
\begin{defn}The \emph{Dirichlet domain} $D_{\Gamma}$ for a discrete group $\Gamma < \mathbf{PU}(n,1)$ centered at $q_{0}$ is
	the intersection of the (closures of the) half-spaces  containing  $q_{0}$ of all bisectors corresponds to  elements in $\Gamma$ not fixing $q_{0}$. That is,
	$$D_{\Gamma}=\{p\in {\bf H}^{n}_{\mathbb C} \cup \partial{\bf H}^{n}_{\mathbb C}: |\langle \mathbf{p},\mathbf{q_0}\rangle|\leq|\langle \mathbf{p},g(\mathbf{q_0})\rangle|,
	\ \forall g \in \Gamma \ \mbox{with} \ g(q_{0})\neq q_{0} \}.$$
\end{defn}

\begin{defn}For a subset $R \subset \Gamma$, the  \emph{partial Dirichlet domain} $D_{R}$ for a discrete group $\Gamma < \mathbf{PU}(n,1)$ centered at $q_{0}$ is
	$$D_{R}=\{p\in {\bf H}^{n}_{\mathbb C} \cup \partial{\bf H}^{n}_{\mathbb C}: |\langle \mathbf{p},\mathbf{q_0}\rangle|\leq|\langle \mathbf{p},g(\mathbf{q_0})\rangle|,
	\ \forall g \in R \ \mbox{with} \ g(q_{0})\neq q_{0} \}.$$
\end{defn}


 From the definition, one can see that parts of bisectors form the boundary of a Dirichlet domain.  For $g \in \Gamma$, when the center $q_0$ is clear  from the context,  we also denote  $B(q_0, g(q_0))$  by $B(g)$, and  we will denote by $s(g)=B(g)  \cap D_{\Gamma}$, which is a $(2n-1)$-facet of $D_{\Gamma}$ in general.
  Facets of codimension one in $D_{\Gamma}$   will also be called {\it sides}.  In general, facets of codimension two in $D_{\Gamma}$   will be called {\it ridges}.
Facets of dimension one and zero in $D_{\Gamma}$ will be called  {\it edges} and {\it vertices} respectively. Moreover, a  {\it bounded ridge} is a ridge which does not intersect  $\partial {\bf H}^{n}_{\mathbb C}$, and if the intersection of a ridge $r$ and $\partial {\bf H}^{n}_{\mathbb C}$ is non-empty, then
$r$ is an {\it infinite ridge}.

It is usually very hard to determine $D_{\Gamma}$ because one should check infinitely many inequalities.
Therefore a general method will be to guess that a partial Dirichlet domain $D_{R}$ is in fact the  Dirichlet domain $D_{\Gamma}$,  and then  check  it using the Poincar\'e polyhedron theorem.
The basic idea is that the sides of $D_{\Gamma}$ should be paired by isometries, and the images of $D_{\Gamma}$
under these so-called side-pairing maps should give a local tiling of ${\bf H}^{n}_{\mathbb C}$.  If they do (and if
the quotient of $D_{\Gamma}$ by the identification given by the side-pairing maps is complete), then the Poincar\'{e} polyhedron
theorem implies that the images of $D_{\Gamma}$ actually give a global tiling of ${\bf H}^{n}_{\mathbb C}$.  So  $D_{\Gamma}$ is a fundamental domain of $\Gamma$-action on ${\bf H}^{n}_{\mathbb C}$.

Once a fundamental domain is obtained, one gets an explicit presentation of $\Gamma$ in terms of the generators given by the side-pairing maps together
with a generating set for the stabilizer $p_{0}$. Where the relations correspond to so-called ridge cycles, which correspond to the local tilings bear each co-dimension two facet. For more on the Poincar\'e polyhedron theorem, see \cite{dpp:2016, ParkerWill:2017}.

\subsection{Hermitian cross product in  ${\bf H}^2_{\mathbb C}$ and double intersection of  bisectors}\label{subsection:Spinalcoordinates}
The intersections of bisectors in ${\bf H}^2_{\mathbb C}$ are a little easier to describe than in higher dimensions. We show this in this subsection.


For complex hyperbolic plane  ${\bf H}^2_{\mathbb C}$ with Hermitian form given by
$$H=\begin{pmatrix}
Id_{2} & 0\\
0 & -1\\
\end{pmatrix}. $$
If
$$p=\left[\begin{matrix}
p_1\\ p_2\\p_3\end{matrix}\right],\quad q=\left[\begin{matrix}
q_1\\ q_2\\q_3\end{matrix}\right]$$ are two points in $\mathbf{H}^2_{\mathbb C}$, then the {\it Hermitian cross product} of $p$ and $q$  is a point in $ {\mathbb C}{\mathbf P}^{2}$ defined by
$$p \boxtimes
q=
\left[\begin{array}{c}
\overline{p}_3\overline{q}_2-\overline{p}_2\overline{q}_3\\  [2 pt] \overline{p}_1\overline{q}_3-\overline{p}_3\overline{q}_1\\ [2 pt] \overline{p}_1\overline{q}_2-\overline{p}_2\overline{q}_1
\end{array}
\right],$$
see Page 45 of \cite{Go}. Any lift of   $p \boxtimes
q$ is orthogonal to lifts of both $p$ and $q$  with  respect to the Hermitian form  $\langle \cdot,\cdot\rangle$.
It is a Hermitian version of the Euclidean cross product in  ${\mathbb R}^3$.

In order to analyze  2-faces of a Dirichlet polyhedron in ${\bf H}^2_{\mathbb C}$, we must study the intersections of bisectors.
From the detailed analysis in \cite{Go}, we know that the intersection of two bisectors is usually not totally geodesic and it can be somewhat complicated. In  this paper, we shall only consider the intersections of
\emph{co-equidistant bisectors}, i.e. bisectors equidistant from a common point.   When $p,q$ and $r$ are not in a common complex line, that is,  their lifts are linearly independent in $\mathbb {C}^{2,1}$, then the locus $$B(p,q,r)=B(p,q) \cap B(p,r)$$ of points in $ {\bf H}^2_{\mathbb C}$
equidistant to  $p$, $q$ and $r$ is a smooth disk that is not totally geodesic, and is often called a \emph{Giraud disk}.  The following property is crucial when studying fundamental domain.

\begin{prop}[Giraud] \label{prop:Giraud}
	If $p$, $q$ and $r$ in ${\bf H}^2_{\mathbb C}$ are not in a common complex line, then the Giraud disk $B(p,q,r)$ is contained in precisely three bisectors, namely $B(p,q)$, $B(q,r)$ and  $B(p,r)$.
\end{prop}

Note that checking whether an isometry maps  a Giraud disk to another is equivalent to checking that corresponding triples of  points are mapped to each other.

In order to study Giraud  disks, we will use {\it spinal coordinates}. The  {\it complex spine} of the bisector $B(p,q)$ is the complex line through the two points $p$ and $q$. The {\it real spine} of $B(p,q)$
is the intersection of the complex spine with the bisector itself, which is a (real) geodesic. The real spine is the locus of points inside the complex spine which are equidistant from $p$ and $q$.
Bisectors are not totally geodesic, but they have a very nice foliation by two different families of totally geodesic submanifolds. Mostow \cite{Mostow:1980}  showed that a bisector is the preimage of the real
spine under the orthogonal projection onto the complex spine. The fibres of this projection are complex lines ${\bf H}^{1}_{\mathbb C} \hookrightarrow {\bf H}^{2}_{\mathbb C}$ called the {\it complex slices} of the bisector. Goldman \cite{Go} showed that a bisector is
the union of all  totally real totally geodesic planes containing the real spine. Such Lagrangian planes are called the {\it real slices} or {\it meridians} of the bisector.
 The complex slices of $B(p,q)$ are given explicitly by choosing a lift $\mathbf{p}$ (resp. $\mathbf{q}$) of $p$ (resp. $q$).
When $p,q\in  {\bf H}^2_{\mathbb C}$, we simply choose lifts such that $\langle \mathbf{p},\mathbf{p}\rangle= \langle \mathbf{q},\mathbf{q}\rangle$.
The complex slices of $B(p,q)$ are obtained as the set of negative lines $(\overline{z}\mathbf{p}-\mathbf{q})^{\bot}$ in ${\bf H}^2_{\mathbb C}$ for some arc of values of $z\in \mathbb{S}^1$,  which is determined by requiring that $\langle \overline{z}\mathbf{p}-\mathbf{q},\overline{z}\mathbf{p}-\mathbf{q}\rangle>0$.


Since a point of the bisector is on precisely one complex slice, we can parameterize the {\it Giraud torus}  $\hat{B} (p,q,r)$  in ${\bf P}^2_{\mathbb C}$   by $(z_1,z_2)=(e^{it_1},e^{it_2})\in \mathbb{S}^1\times \mathbb{S}^1$ via
	\begin{equation} \label{equaation:girauddisk2dim}
	V(z_1,z_2)=(\overline{z}_1\mathbf{p}-\mathbf{q})\boxtimes (\overline{z}_2\mathbf{p}-\mathbf{r})=\mathbf{q}\boxtimes \mathbf{r}+z_1 \mathbf{r}\boxtimes \mathbf{p}+z_2 \mathbf{p}\boxtimes \mathbf{q}.
	\end{equation}
The Giraud disk $B(p,q,r)$ corresponds to $(z_1,z_2)\in \mathbb{S}^1\times \mathbb{S}^1 $  with  $$\langle V(z_1,z_2),V(z_1,z_2)\rangle<0.$$ 
It is well known that  this region is a topological disk if it is non empty \cite{Go}.
The boundary at infinity $\partial B(p,q,r)$ is a circle, given in spinal coordinates by the equation
$$\langle V(z_1,z_2),V(z_1,z_2)\rangle=0.$$
Note that the choices of different lifts of  $p$, $q$ and $r$ affect the spinal coordinates by rotation on each of the $\mathbb{S}^1$-factors.


A defining equation for the trace of another bisector $B(u,v)$ on the Giraud disk $B(p,q,r)$ can be written in the form
$$| \langle V(z_1,z_2),\mathbf{u}\rangle|=| \langle V(z_1,z_2),\mathbf{v}\rangle|,$$  provided that $\mathbf{u}$ and $\mathbf{v}$ are suitably chosen lifts.
The expressions $\langle V(z_1,z_2),\mathbf{u}\rangle$ and $\langle V(z_1,z_2), \mathbf{v}\rangle$ are affine in $z_1$ and $z_2$.

This triple bisector intersection can be parameterized  fairly explicitly, because one can solve the equation
$$|\langle V(z_1,z_2),\mathbf{u}\rangle|^2=|\langle V(z_1,z_2),\mathbf{v}\rangle|^2 $$ for one of the variables $z_1$ or $z_2$ simply by solving a quadratic equation.
A detailed explanation of how this works can be found in \cite{Deraux:2016gt, DerauxF:2015, dpp:2016}.

\subsection{Triple Hermitian cross product in  ${\bf H}^3_{\mathbb C}$ and  triple intersection of bisectors}\label{subsection:TripleHermitiancrossproduct}

Now we consider   triple intersections of bisectors  in  ${\bf H}^3_{\mathbb C}$.

Let $q_0$, $q_1$, $q_2$ and $q_3$ be four points in   $\mathbf{H}^3_{\mathbb C}$. We take lifts  $\mathbf{q_i}$  of $q_i$ such that $$\langle \mathbf{q_0},\mathbf{q_0} \rangle=\langle \mathbf{q_1},\mathbf{q_1} \rangle=\langle \mathbf{q_2},\mathbf{q_2} \rangle=\langle \mathbf{q_3},\mathbf{q_3} \rangle. $$
We also  assume $\{\mathbf{q_0},\mathbf{q_1}, \mathbf{q_2}, \mathbf{q_3} \}$  are linearly independent as vectors in  $\mathbb{C}^{3,1}$.
We have three bisectors $B(q_0,q_1)$,  $B(q_0,q_2)$ and $B(q_0,q_3)$ in  ${\bf H}^3_{\mathbb C}$. We now   parameterize the triple intersection $$B(q_0,q_1) \cap B(q_0,q_2) \cap B(q_0,q_3).$$
The bisector $B(q_0,q_1)$ has a decomposition into a set of complex slices (each of them is a totally geodesic  $\mathbf{H}^2_{\mathbb C} \hookrightarrow \mathbf{H}^3_{\mathbb C}$), these complex slices are obtained as  the set of negative lines in  $(w \mathbf{q_0}-\mathbf{q_1})^{\perp}$ 
for some arc of values $w \in \mathbb{S}^1$,   see \cite{Go,Deraux:2016gt}.
Similarly, $B(q_0,q_2)$ and  $B(q_0,q_3)$ also have  decompositions into a set of complex slices  which are    parameterized  by  $(w \mathbf{q_0}- \mathbf{q_2})^{\perp}$ and  $(w \mathbf{q_0}-\mathbf{q_3})^{\perp}$ for some arc of values $w \in \mathbb{S}^1$. 

Consider triple Hermitian cross product with respect to the Hermitian form $H$. Recall that for three linearly independent vectors 
 \begin{equation}\label{abc}
a=\left(\begin{matrix} a_1, a_2, a_3, a_4 \end{matrix}\right)^{t}
,~~b= \left(\begin{matrix} b_1, b_2, b_3,b_4 \end{matrix}\right)^{t},~~c= \left(\begin{matrix} c_1, c_2, c_3,c_4 \end{matrix}\right)^{t}
\end{equation}
in $\mathbb{R}^4$,  the \emph{generalized cross product $a \times b \times c$} of $a,b,c$ is a vector 
 \begin{equation} \nonumber d=\left(\begin{matrix} d_1, d_2, d_3, d_4 \end{matrix}\right)^{t}.
 \end{equation}
 Where $d_i$ is the coefficient of $e_i$ in  the determinant of the matrix 
\begin{equation} \nonumber \begin{pmatrix}
a_1&b_1& c_1&e_1\\
a_2& b_2 &c_2& e_2\\
a_3&b_3&c_3 &e_3\\
a_4 & b_4&c_4&e_4\\
\end{pmatrix}.\end{equation}
The vector $d=a \times b \times c$  is perpendicular to $a$, $b$ and $c$ with respect to the  standard Euclidean metric in  $\mathbb{R}^4$.
Now for  three linearly independent vectors $a$, $b$, $c$ as in (\ref{abc}) in $\mathbb{C}^{3,1}$ with $a_i,b_i,c_i \in \mathbb{C}$,
 the \emph{triple  Hermitian cross product $a \boxtimes b \boxtimes c$} of $a$, $b$, $c$ is a vector  \begin{equation} \nonumber
 d=\left(\begin{matrix} d_1, d_2, d_3, d_4 \end{matrix}\right)^{t}.
\end{equation}
Where by definition 
\begin{equation} \nonumber d=\begin{pmatrix}
-1&0& 0&0\\
0& -1&0& 0\\
0&0&-1 &0\\
0& 0&0&1\\
\end{pmatrix} \cdot (\bar{a}\times \bar{b} \times \bar{c}),
\end{equation}
here for example  \begin{equation}\nonumber
\bar{a}=\left(\begin{matrix} \overline{a_1}, \overline{a_2}, \overline{a_3}, \overline{a_4} \end{matrix}\right)^{t}
\end{equation} with $ \overline{a_i}$ the complex conjugate of $a_i$.
 Then  $d=a \boxtimes b \boxtimes c$ is the vector 
 \begin{equation}\nonumber
 \left(\begin{matrix} \det\begin{pmatrix}
 \overline{a_2}&\overline{b_2}& \overline{c_2}\\
\overline{a_3}&\overline{b_3}& \overline{c_3}\\
\overline{a_4}&\overline{b_4}& \overline{c_4} \\
 \end{pmatrix}, -\det\begin{pmatrix}
 \overline{a_1}&\overline{b_1}& \overline{c_1}\\
 \overline{a_3}&\overline{b_3}& \overline{c_3}\\
 \overline{a_4}&\overline{b_4}& \overline{c_4} \\
 \end{pmatrix}, \det\begin{pmatrix}
 \overline{a_1}&\overline{b_1}& \overline{c_1}\\
 \overline{a_2}&\overline{b_2}& \overline{c_2}\\
 \overline{a_4}&\overline{b_4}& \overline{c_4} \\
 \end{pmatrix}, \det\begin{pmatrix}
 \overline{a_1}&\overline{b_1}& \overline{c_1}\\
 \overline{a_2}&\overline{b_2}& \overline{c_2}\\
 \overline{a_3}&\overline{b_3}& \overline{c_3} \\
 \end{pmatrix} \end{matrix}\right)^{t}.
 \end{equation}
 By direct calculation, we have $$\langle a, a \boxtimes b \boxtimes c\rangle=\langle b, a \boxtimes b \boxtimes c\rangle=\langle c, a \boxtimes b \boxtimes c\rangle=0 $$ with respect to the Hermitioan form $H$ on $\mathbb{C}^{3,1}$.

 Since a point in a bisector $B$  lies in precisely one complex slice of $B$, we can  parameterize  $B(q_0,q_1) \cap B(q_0,q_2) \cap B(q_0,q_3)$  by $$V(w_1,w_2,w_3)=(\overline{w_1}\mathbf{q_0}-\mathbf{q_1}) \boxtimes(\overline{w_2}\mathbf{q_0}-\mathbf{q_2)} \boxtimes(\overline{w_3} \mathbf{q_0}-\mathbf{q_3})$$
 with $(w_1,w_2,w_3) \in \mathbb{S}^1 \times\mathbb{S}^1 \times\mathbb{S}^1$.
 Up to sign and rewriting, we can   parameterize  $B(q_0,q_1) \cap B(q_0,q_2) \cap B(q_0,q_3)$  by 
 \begin{equation}\label{equation:tripleintersection} V(z_1,z_2,z_3)=\mathbf{q_1} \boxtimes \mathbf{q_2} \boxtimes \mathbf{q_3}+z_1 \cdot \mathbf{q_0} \boxtimes \mathbf{q_2} \boxtimes \mathbf{q_3}+z_2 \cdot \mathbf{q_0} \boxtimes \mathbf{q_1} \boxtimes \mathbf{q_3}+z_3 \cdot \mathbf{q_0} \boxtimes \mathbf{q_1} \boxtimes \mathbf{q_2}\end{equation}
 with $(z_1,z_2,z_3) \in \mathbb{S}^1 \times\mathbb{S}^1 \times\mathbb{S}^1$ such that  $\langle V(z_1,z_2,z_3), V(z_1,z_2,z_3)\rangle$ is negative.

 We remark that it is reasonable to guess that for fixed $\{q_0, q_1, q_2, q_3\}$ in $\mathbf{H}^3_{\mathbb C}$, if there is  $(z_1,z_2,z_3) \in \mathbb{S}^1 \times\mathbb{S}^1 \times\mathbb{S}^1$  such that $V=V(z_1,z_2,z_3)$ is negative in (\ref{equation:tripleintersection}), then    all $(z_1,z_2,z_3) \in \mathbb{S}^1 \times\mathbb{S}^1 \times\mathbb{S}^1$ satisfying this condition should be a 3-ball in $\mathbb{S}^1 \times\mathbb{S}^1 \times\mathbb{S}^1$. But the author has difficult to show this. See Question \ref{question:3-ball}.

\section{The moduli space of representations of $G$ into $\mathbf{PU}(3,1)$} \label{sec:gram}
\label{sec:moduli}

In this section, we give the matrix representations of $G$ into $\mathbf{PU}(3,1)$ with complex reflection generators and $\iota_{i}\iota_{i+1}$ is mapped to a   parabolic element  for $i=1,2,3,4$ mod $4$.

Let $G$ be the  abstract group with the presentation
$$G=\left\langle \iota_{1}, \iota_{2},\iota_{3},\iota_{4} \Bigg| \begin{array}  {c}  \iota_{1}^2= \iota_{2}^2= \iota_{3}^2=  \iota_{4}^2=id,\\ [ 3 pt]
( \iota_{1} \iota_{3})^{2}=  (\iota_{2} \iota_{4})^{2}=id
\end{array}\right\rangle.$$
$G$ is isomorphic to  $(\mathbb{Z}_2 \oplus\mathbb{Z}_2) \ast(\mathbb{Z}_2 \oplus\mathbb{Z}_2)$ abstractly.   Then $K=\langle \iota_{1}\iota_{3},\iota_{2}\iota_{4},\iota_{1}\iota_{2} \rangle$ is an index  two subgroup of $G$, which is isomorphic to $\mathbb{Z}_2 \ast \mathbb{Z}_2 \ast \mathbb{Z}$. 

\subsection{The Gram matrices of four complex hyperbolic planes in $\mathbf{H}^3_{\mathbb C}$ for the group $G$} \label{subsec:gram}

Recall that for two $\mathbb C$-planes $\mathcal{P}$ and $\mathcal{P}'$  in $\mathbf{H}^3_{\mathbb C}$ with  polar vectors  $n$ and $n'$ such that
$\langle n, n\rangle =\langle n', n'\rangle =1$:
\begin{itemize}
	\item
	If  $\mathcal{P}$ and $\mathcal{P}'$  intersect in a $\mathbb{C}$-line in $\mathbf{H}^3_{\mathbb C}$, then the angle $\alpha $ between
	them has $|\langle n, n'\rangle|=\cos(\alpha)$;

	\item If  $\mathcal{P}$ and $\mathcal{P}'$ are hyper-parallel  in  $\mathbf{H}^3_{\mathbb C}$,
	then the distance $d$ between them has $|\langle n, n'\rangle |=\cosh \frac{d}{2}$;   	
	
	\item If  $\mathcal{P}$ and $\mathcal{P}'$ are asymptotic in  $\mathbf{H}^3_{\mathbb C}$,  then $|\langle n, n' \rangle|=1$.\end{itemize}

We consider  $\mathbb C$-planes $\mathcal{P}_{i}$ for $i=1,2,3,4$, so each $\mathcal{P}_{i}$ is a  totally geodesic ${\bf H}^2_{\mathbb C} \hookrightarrow {\bf H}^3_{\mathbb C}$. Let $n'_{i}$ be the polar vector of  $\mathcal{P}_{i}$ in $\mathbb{C}{\mathbf P}^{3}-\overline{\mathbf{H}^3_{\mathbb C}}$.
We assume
\begin{itemize}
	\item  the angle between $\mathcal{P}_{1}$ and $\mathcal{P}_{3}$ is $\frac{\pi}{2}$;
	\item the angle between $\mathcal{P}_{2}$ and $\mathcal{P}_{4}$ is $\frac{\pi}{2}$;
	
		\item the planes
		 $\mathcal{P}_{i}$ and $\mathcal{P}_{i+1}$  are   asymptotic  for $i=1,2,3,4$  mod $4$.
	\end{itemize}
Then we can normalize the Gram matrix of  $\{n'_{i}\}^4_{i=1}$
into  the following form
	$$\mathcal{G}'=(
	\langle  n'_{i},n'_{j}\rangle)_{1\leq i, j \leq 4}=\begin{pmatrix}
	1& 1& 0&\mathrm{e}^{-4\theta\mathrm{i}}\\
	1& 1 & 1& 0\\
	0&1&1 &1\\
	\mathrm{e}^{4\theta\mathrm{i}} & 0&1&1\\
	\end{pmatrix}.$$
	Where up to anti-holomorphic isometry of  ${\bf H}^3_{\mathbb C}$, we may assume $4 \theta \in [a,a+\pi]$ for any $a \in \mathbb{R}$, that is $\theta \in [\frac{a}{4}, \frac{a+ \pi}{4}]$.
	Moreover $4\theta= 2k \pi$  for $k \in \mathbb{Z}$ corresponds to the case of an infinite volume 3-dimensional real hyperbolic Coxeter tetrahedron, see \cite{VinbergS:1993}.
	By 	\cite{CunhaDGT:2012}, for  a Gram matrix above, there is a unique configuration of four
	 $\mathbb C$-planes $\mathcal{P}_{i}$ in ${\bf H}^3_{\mathbb C}$ for $i=1,2,3,4$  up to $\mathbf{PU}(3,1)$  realizing the Gram matrix.
	
	
	In fact, we can re-normalize $n'_{i}$ above into $$n_1=n'_1,~~~~~n_2=\mathrm{e}^{-\theta\mathrm{i}}  \cdot n'_2,~~~~~n_3=\mathrm{e}^{-2\theta\mathrm{i}} \cdot n'_3,~~~~~n_4=\mathrm{e}^{-3\theta\mathrm{i}} \cdot n'_4,$$ 
	Then we can re-normalize the Gram matrix
	to  the following form
	\begin{equation}\label{matrix:Gram}\mathcal{G}=(
	\langle  n_{i},n_{j}\rangle)_{1\leq i, j \leq 4}=\begin{pmatrix}
	1& \mathrm{e}^{\theta\mathrm{i}}& 0&\mathrm{e}^{-\theta\mathrm{i}}\\
	\mathrm{e}^{-\theta\mathrm{i}}& 1 & \mathrm{e}^{\theta\mathrm{i}}& 0\\
	0&\mathrm{e}^{-\theta\mathrm{i}}&1 &\mathrm{e}^{\theta\mathrm{i}}\\
	\mathrm{e}^{\theta\mathrm{i}} & 0&\mathrm{e}^{-\theta\mathrm{i}}&1\\
	\end{pmatrix}.\end{equation}

	From now on we may assume that $\theta \in [\frac{3 \pi}{4}, \pi]$. It is easy to see $$\det(\mathcal{G})=-1-2 \cos(4 \theta),$$ and the eigenvalues of $\mathcal{G}$ are $$1 \pm 2\cos(\theta), ~~~1 \pm 2\sin(\theta).$$ 
	We have the followings:
	\begin{itemize} 
		
	\item   when $\theta = \pi$,
	all entries of $\mathcal{G}$ are reals, so	it degenerates to ${\bf H}^3_{\mathbb R}$-geometry;
	
			\item  when $\theta =\frac{5 \pi}{6}$, $\mathcal{G}$ has eigenvalues $$0, ~~~2, ~~~1 + \sqrt{3},~~~ 1- \sqrt{3},$$
		so it degenerates to ${\bf H}^2_{\mathbb C}$-geometry;

		\item  when $\theta \in (\frac{5 \pi}{6}, \pi]$, $\mathcal{G}$ has signature $(3,1)$, we have   ${\bf H}^3_{\mathbb C}$-geometry;

	\item when $\theta \in [\frac{3 \pi}{4},\frac{5 \pi}{6})$, $\mathcal{G}$ has signature $(2,2)$. We will not study them in this paper.
	
	\end{itemize}
	So our moduli space is
	\begin{equation}
	\mathscr{M}= \left[\frac{5 \pi}{6}, \pi \right].	\end{equation}
	

From the Gram matrix (\ref{matrix:Gram}), there is a $\mathbb{Z}_4$-symmetry  of the configurations of $\mathbb{C}$-planes above. Take $$J=\begin{pmatrix}
	-1& 0& 0&0\\
	0& \rm{i} & 0& 0\\
	0&0&-\rm{i} &0\\
	0 & 0&0&1\\
	\end{pmatrix}.$$
	Then $J^* HJ=H$, $J$ has order 4 with fixed point 
	\begin{equation}\label{p0}
	p_0=\left[\begin{matrix} 0, 0,0, 1 \end{matrix}\right]^{t}
	\end{equation} in ${\bf H}^3_{\mathbb C}$.
	Note that $\det(J)=-1$, so $J \in \mathbf{U}(3,1)$, but $J \notin \mathbf{PU}(3,1)$, and  $\rm{e}^{\frac{\pi\rm{i}}{4}}\cdot J \in \mathbf{PU}(3,1)$. But we can also use $J$ to study the $\mathbb{Z}_4$-symmetry for simplicity of notations.

	We take
	\begin{equation}\label{vector:n1}
{n_1}=\left[
\begin{array}{c}
\displaystyle{\frac{\sqrt{1-2\cos(\theta)} }{2}}\\ [3 ex]
\displaystyle{\frac{\sqrt{1-2\sin(\theta)} }{2}} \\ [3 ex]
\displaystyle{\frac{\sqrt{1+2\sin(\theta)} }{2}} \\ [3 ex]
\displaystyle{\frac{\sqrt{-1-2\cos(\theta)} }{2}}\\
\end{array}
\right]
\end{equation}
in $ {\mathbb C}{\mathbf P}^{3}- \overline{ {\bf H}^3_{\mathbb C}}$.
When $\theta \in [\frac{5 \pi}{6}, \pi]$, all entries of $n_1$ are non-negative reals. 
Let $n_{i}$  be $J^{i-1}(n_1)$ for $i=2,3,4$.  Then $$(
\langle  n_{i},n_{j}\rangle)_{1\leq i, j \leq 4}$$  is $\mathcal{G}$ in (\ref{matrix:Gram}). By 	\cite{CunhaDGT:2012}, there is a one-to-one correspondence between the Gram matrix and configurations $\mathbb{C}$-planes, so 	there is a $\mathbb{Z}_4$-symmetry of  the configurations $\mathbb{C}$-planes above.

\begin{remark} \label{remark:symmetry}The author notes that the $\mathbb{Z}_4$-symmetry above is lucky since it will simplify the calculations dramatically. The $\mathbb{Z}_4$-symmetry is also a little mysterious  since the author can not explain it from geometric point of view. 
\end{remark}

 	Now for any $\theta \in \mathcal{M}$, let $\rho_{\theta}:G \rightarrow  \mathbf{PU}(3,1)$  be the representation with $\rho_{\theta}(\iota_{i})=I_{i}$  the order two $\mathbb{C}$-reflection about $\mathcal{P}_{i}$.
 	We also denote by $$\Gamma=\Gamma_{\theta} =\langle I_1,I_2,I_3,I_4 \rangle.$$
 	\begin{itemize}
	\item  When $\theta= \frac{5 \pi}{6}$, $\Gamma$ preserves a totally geodesic ${\bf H}^2_{\mathbb C} \hookrightarrow{\bf H}^3_{\mathbb C}$ invariant, so we will view the representation as degenerating to a representation into  $\mathbf{PU}(2,1)$. But the discreteness and faithfulness of this representation are still non-trivial,  see  Section \ref{sec:Dirich2dim}  for more details.
	
\item When  $\theta =\pi$, $\Gamma$ preserves a totally geodesic ${\bf H}^3_{\mathbb R} \hookrightarrow{\bf H}^3_{\mathbb C}$ invariant.	We have a 3-dimensional real hyperbolic Coxeter tetrahedron, so both  the  discreteness and faithfulness of this  representation are trivial. The Dirichlet domain of this  group  representation in  ${\bf H}^3_{\mathbb R}$
 is not difficult, see Section  \ref{sec:Dirich3dimreal} for more details.
 
 \item  When we decrease $\theta$  from $\pi$ to $\frac{5 \pi}{6}$, we still have a representation of $G$ into   $\mathbf{PU}(3,1)$, but the discreteness of the  representation is highly non-trivial. This is what we will do in the paper.
 	\end{itemize}

			\subsection{Matrices  representations of the group $G$ into $\mathbf{PU}(3,1)$ } \label{subsec:matricesin3dim}
	For each  $\theta\in \mathcal{M}$, we now give the matrix presentation of $I_{i}$, the order-two complex reflection with mirror  $\mathcal{P}_{i}$  for $i=1, 2, 3, 4$.


From   (\ref{equation:reflection}) and the vector $n_1$ in (\ref{vector:n1}), it is easy to see  
\begin{equation}\nonumber
I_1=\left(\begin{array}{cccc}
-\displaystyle{\frac{1+2 \cos(\theta)}{2}}&a_{12} &a_{13}&a_{14}\\ [ 5 pt]
a_{12}&-\displaystyle{\frac{1+2 \sin(\theta)}{2}} &a_{23} &a_{24}\\ [5 pt]
a_{13}&a_{23}&\displaystyle{\frac{-1+2 \sin(\theta)}{2}} &a_{34}\\ [5 pt]
-a_{14}&-a_{24}&-a_{34} &\displaystyle{\frac{-1+2 \cos(\theta)}{2}}\\ \end{array}\right), 
\end{equation}
where $$a_{12}=\frac{\sqrt{(1-2 \sin(\theta))\cdot (1-2 \cos(\theta))}}{2},$$ \\ [-2 ex] $$a_{13}=\frac{\sqrt{(1+2 \sin(\theta))\cdot (1-2 \cos(\theta))}}{2},$$\\ [-2 ex]
 $$a_{14}=-\frac{\sqrt{(2 \cos(\theta)-1)\cdot (1+2 \cos(\theta))}}{2},$$\\ [-2 ex]
 $$a_{23}=\frac{\sqrt{(1+2 \sin(\theta))\cdot(1-2 \sin(\theta))}}{2},$$\\ [-2 ex]
 $$a_{24}=-\frac{\sqrt{(2 \sin(\theta)-1)\cdot(1+2 \cos(\theta))}}{2},$$\\ [-2 ex]
and $$a_{34}=-\frac{\sqrt{(1+2 \sin(\theta))\cdot (-1-2 \cos(\theta))}}{2}.$$
Since $\theta \in [\frac{5 \pi}{6}, \pi]$, all the terms $a_{ij}$ above are real.

Then $I_1$ is the order two $\mathbb{C}$-reflection with mirror  $\mathcal{P}_{1}$.
Let  $I_{i}=J I_{i-1}J^{-1}$, then $I_i$ is the order two $\mathbb{C}$-reflection with mirror  $\mathcal{P}_{i}$  for $i=2,3,4$.

By direct calculation we have $$\det(I_1)=\det(I_2)=\det(I_3)=\det(I_4)=-1,$$  and $$I^*_{i} H I_{i}=H$$ for $i=1,2,3,4$, here  $I^*_{i}$ is the Hermitian transpose of $I_{i}$. Moreover  $$(I_1 I_3)^{2}=(I_2 I_4)^{2}=id,$$ and $I_1I_2$ is parabolic, so we get a representation $\rho$ of  $G$ into $\mathbf{PU}(3,1)$ (with additional condition that $I_{i}I_{i+1}$ is parabolic  for $i=1,2,3,4$ mod 4). Note that $\det(I_i)=-1$, so $I_i \in \mathbf{U}(3,1)$, but $I_i \notin \mathbf{PU}(3,1)$, and  ${\rm e}^{\frac{\pi\rm{i}}{4}}\cdot I_i \in \mathbf{PU}(3,1)$. But we can also use the matrices $I_i$ to study these representations for simplicity of notations.


 In the following, we also denote $A_{i}=I_{i}I_{i+1}$  for $i=1,2,3,4$ mod  $4$.
Then $A_i$ is parabolic,  $A_1A_2=A_3A_4$ and $A_2A_3=A_4A_1$ have order 2.  Then $$\rho_{\theta}(K)=\langle A_1, A_2, A_3,A_4 \rangle$$ is an index two subgroup of $\Gamma=\rho_{\theta}(G)$.  We have $A_1A_2=I_1I_3$ is 	\begin{equation}\nonumber
\left(\begin{array}{cccc}
2 \cos(\theta)&0 &0&-\sqrt{4 \cos^2(\theta)-1}\\ [ 5 pt]
0&2 \sin(\theta) &-\sqrt{1-4 \sin^2(\theta)} &0\\ [ 5 pt]
0&-\sqrt{1-4 \sin^2(\theta)}&-2 \sin(\theta)&0\\ [ 5 pt]
-\sqrt{4 \cos^2(\theta)-1}&0&0&-2 \cos(\theta)\\ \end{array}\right).
\end{equation}
The matrix $A_1$ is a little complicated, we omit it, but it is easy to get it from the matrices of $I_1$ and $J$.


\section{Dirichlet domain  of $\rho_{\frac{5 \pi}{6}}(K) <\mathbf{PU}(2,1)$ in  ${\bf H}^2_{\mathbb C}$} \label{sec:Dirich2dim}

In this section,  we will prove Theorem \ref{thm:complex3dim} for $\theta =\frac{5 \pi}{6}$  via the Dirichlet domain of $\rho_{\frac{5 \pi}{6}}(K)$ with center  $p_0$. We note that the proof of Theorem  \ref{thm:complex3dim} when $\theta\in (\frac{5 \pi}{6}, \pi]$ in Section  \ref{sec:complex3dim} does not  work for $\theta=\frac{5 \pi}{6}$, since (the lifts of) four points $$\left\{p_0, ~~~I_1I_2(p_0), ~~I_1I_3(p_0),~~ I_1I_4(p_0)\right\}$$ are linearly dependent when $\theta=\frac{5 \pi}{6}$ (see Lemma \ref{lemma:noncoplane}).
In spite of this, the  proof in this section is  a  model for the proof of Theorem \ref{thm:complex3dim} for general $\theta$ in Section  \ref{sec:complex3dim}. 





	\subsection{Matrices in $\mathbf{PU}(2,1)$}\label{subsec:matricesin2dim}
From (\ref{vector:n1}) we have 
\begin{equation}\label{n12dim}
n_1=\left[\begin{matrix} \displaystyle{\frac{\sqrt{1+\sqrt{3}}}{2}}, ~~~~0,~~~~\frac{\sqrt{2}}{2},~~~~ \frac{\sqrt{\sqrt{3}-1}}{2} \end{matrix}\right]^{t}
\end{equation}
when $\theta=\frac{5 \pi}{6}$.

So  the second entry  of each $n_{i}$
is zero for $i=1, 2, 3, 4$. We take
\begin{equation}\label{n2dim}
n=\left[\begin{matrix} 0, 1,0, 0 \end{matrix}\right]^{t},
\end{equation}
  in  $\mathbb C{\bf P}^3-\overline{{\bf H}^3_{\mathbb C}}$, then $\langle n, n_{i} \rangle =0$ for $i=1,2,3,4$.
The intersection of  ${\bf H}^3_{\mathbb C}$  and   the dual of $n$ in  $\mathbb C{\bf P}^3$ with respect to the Hermitian form    is a copy of  ${\bf H}^2_{\mathbb C}$,  which is $$\mathcal{P}=\left\{	\left[
\begin{array}{c}
z_1 \\
0\\
z_3 \\
1 \\
\end{array} 
\right] \in {\bf H}^3_{\mathbb C}\right\}.$$
Each $I_{i}$ preserves this ${\bf H}^2_{\mathbb C}\hookrightarrow {\bf H}^{3}_{\mathbb C}$ invariant.
So we delete the second column and the second row of the matrix  $I_{i}$ in Subsection  \ref{subsec:matricesin3dim}, we get new matrices in $\mathbf{U}(2,1)$. But we still denote them by $I_{i}$ for the simplification of notations. 
We have
\begin{equation}\nonumber
I_1=\left(\begin{array}{ccc}
\displaystyle{\frac{\sqrt{3}-1}{2}}&	\displaystyle{\frac{1}{\sqrt{\sqrt{3}-1}}} &-\displaystyle{\frac{\sqrt{2}}{2}}\\ [12 pt]
\displaystyle{\frac{1}{\sqrt{\sqrt{3}-1}}}&0  &-\displaystyle{\frac{1}{\sqrt{\sqrt{3}+1}}}\\ [16 pt]
\displaystyle{\frac{\sqrt{2}}{2}}&\displaystyle{\frac{1}{\sqrt{\sqrt{3}+1}}} &-\displaystyle{\frac{\sqrt{3}+1}{2}}\\ \end{array}\right),
\end{equation}

\begin{equation}\nonumber
J=\left(\begin{matrix}
-1&	0 &0\\
0&-\rm{i}  &0\\

0&0 &1\\ \end{matrix}\right),
\end{equation}
and 
\begin{equation}\nonumber
A_1=I_{1}I_2=\left(\begin{array}{ccc}
\displaystyle{\frac{3+ \rm{i}+\sqrt{3}\rm{i}-\sqrt{3}}{2}}& -\rm{i}\displaystyle{\sqrt{\sqrt{3}-1}} & \displaystyle{\frac{\sqrt{2}(\sqrt{3}+\rm{i})}{2}}\\  [10 pt]
\displaystyle{\sqrt{\sqrt{3}-1}}&-\sqrt{3} \rm{i}  &\displaystyle{\sqrt{\sqrt{3}+1}}\\   [10 pt]
\displaystyle{\frac{\sqrt{2}(\sqrt{3}+\rm{i})}{2}}&-\rm{i}\displaystyle{\sqrt{\sqrt{3}+1}} &\displaystyle{\frac{3-\rm{i}+\sqrt{3}\rm{i}+\sqrt{3}}{2}}\\ \end{array}\right),
\end{equation}
$A_{i}=JA_{i-1}J^{-1}$ for $i=2,3,4$ and 
\begin{equation}\nonumber
 A_1A_2= I_{1}I_3=\left(\begin{array}{ccc}
-\sqrt{3}&0 &\sqrt{2}\\[2 pt] 
0&-1 &0\\[2 pt]
-\sqrt{2}&0 &\sqrt{3}\\ \end{array}\right).
\end{equation}

In fact the above $I_i$ for $i=1,2,3,4$ and $J$ are in  $\mathbf{U}(2,1)$, but not in  $\mathbf{PU}(2,1)$ since the determinants of them are not 1.   But we can also use these matrices to study the group action on  ${\bf H}^2_{\mathbb C}$ for simplicity of notations.




\subsection{A partial  Dirichlet domain $D_{R}$ of $\rho_{\frac{5 \pi}{6}}(K) <\mathbf{PU}(2,1)$ in ${\bf H}^2_{\mathbb C}$}\label{subsec:2dimDirichlet}

In  this subsection, we will define a subset $R \subset \rho_{\frac{5 \pi}{6}}(K)$, 
from which we have the partial  the partial  Dirichlet domain $D_{R}$  in ${\bf H}^2_{\mathbb C}$.

 
 First note that $I_1I_2$ is parabolic with fixed point 
 $$q_{12}=[\displaystyle{\frac{-1+\rm{i}}{\sqrt{2}\sqrt{\sqrt{3}+1}}},~~1,~~\frac{1+\rm{i}}{\sqrt{2}\sqrt{\sqrt{3}-1}}]^{t}$$ in $\partial {\bf H}^2_{\mathbb C}$.
 We denote by $p_{ij}=I_iI_j(p_0) \in {\bf H}^2_{\mathbb C}$ and $B_{ij}=B(p_0, p_{ij})$  the bisector with respect to the two points $p_0$ and $p_{ij}$ for certain $i,j \in \{1,2,3,4\}$.
Now \begin{itemize}
	\item $p_{0}=[0,0,1]^{t}$; \\ [-0.0 ex]
	
		\item $p_{12}=[\displaystyle{\frac{\sqrt{2}(\sqrt{3}+\rm{i})}{2}},~~~\sqrt{\sqrt{3}+1},~~~\frac{3-\rm{i}+\sqrt{3}+\sqrt{3}\rm{i}}{2}]^{t}$; \\ [-0.0 ex]
		
			\item 
		$p_{21}=[-\displaystyle{\frac{\sqrt{2}(\sqrt{3}-\rm{i})}{2}},~~~-\rm{i}\sqrt{\sqrt{3}+1},~~~\frac{3+\rm{i}+\sqrt{3}-\sqrt{3}\rm{i}}{2}]^{t}$; \\ [-0.0 ex]
		\item 	$p_{23}=[-\displaystyle{\frac{\sqrt{2}(\sqrt{3}+\rm{i})}{2}},~~~-\rm{i}\sqrt{\sqrt{3}+1},~~~\frac{3-\rm{i}+\sqrt{3}+\sqrt{3}\rm{i}}{2}]^{t}$; \\ [-0.0 ex]
		\item	$p_{32}=[\displaystyle{\frac{\sqrt{2}(\sqrt{3}-\rm{i})}{2}},~~~-\sqrt{\sqrt{3}+1},~~~\frac{3+\rm{i}+\sqrt{3}-\sqrt{3}\rm{i}}{2}]^{t}$; \\ [-0.0 ex]
		\item	$p_{34}=[\displaystyle{\frac{\sqrt{2}(\sqrt{3}+\rm{i})}{2}},~~~-\sqrt{\sqrt{3}+1},~~~\frac{3-\rm{i}+\sqrt{3}+\sqrt{3}\rm{i}}{2}]^{t}$; \\ [-0.0 ex]
		\item	$p_{43}=[-\displaystyle{\frac{\sqrt{2}(\sqrt{3}-\rm{i})}{2}},~~~\rm{i}\sqrt{\sqrt{3}+1},~~~\frac{3+\rm{i}+\sqrt{3}-\sqrt{3}\rm{i}}{2}]^{t}$; \\ [-0.0 ex]
		\item	$p_{41}=[-\displaystyle{\frac{\sqrt{2}(\sqrt{3}+\rm{i})}{2}},~~~\rm{i}\sqrt{\sqrt{3}+1},~~~\frac{3-\rm{i}+\sqrt{3}+\sqrt{3}\rm{i}}{2}]^{t}$; \\ [-0.0 ex]
		\item	$p_{14}=[\displaystyle{\frac{\sqrt{2}(\sqrt{3}-\rm{i})}{2}},~~~\sqrt{\sqrt{3}+1},~~~\frac{3+\rm{i}+\sqrt{3}-\sqrt{3}\rm{i}}{2}]^{t}$; \\ [-0.0 ex]
		
		\item	$p_{13}=[\sqrt{2},0,\sqrt{3}]^{t}$; \\ [-0.0 ex]
		
		\item	$p_{24}=[-\sqrt{2},0,\sqrt{3}]^{t}$.

	\end{itemize}

We fix lifts $\bf{p_0}$ and $\bf{p_{ij}}$ of $p_0$ and $p_{ij}$ as vectors in $\mathbb{C}^{2,1}$,  such that the entries of  $\bf{p_0}$ and $\bf{p_{ij}}$ are just the same as  entries of $p_0$ and $p_{ij}$ above.  We have $$\langle {\bf p_0}, {\bf p_0} \rangle =\langle {\bf p_{ij}},{\bf p_{ij}} \rangle=-1$$
	for ${\bf p_{ij}}$ above. All of these lifts have the same norm will be very convenient  later. 

	


Recall that $A_{i}=I_{i}I_{i+1}$ for $i=1,2,3,4$ mod 4.   Let  $R \subset \rho_{\frac{5 \pi}{6}}(K)$  be the set of ten words in $\Gamma$:
$$\left\{ (I_1I_2)^{\pm 1},~~~(I_2I_3)^{\pm 1},~~~(I_3I_4)^{\pm 1},~~~(I_4I_1)^{\pm 1},~~~ I_1I_3,~~~ I_2I_4 \right\}.$$
We will show the partial  Dirichlet domain $D_{R}$ centered at the fixed point $p_0$ of $J$  is in fact the  Dirichlet domain of $\rho_{\frac{5 \pi}{6}}(K)$. The main tool for our study is the Poincar\'e polyhedron theorem, which gives
sufficient conditions for $D_{R}$
to be a fundamental domain of the group. We refer to \cite{ParkerWill:2017} for the precise statement of this
version of Poincar\'e polyhedron theorem we need.
The  main technical result in this section is 
\begin{thm} \label{thm:2dimdirichlet}
	$D_{R}$ is the Dirichlet  domain of  $\rho_{\frac{5 \pi}{6}}(K)$ acting on   ${\bf H}^{2}_{\mathbb C}$ with center $p_0$. Moreover,
	the group $\rho_{\frac{5 \pi}{6}}(K) = \langle A_1, A_2,A_3,A_4 \rangle$ has a presentation  $$\langle A_1, A_2,A_3,A_4: A_1A_2A_3A_4=id,(A_1A_2)^2=(A_2A_3)^2=(A_3A_4)^2=(A_4A_1)^2=id\rangle.$$
So $\rho_{\frac{5 \pi}{6}}:K \rightarrow \mathbf{PU}(2,1)$	is a discrete and faithful presentation of $K$.
	
\end{thm}


\subsection{Intersection patterns of the bisectors for $D_{R}$}\label{subsec:2dimintersection}
In this subsection, we will study the information on intersection patterns of the bisectors for $D_{R}$.  We summarize  the intersections of them  in Table \ref{table:intersecton} and we will show this carefully. Moreover,  Table \ref{table:intersecton} should be compared with Figures 		\ref{figure:22infty2dimdirichlet} and \ref{figure:22infty2dimdirichletabstract}.
Where Figure \ref{figure:22infty2dimdirichlet}	is  a realistic view of the boundary of  the partial  Dirichlet domain $D_{R}$ of $\rho_{\frac{5 \pi}{6}}(K) <\mathbf{PU}(2,1)$. For example,  the sphere labeled by  $B_{41}$ is in fact the spinal sphere $B_{41} \cap \partial {\bf H}^{2}_{\mathbb C}$.  Unfortunately, since the union of these spinal spheres are   ``twisted"  in $\partial {\bf H}^{2}_{\mathbb C}$, it seems impossible to see all the spinal spheres from  only one point of view. 
 Figure \ref{figure:22infty2dimdirichletabstract}  is an abstract picture of the boundary of the partial   Dirichlet domain $D_{R}$, which is much more transparent.

		 
		 
%
		 
		 

\begin{table}[H]
		\centering
	\begin{tabular}{|c|c|}
		\hline
		
	 $B_{12}\cap B_{21}$,  tangent  &  $B_{12}\cap B_{41}=\emptyset$  	 \\  [2 ex]

	$B_{12}\cap B_{23}=\emptyset$   & $B_{12}\cap B_{14}\neq \emptyset$  
	 \\  [2 ex]
	 
	 	$B_{12}\cap B_{32}=\emptyset$   & $B_{12}\cap B_{13}\neq \emptyset$  
	 \\  [2 ex]
	 
	 	$B_{12}\cap B_{34}\neq \emptyset$   & $B_{12}\cap B_{24}=\emptyset$  
	 \\  [2 ex]
	 
	 	$B_{12}\cap B_{43}=\emptyset$   & $B_{13}\cap B_{24}= \emptyset$  
	 \\  
		\hline	
	\end{tabular}
	\caption{The intersections of bisectors we  should be   concerned with up to $J$-action.}
	\label{table:intersecton}
\end{table}

We first  consider the intersections of the bisector $B_{12}$ with other bisectors.
\begin{prop}\label{prop:B12} For the bisector $B_{12}$ of $I_1I_2$, we have
	
	\begin{enumerate}
			\item  \label{item:B12andB21}$B_{12}$ is tangent to   $B_{21}$;
			
		\item  \label{item:B12empty}$B_{12}$ does not intersect  $B_{23}$, $B_{32}$, $B_{43}$, $B_{41}$ and $B_{24}$;
		
			\item  \label{item:B12nonempty} $B_{12}$ intersects $B_{34}$ in a non-empty Giraud disk  $B_{12}\cap B_{34}$. Moreover,  the disk $B_{12}\cap B_{34}$ lies in the component of  ${\bf H}^2_{\mathbb{C}}-B_{13}$  which does not contain the  point $p_0$. In particular,  $B_{12}\cap B_{34}$ does not lie in the partial Dirichlet domain $D_{R}$.
  	 \end{enumerate}
  	 
		\end{prop}

\begin{proof} The trace of $I_1I_2$ is $3$ (when we normalize them such that $\det(I_1I_1)=1$), so $I_1I_2$ is unipotent. By Theorem 6.1 of \cite{Phillips:1992}, the Dirichlet domain (with center $p_{0}$) of the infinite cyclic group  $\langle I_1I_2 \rangle $ has two sides $B_{12}$ and $B_{21}$. These two bisectors intersect exactly in the fixed point $q_{12}$ of $I_1I_2$.  This proves (\ref{item:B12andB21}) of Proposition \ref{prop:B12}.

We now  consider $B_{12}\cap B_{23}$. It is easy to see  $\bf{p_{12}}$,  $\bf{p_{23}}$ and $\bf{ p_{0}}$ are linearly independent in $\mathbb{C}^{2,1}$. 
In Equation (\ref{equaation:girauddisk2dim}), we take ${\bf q}=\bf{p_{12}}$,  ${\bf r}=\bf{p_{23}}$ and ${\bf p}=\bf{ p_{0}}$, then we can parameterize the intersection $B_{12}\cap B_{23}$  of the bisectors $B_{12}$ and $B_{23}$  by $V=V(z_1,z_2)$ with $\langle V,V \rangle <0$.
Where
	\begin{equation}\label{para:VB12capB23}
V=\left(
\begin{array}{c}
-\displaystyle{\frac{(\rm{i}\rm{e}^{r \rm{i}}-\rm{e}^{s \rm{i}}+2-2 \rm{i})\sqrt{3}+\rm{i}\rm{e}^{r \rm{i}}-\rm{i}\rm{e}^{s\rm{i}}+2-4\rm{i}}{\sqrt{\sqrt{3}+1}}} \\   [4 ex]
	-\displaystyle{\frac{(\rm{e}^{r \rm{i}}+\rm{e}^{s \rm{i}}-2)\sqrt{3}-\rm{i}\rm{e}^{r \rm{i}}-\rm{i}\rm{e}^{s \rm{i}}-4+6\rm{i}}{\sqrt{2}}} \\   [4 ex]
\displaystyle{\frac{1-\rm{i}+(1+\rm{i})\sqrt{3}}{\sqrt{\sqrt{3}-1}}} \\
\end{array}
\right),
\end{equation}	
and $(z_1,z_2)=(\rm{e}^{r \rm{i}},\rm{e}^{s \rm{i}}) \in \mathbb{S}^1 \times \mathbb{S}^1$.
Now $\langle V,V \rangle = V^{*} \cdot H \cdot  V$ is
	\begin{flalign} \label{para:B12andB23} &
(8\sqrt{3}+4)\sin(s)-(8\sqrt{3}+20)\cos(r)-(8\sqrt{3}+16)\cos(s)
& \\ &+(2\sqrt{3}+2)\sin(r-s)+4\cos(r-s)-8\sin(r)+16\sqrt{3}+48.& \nonumber \end{flalign}
With Maple, the minimum  of (\ref{para:B12andB23}) is 12.752 numerically. In particular, any $V$ in (\ref{para:VB12capB23}) is a positive vector. So $B_{12}\cap B_{23}=\emptyset$.

We then consider $B_{12}\cap B_{32}$. 
In Equation (\ref{equaation:girauddisk2dim}), we take ${\bf q}=\bf{p_{12}}$,  ${\bf r}=\bf{p_{32}}$ and ${\bf p}=\bf{p_{0}}$, then we can parameterize the intersection  $B_{12}\cap B_{32}$ of the bisectors $B_{12}$ and $B_{32}$  by $V=V(z_1,z_2)$ with $\langle V,V \rangle <0$.
Where
\begin{equation}\label{para:VB12capB32}
V=\left(
\begin{array}{c}
\displaystyle{\frac{(\rm{e}^{r \rm{i}}+\rm{e}^{s \rm{i}}-4)\sqrt{3}+\rm{e}^{r \rm{i}}+\rm{e}^{s\rm{i}}-6}{\sqrt{\sqrt{3}+1}} }\\  [4 ex]
\displaystyle{\frac{(\rm{e}^{r \rm{i}}-\rm{e}^{s \rm{i}}-2\rm{i})\sqrt{3}+\rm{i}\rm{e}^{r \rm{i}}+\rm{e}^{s \rm{i}}}{\sqrt{2}}} \\ [4 ex]
-\sqrt{6}{\sqrt{\sqrt{3}+1}} \\
\end{array}
\right),
\end{equation}	
and $(z_1,z_2)=(\rm{e}^{r \rm{i}},\rm{e}^{s \rm{i}}) \in \mathbb{S}^1\times \mathbb{S}^1$.
Now $\langle V,V \rangle = V^{*} \cdot H \cdot  V$ is
\begin{flalign} \label{item:B12andB32} &
-(10\sqrt{3}+12)\cdot (\cos(r)+\cos(s))+2\sqrt{3}\cdot (\sin(r-s)+\cos(r-s))
& \\ &+14\sqrt{3}-6\sin(r)+6\sin(s)+36.& \nonumber \end{flalign}
With Maple, the minimum  of (\ref{item:B12andB32}) is 4.78 numerically. In particular, any $V$ in (\ref{para:VB12capB32}) is a positive vector. So $B_{12}\cap B_{32}=\emptyset$.

For $B_{12}\cap B_{43}$, in Equation (\ref{equaation:girauddisk2dim}), we take ${\bf q}=\bf{p_{12}}$,  ${\bf r}=\bf{p_{43}}$ and ${\bf p}=\bf{p_{0}}$. Then we can parameterize the intersection  $B_{12}\cap B_{43}$   of the bisectors $B_{12}$ and $B_{43}$  by $V=V(z_1,z_2)$ with $\langle V,V \rangle <0$.
Where
\begin{equation}\label{para:VB12capB43}
V=\left(
\begin{array}{c}
\displaystyle{\frac{(\rm{e}^{s \rm{i}}+\rm{i}\rm{e}^{r \rm{i}}-2-2\rm{i})\sqrt{3}+\rm{i}\rm{e}^{r \rm{i}}+\rm{e}^{s\rm{i}}-4-4\rm{i}}{\sqrt{\sqrt{3}+1}}} \\  [4 ex] 
-\displaystyle{\frac{(\rm{e}^{r \rm{i}}+\rm{e}^{s \rm{i}}-4)\sqrt{3}+\rm{i}\rm{e}^{r \rm{i}}-\rm{i}\rm{e}^{s \rm{i}}-2}{\sqrt{2}}} \\  [4 ex]
(1-\rm{i}){\sqrt{\sqrt{3}-1}} \\
\end{array}
\right),
\end{equation}	
and $(z_1,z_2)=(\rm{e}^{r \rm{i}},\rm{e}^{s \rm{i}}) \in \mathbb{S}^1\times \mathbb{S}^1$.
Now $\langle V,V \rangle = V^{*} \cdot H \cdot V$ is
\begin{flalign} \label{item:B12andB43} &
2\cdot (\cos(r-s)-\sin(r-s))-(6\sqrt{3}+20)\cdot (\cos(r)+\cos(s))
& \\ &-4\sqrt{3}\sin(r-s)+(8\sqrt{3}+10) \cdot (\sin(r)-\sin(s))+20\sqrt{3}+54.& \nonumber \end{flalign}
With Maple, the minimum  of (\ref{item:B12andB43}) is 20.51 numerically. In particular, any $V$ in (\ref{para:VB12capB43}) is a positive vector. So $B_{12}\cap B_{43}=\emptyset$.

Since $J(B_{41}\cap B_{12})=B_{12}\cap B_{23}$, and  we have proved that $B_{12}\cap B_{23}$ is empty, then  
$B_{41}\cap B_{12}$ is empty. 

For $B_{12}\cap B_{24}$, in Equation (\ref{equaation:girauddisk2dim}), we take ${\bf q}=\bf{p_{12}}$,  ${\bf r}=\bf{p_{24}}$ and ${\bf p}=\bf{p_{0}}$. Then we can parameterize the intersection $B_{12}\cap B_{24}$ of the bisectors $B_{12}$ and $B_{24}$  by $V=V(z_1,z_2)$ with $\langle V,V \rangle <0$.
Where
\begin{equation}\label{para:VB12capB24}
V=\left(
\begin{array}{c}
\displaystyle{\frac{(\rm{e}^{s \rm{i}}-1)\sqrt{3}+\rm{e}^{s\rm{i}}-3}{\sqrt{\sqrt{3}+1}}} \\ [4 ex]
\displaystyle{-\frac{(\rm{e}^{s \rm{i}}-1+2\rm{i})\sqrt{3}-\rm{i}\rm{e}^{s \rm{i}}+2\rm{e}^{r \rm{i}}-6-\rm{i}}{\sqrt{2}}} \\ [4 ex]
\sqrt{2}{\sqrt{\sqrt{3}+1}} \\
\end{array}
\right),
\end{equation}	
and $(z_1,z_2)=(\rm{e}^{r \rm{i}},\rm{e}^{s \rm{i}}) \in \mathbb{S}^1\times \mathbb{S}^1$.
Now $\langle V,V \rangle = V^{*}\cdot H  \cdot V$ is
\begin{flalign} \label{item:B12andB24} &
2\sqrt{3}\cdot (\cos(r-s)-\sin(s))-2\sin(r-s)-(10\sqrt{3}+8)\cos(s)
& \\ &-(12+2\sqrt{3})\cos(r)+(-2+4\sqrt{3})\sin(r)+6\sqrt{3}+32.& \nonumber \end{flalign}
With Maple, the minimum  of (\ref{item:B12andB24}) is 4.58 numerically. In particular, any $V$ in (\ref{para:VB12capB24}) is a positive vector. So $B_{12}\cap B_{24}=\emptyset$.
We end the proof  of   (\ref{item:B12empty}) of Proposition \ref{prop:B12}.

For $B_{12}\cap B_{34}$, in Equation (\ref{equaation:girauddisk2dim}), we take ${\bf q}=\bf{p_{12}}$,  ${\bf r}=\bf{p_{34}}$ and ${\bf p}=\bf{p_{0}}$. Then we can parameterize the intersection $B_{12}\cap B_{34}$ of the bisectors $B_{12}$ and $B_{34}$  by $V=V(z_1,z_2)$ with $\langle V,V \rangle <0$.
Where
\begin{equation}\label{para:VB12capB34}
V=\left(
\begin{array}{c}
\displaystyle{\frac{(\rm{e}^{r \rm{i}}+\rm{e}^{s \rm{i}}-4)\sqrt{3}+\rm{e}^{r \rm{i}}+\rm{e}^{s\rm{i}}-6+2\rm{i}}{\sqrt{\sqrt{3}+1}}} \\   [4 ex]
\displaystyle{\frac{(\rm{i}-\sqrt{3})(\rm{e}^{s \rm{i}}-\rm{e}^{r \rm{i}})}{\sqrt{2}}} \\  [4 ex]
(\rm{i}-\sqrt{3})\sqrt{2}{\sqrt{\sqrt{3}+1}} \\
\end{array}
\right),
\end{equation}	
and $(z_1,z_2)=(\rm{e}^{r \rm{i}},\rm{e}^{s \rm{i}}) \in \mathbb{S}^1 \times \mathbb{S}^1$.
Now $\langle V,V \rangle = V^{*}\cdot H \cdot V$ is
\begin{flalign} \label{item:B12andB34} &
-(8\sqrt{3}+12)\cdot (\cos(r)+\cos(s))+(2\sqrt{3}-2)\cos(r-s)
& \\ &+4\cdot (\sin(r)+\sin(s))+14\sqrt{3}+26.& \nonumber \end{flalign}
There is $(r,s)$ such that (\ref{item:B12andB34}) is negative, so $V=V(\rm{e}^{r \rm{i}},\rm{e}^{s \rm{i}})$ is a negative vector. A sample point is $(r,s)=(0, -\frac{\pi}{16})$.

\begin{figure}
	\begin{center}
		\begin{tikzpicture}
		\node at (0,0) {\includegraphics[width=4cm,height=4cm]{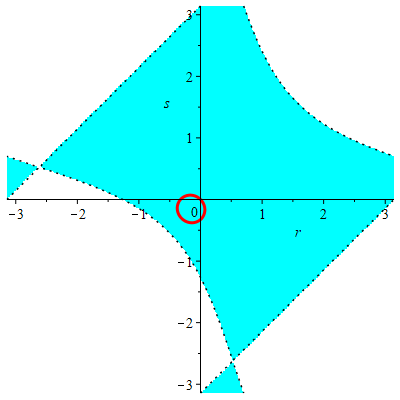}};
		
		\end{tikzpicture}
	\end{center}
	\caption{$B_{12}\cap B_{34}$ lies in the half space of  ${\bf H}^2_{\mathbb{C}}-B_{13}$  which does not contain the fixed point $p_0$ of $J$. The red circle is  $B_{12}\cap B_{34}\cap \partial {\bf H}^2_{\mathbb C}$ in coordinates of $(r,s) \in [-\pi,\pi]^2$. So the small disk bounded by the red circle is $B_{12}\cap B_{34}\cap {\bf H}^2_{\mathbb C}$. The cyan region consists of points such that whose distance to $p_{13}$ is smaller than whose distance to $p_{0}$. This figure  illustrates  that  the disk $B_{12}\cap B_{34}$ lies in the component of  ${\bf H}^2_{\mathbb{C}}-B_{13}$  which does not contain the  point $p_0$.}
	\label{figure:B13coverB12capB34}
\end{figure}

But for $V$ in (\ref{para:VB12capB34}), $|V^{*} \cdot H \cdot {\bf p_0}|^2$ is the constant $8+8\sqrt{3}$, which is 21.85 numerically.
For  $V$ in (\ref{para:VB12capB34}), $|V^{*}\cdot H \cdot  {\bf p_{13}}|^2$ is 
\begin{flalign} \label{item:B12andB34distancep13} &
4(1+\sqrt{3}) \cdot (\cos(r-s)-(\cos(r)+\cos(s))\sqrt{3}-\sin(r)-\sin(s)+3).
& \end{flalign}
With Maple, the maximum  of (\ref{item:B12andB34distancep13}) with the condition $ V^{*}\cdot H \cdot  V=0$
is 14.23 numerically, which is smaller than $|V^{*} \cdot H \cdot  {\bf p_0}|^2$.  Moreover, the  minimum  of (\ref{item:B12andB34distancep13}) with the condition $ V^{*}\cdot  H \cdot  V=0$ is 5.85. 
In particular, the ideal  boundary of $B_{12}\cap B_{34}$, say $B_{12}\cap B_{34}\cap \partial {\bf H}^2_{\mathbb{C}}$, is disjoint from $B_{13}$.
Then we have $B_{12}\cap B_{34}$ is disjoint from $B_{13}$ by well-known properties of bisectors in ${\bf H}^2_{\mathbb{C}}$, see	\cite{Go}.  For the sample point $V=V(\rm{e}^{0\cdot  \rm{i}},\rm{e}^{-\frac{\pi}{16} \rm{i}})$, it is easy to see 
its distance to $p_{13}$ is smaller than its distance to $p_{0}$.
We have $B_{12}\cap B_{34}$ lies in the half space of  ${\bf H}^2_{\mathbb{C}}-B_{13}$  which does not contain the fixed point $p_0$ of $J$. In particular,  $B_{12}\cap B_{34}$ does not lie in the partial Dirichlet domain $D_{R}$. See also  Figure 	\ref{figure:B13coverB12capB34} for an  illustration of this fact.
This ends the proof  of   (\ref{item:B12nonempty}) of Proposition \ref{prop:B12}.

\end{proof}




\begin{prop}\label{prop:B13} For the bisector $B_{13}$ of $I_1I_3$, we have $B_{13}$ does not intersect  $B_{21}$, $B_{23}$, $B_{43}$, $B_{41}$ and $B_{24}$.
\end{prop}

\begin{proof}We have proved $B_{12}\cap B_{24}=\emptyset$. By the $\langle J \rangle= \mathbb{Z}_4$ symmetry, we have $$B_{13}\cap B_{23}=\emptyset, ~~~~B_{13}\cap B_{41}=\emptyset.$$
	The fact that $B_{13}\cap B_{21}=\emptyset$ and $B_{13}\cap B_{43}=\emptyset$ can be proved similarly to $B_{12}\cap B_{24}=\emptyset$ in Proposition \ref{prop:B12}.

For  $B_{13}\cap B_{24}$. It is easy to  see  $p_0$, $p_{13}$ and $p_{24}$ lie in the $\mathbb{C}$-line 
$$l=\left\{	\left[z_1,~~0, ~~1\right]^{t} \in {\bf H}^2_{\mathbb C}\right\}.$$
Now it is easy to see $B_{13}\cap B_{24} \cap l=\emptyset$. Then from the projection of $ {\bf H}^2_{\mathbb C}$ to $l$, we get  $B_{13}\cap B_{24}=\emptyset$.


\end{proof}

\begin{prop}\label{prop:B12intersection} For the bisectors $B_{12}$,  $B_{13}$ and  $B_{14}$,
	 we have:
	
	\begin{enumerate}

		\item Each of the intersections   $B_{12} \cap B_{13}$, $B_{12} \cap B_{14}$ and   $B_{13}\cap B_{14}$ is a non-empty Giraud disk;
		
		\item  \label{item:B12B13B14}The triple intersection $B_{12} \cap B_{13}\cap B_{14}$ is an arc  in each of  the Giraud disks $B_{12} \cap B_{13}$, $B_{12} \cap B_{14}$ and $B_{13} \cap B_{14}$.
		
	
		

	\end{enumerate}
\end{prop}

\begin{proof}
For $B_{12}\cap B_{14}$, in Equation (\ref{equaation:girauddisk2dim}), we take ${\bf q}=\bf{p_{12}}$,  ${\bf r}=\bf{p_{14}}$ and ${\bf p}=\bf{p_{0}}$. It is easy to see these three vectors are linearly independent. Then we can parameterize the intersection $B_{12}\cap B_{14}$  of the bisectors $B_{12}$ and $B_{14}$  by $V=V(z_1,z_2)$ with $\langle V,V \rangle <0$.
Where
\begin{equation}\label{para:VB12capB14}
V=\left(
\begin{array}{c}
\displaystyle{\frac{(\rm{e}^{s \rm{i}}-\rm{e}^{r \rm{i}})\sqrt{3}-\rm{e}^{r \rm{i}}+\rm{e}^{s\rm{i}}-2\rm{i}}{\sqrt{\sqrt{3}+1}}} \\ [4 ex]
\displaystyle{\frac{(\rm{e}^{r \rm{i}}-\rm{e}^{s \rm{i}}-2\rm{i})\sqrt{3}+\rm{i}\rm{e}^{r \rm{i}}+\rm{i}\rm{e}^{s \rm{i}}}{\sqrt{2}}} \\ [4 ex]
-\rm{i}\sqrt{2}{\sqrt{\sqrt{3}+1}} \\
\end{array}
\right)
\end{equation}	
and $(z_1,z_2)=(\rm{e}^{r \rm{i}},\rm{e}^{s \rm{i}}) \in \mathbb{S}^1 \times \mathbb{S}^1$.
Now $\langle V,V \rangle = V^{*} \cdot H \cdot V$ is
\begin{flalign} \label{item:B12andB14} &
2\sqrt{3} \cdot (\sin(r-s)-\cos(r)-\cos(s)-\cos(r-s)+1)
& \\ &-4\cos(r-s)-2\sin(r)+2\sin(s)+8.& \nonumber \end{flalign}
There are $(r,s)$ such that (\ref{item:B12andB14}) is negative, that is  $V=V(\rm{e}^{r \rm{i}},\rm{e}^{s \rm{i}})$ is a  negative vector. For example when $(r,s)=(0,0)$. Then $B_{12}\cap B_{14}$ is non empty, it is a Giraud disk.

Now for $V$  in (\ref{para:VB12capB14}), $|V^{*} \cdot H \cdot {\bf p_0}|^2$ is $2+2\sqrt{3}$, and  $|V^{*} \cdot H \cdot { \bf p_{13}}|^2$
is 
\begin{flalign} \label{item:B12andB14andB13} &
4\sqrt{3}(\sin(s)-\sin(r)-\cos(r-s))+4(\sin(s)-\sin(r)-\cos(r-s))
& \\ &+6\sqrt{3}+6.& \nonumber \end{flalign}
The solutions of  $V$  in (\ref{para:VB12capB14}) with the condition $$|V^{*} \cdot H \cdot {\bf  p_{13}}|^2=2+2\sqrt{3}$$ are 
$$\{s=-\frac{\pi}{2}\},~~~ \{ r= \frac{\pi}{2}\}, ~~~\{r=s\}.$$
It is easy to see when $s=-\frac{\pi}{2}$ or $r=\frac{\pi}{2}$ then $V$ is positive.
But when $r=s$ with $$r \in \left(-\arccos(\frac{\sqrt{3}}{3}),-\arccos(\frac{\sqrt{3}}{3})\right),$$ $V$ is a negative point in $B_{12}\cap B_{14}$. So $B_{12}\cap B_{14}\cap B_{13}$ is an arc in $B_{12}\cap B_{14}$.
See Figure 	\ref{figure:B12coverB14capB13}  for $B_{12}\cap B_{14}$ in coordinates  $(r,s) \in [-\pi,\pi]^2$. 


\begin{figure}
	\begin{center}
		\begin{tikzpicture}
		\node at (0,0) {\includegraphics[width=4cm,height=4cm]{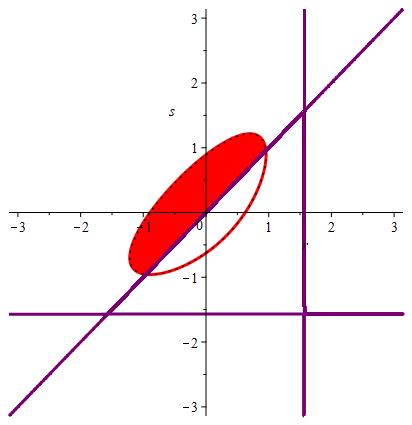}};

		\end{tikzpicture}
	\end{center}
	\caption{The small disk bounded by the red circle is the Giraud disk $B_{12}\cap B_{14}$. The three purple lines are points (in the  Giraud torus  $\hat{B}(p_{0},p_{12},p_{14}$) such that its distances to $p_{13}$ and  $p_{0}$ are the same.  So the red half disk is the part of  $B_{12}\cap B_{14}$ which lies in the half space of  ${\bf H}^2_{\mathbb{C}}-B_{13}$   containing the fixed point of $J$, and then  it is the ridge $s_{12}\cap s_{14}$.
	 }
	\label{figure:B12coverB14capB13}
\end{figure}


\begin{figure}
	\begin{center}
		\begin{tikzpicture}
		\node at (0,0) {\includegraphics[width=4cm,height=4cm]{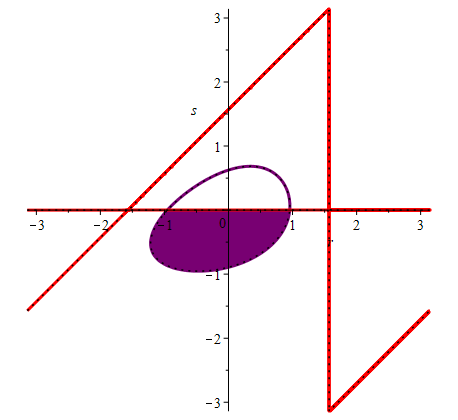}};

		\end{tikzpicture}
	\end{center}
	\caption{The small disk bounded by the purple circle is the Giraud disk $B_{12}\cap B_{13}$. The three red lines are points (in the  Giraud torus  $\hat{B}(p_{0},p_{12},p_{13}$)  such that its distances to $p_{14}$ and $p_{0}$ are the same. So the part of  $B_{12}\cap B_{13}$ which lies in the half space of  ${\bf H}^2_{\mathbb{C}}-B_{14}$  containing the fixed point of $J$ is the purple half disk, and then it is the ridge $s_{12}\cap s_{13}$.}
	\label{figure:B12andB13capB14}
\end{figure}

For $B_{12}\cap B_{13}$, we have proved that $B_{12}\cap B_{13}$ is non-empty, so it is a Giraud disk. In Equation (\ref{equaation:girauddisk2dim}), we take ${\bf q}=\bf{p_{12}}$,  ${\bf r}=\bf{p_{13}}$ and ${\bf p}=\bf{p_{0}}$. It is easy to see these three points are linearly independent. Then we can parameterize the intersection $B_{12}\cap B_{13}$ of the bisectors $B_{12}$ and $B_{13}$  by $V=V(z_1,z_2)$ with $\langle V,V \rangle <0$.
Where
\begin{equation}\label{para:VB12capB13}
V=\left(
\begin{array}{c}
\displaystyle{\frac{(\rm{e}^{s \rm{i}}-1)\sqrt{3}+\rm{e}^{s\rm{i}}-3}{\sqrt{\sqrt{3}+1}}} \\  [4 ex]
\displaystyle{\frac{(-\rm{e}^{s \rm{i}}-1)\sqrt{3}+\rm{i}\rm{e}^{s \rm{i}}+2\rm{e}^{r \rm{i}}-\rm{i}}{\sqrt{2}}} \\  [4 ex]
-\sqrt{2}{\sqrt{\sqrt{3}+1}} \\
\end{array}
\right),
\end{equation}	
and $(z_1,z_2)=(\rm{e}^{r \rm{i}},\rm{e}^{s \rm{i}}) \in \mathbb{S}^1 \times \mathbb{S}^1$.
Now $\langle V,V \rangle = V^{*} \cdot H \cdot  V$ is
\begin{flalign} \label{item:B12andB13} &
2\sqrt{3}\cdot (\sin(s)-\cos(r)-\cos(s)-\cos(r-s)+1)
& \\ &+2\sin(r-s)-2\sin(r)-4\cos(s)+8.& \nonumber \end{flalign}
Now for $V$  in (\ref{para:VB12capB13}), $|V^{*} \cdot H \cdot {\bf p_0}|^2$ is $2+2\sqrt{3}$, and 
$|V^{*} \cdot H \cdot{\bf  p_{14}}|^2$
is 
\begin{flalign} \label{item:B12andB13} &
4\sqrt{3}\cdot (\sin(r-s)-\sin(r)-\cos(s))+4\cdot (\sin(r-s)-\sin(r)-\cos(s))
& \\ &+6\sqrt{3}+6.& \nonumber \end{flalign}
The solutions of  $V$  in (\ref{para:VB12capB13}) with the condition $$|V^{*} \cdot H \cdot {\bf p_{14}}|^2=2+2\sqrt{3}$$ are 
$$\{s=0\},~~~ \{ r= \frac{\pi}{2}\}, ~~~\{r=s-\frac{\pi}{2}\}.$$
It is easy to see when $ r= \frac{\pi}{2}$ and $r=s-\frac{\pi}{2}$ then $V$ is positive.
But when $s=0$ with $$r \in \left(-\arccos(\frac{\sqrt{3}}{3}),-\arccos(\frac{\sqrt{3}}{3})\right),$$ $V$ is a negative point in $B_{12}\cap B_{13}$.
See Figure 	\ref{figure:B12andB13capB14}  for $B_{12}\cap B_{13}$ in coordinates $(r,s) \in [-\pi,\pi]^2$.  

Similarly, $B_{13}\cap B_{14}$ is a non-empty Giraud disk.  We omit it. 
\end{proof}





From Propositions  \ref{prop:B12}, \ref{prop:B13} and  \ref{prop:B12intersection}, we have the following two propositions, which study  the  combinatorics of 3-sides and ridges of  $D_{R}$.
They  are crucial for the application of the  Poincar\'e polyhedron theorem.
For each $I_{i}I_{j} \in R$, the side $s_{ij}$ by definition is $B_{ij} \cap D_{R}$, which is a 3-dimensional object. The reader may compare to  Figure	\ref{figure:22infty2dimdirichletabstract}, which illustrates the ideal boundary behaviors of these 3-sides $s_{ij}$. 

\begin{prop}\label{prop:s12}	
	
	The side  $s_{12}=B_{12}\cap D_{R}$  is 3-ball in $ {\bf H}^2_{\mathbb{C}} \cup \partial {\bf H}^2_{\mathbb{C}}$:
	\begin{itemize} \item   $s_{12} \cap \partial {\bf H}^2_{\mathbb{C}}$ is a disk with the point $q_{12}$ in its interior;
		\item The frontier of $s_{12} \cap  {\bf H}^2_{\mathbb{C}}$ is a disk   consisting  of two half disks $s_{12}\cap s_{14}$ and $s_{12}\cap s_{13}$ glued  along the arc $$s_{12}\cap s_{13}\cap s_{14}=B_{12}\cap B_{13}\cap B_{14}.$$
		
		\end{itemize}
\end{prop}

\begin{prop}\label{prop:s13}	
	
	The side  $s_{13}=B_{13}\cap D_{R}$  is 3-ball in ${\bf H}^2_{\mathbb{C}} \cup \partial {\bf H}^2_{\mathbb{C}}$:  
	
		\begin{itemize} \item  $s_{13} \cap \partial {\bf H}^2_{\mathbb{C}}$ is an annulus;
			 \item  The frontier of $s_{13} \cap  {\bf H}^2_{\mathbb{C}}$  are two disks, each of them  consists of two half disks:  one component is the union of $s_{13}\cap s_{12}$ and $s_{13}\cap s_{14}$ glued along  the arc  $$s_{12}\cap s_{13}\cap s_{14}=B_{12}\cap B_{13}\cap B_{14},$$ and the other component is union of $s_{13}\cap s_{32}$ and $s_{13}\cap s_{34}$ glued along  the arc  $$s_{13}\cap s_{32}\cap s_{34}=B_{13}\cap B_{32}\cap B_{34}.$$
		\end{itemize}
\end{prop}

We have similar properties for the 3-side $s_{21}$.  Since there is a $\langle J \rangle= \mathbb{Z}_4$ symmetry, we omits the statements of the combinatorics of $s_{23}$, $s_{32}$, $s_{34}$, $s_{43}$,  $s_{41}$, $s_{14}$ and  $s_{24}$.

\begin{figure}
	\begin{center}
		\begin{tikzpicture}
		\node at (0,0) {\includegraphics[width=10cm,height=10cm]{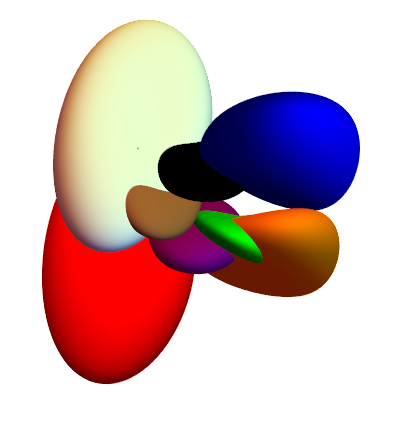}};

		
		\node at (4.2,2.1){\Large $B_{23}$};
		
		\node at (3.6,-1.1){\Large $B_{32}$};
		\node at (0.5,-0.4){\Large $B_{34}$};
		\node at (-2,-3.1){\Large $B_{14}$};

		\node at (-1.1,0.0){\Large $B_{43}$};
		\node at (-3.1,2.0){\Large $B_{41}$};

		\node at (0,-0.9){\Large $B_{13}$};

		\end{tikzpicture}
	\end{center}
	\caption{A realistic view of the boundary of  Dirichlet domain of $\rho_{\frac{5 \pi}{6}}(K) <\mathbf{PU}(2,1)$. For example,  the sphere labeled by  $B_{41}$ is in fact the spinal sphere  $B_{41} \cap \partial {\bf H}^{2}_{\mathbb C}$. The other labels have similar meanings. In this figure, we can not see $B_{21}$ and $B_{12}$, and the black one is $B_{24}$. The  brown abd green spheres labeled by $B_{43}$ and $B_{34}$ are  tangent at a point $p_{34}$, which is disjoint from the purple sphere labeled by $B_{13}$. }
	\label{figure:22infty2dimdirichlet}
\end{figure}
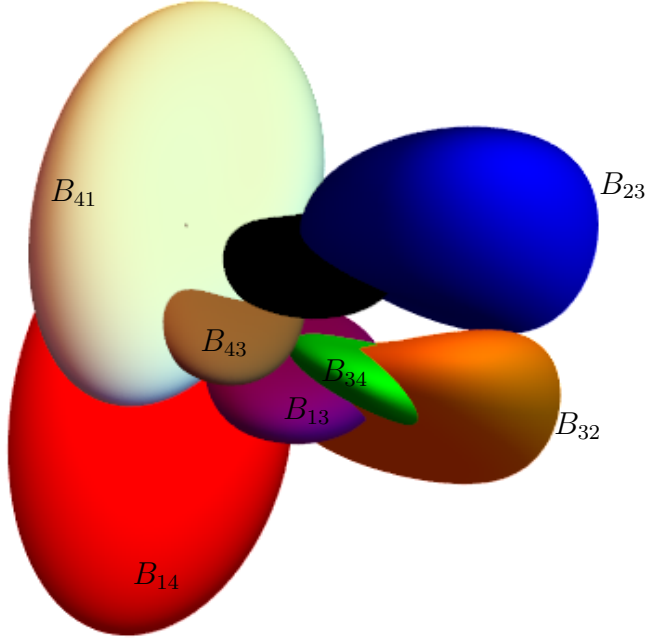

	\begin{remark}From Figures 	\ref{figure:22infty2dimdirichlet} and 	\ref{figure:22infty2dimdirichletabstract}, we strongly believe that $\partial {\bf H}^{2}_{\mathbb C} \cap D_{R}$ is a genus three handlebody. In other words, if we denote by $\mathcal{H}_{ij}$ the half space of ${\bf H}^2_{\mathbb{C}}-B_{ij}$ which does not contain the point $p_0$ for $I_{i}I_{j} \in R$ (so  $\mathcal{H}_{ij}$ contains $p_{ij}$). $\overline{ \mathcal{H}_{ij}}$ is the closure of  $\mathcal{H}_{ij}$ in $\overline{{\bf H}^2_{\mathbb{C}}}$. 
			Then $$(\cup _{I_{ij} \in R} \overline{ \mathcal{H}_{ij}}) \cap \partial {\bf H}^2_{\mathbb{C}}$$
		is a union of 3-balls in $\partial {\bf H}^2_{\mathbb{C}}$ (some of them are tangent), so it is a handlebody. This handlebody is the union of all the 3-balls bounded by spheres in Figure 	\ref{figure:22infty2dimdirichlet}.
		We have   $\partial {\bf H}^{2}_{\mathbb C} \cap D_{R}$  is the complement of $(\cup _{I_{ij} \in R} \overline{ \mathcal{H}_{ij}}) \cap \partial {\bf H}^2_{\mathbb{C}}$ in $ \partial {\bf H}^2_{\mathbb{C}}$, which is the region outside all of the spheres in Figure 	\ref{figure:22infty2dimdirichlet}. We guess that   $(\cup _{I_{ij} \in R} \overline{ \mathcal{H}_{ij}}) \cap \partial {\bf H}^2_{\mathbb{C}}$ is an unknotted genus three handlebody. In this paper  we do not  care it, and we do not show this rigorously. 
	\end{remark}


\begin{figure}
	\begin{center}
		\begin{tikzpicture}
		\node at (0,0) {\includegraphics[width=10cm,height=10cm]{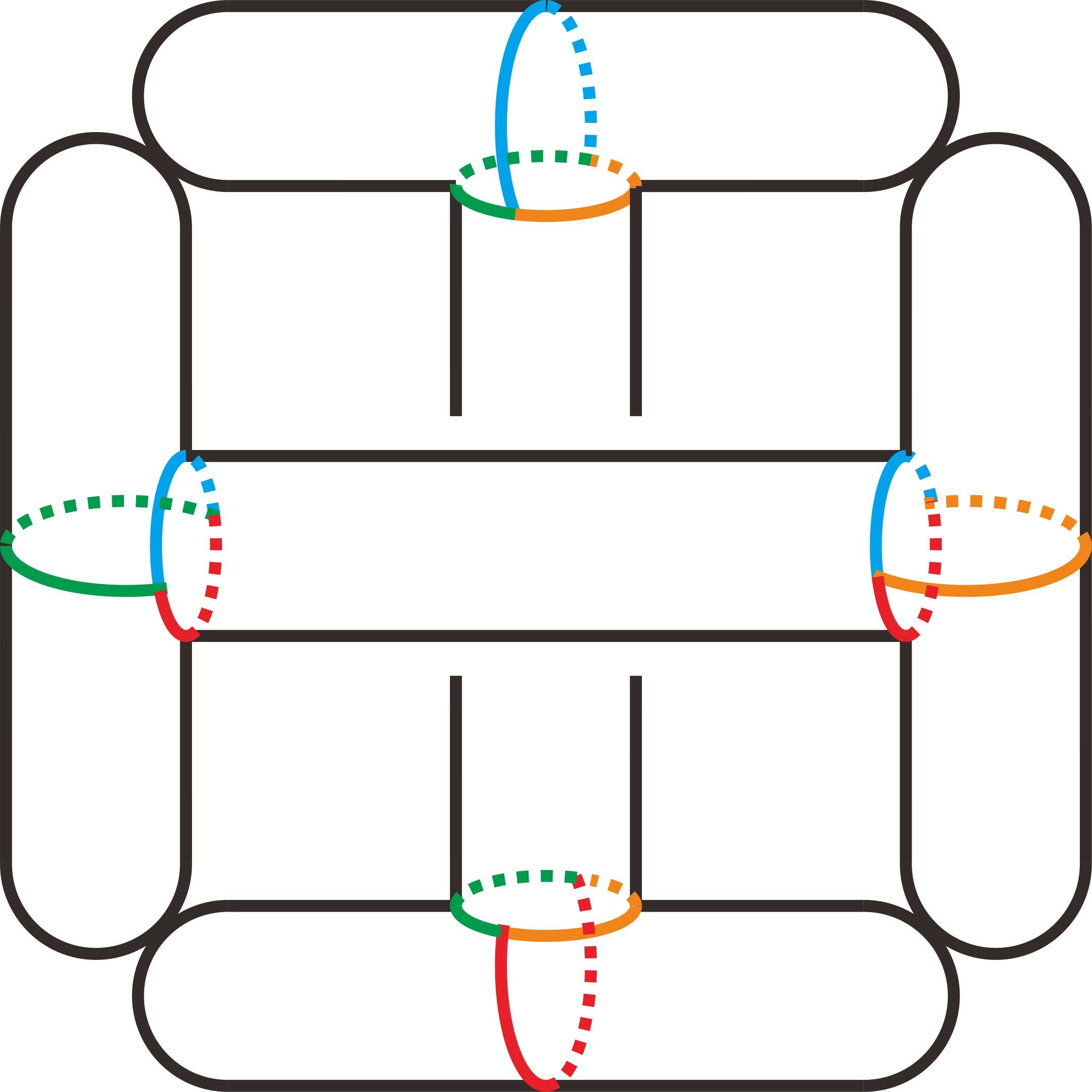}};

		\node at (4.1,2.0){\Large $B_{21}$};
		
		\node at (4.1,-2.0){\Large $B_{23}$};
		
		\node at (2,-4.1){\Large $B_{32}$};
		\node at (-2,-4.1){\Large $B_{34}$};
		\node at (-2,4.1){\Large $B_{14}$};
		\node at (2,4.1){\Large $B_{12}$};

		\node at (-4.1,-2.0){\Large $B_{43}$};
		\node at (-4.1,2.0){\Large $B_{41}$};

		\node at (-0,0.0){\Large $B_{24}$};
		\node at (0,2.0){\Large $B_{13}$};

		\end{tikzpicture}
	\end{center}
	\caption{The abstract picture of the boundary of  Dirichlet domain of $\rho_{\frac{5 \pi}{6}}(K) <\mathbf{PU}(2,1)$. What we draw is in fact $D_{R} \cap \partial {\bf H}^{2}_{\mathbb C}$. Here the once-punctured disk labeled by $B_{ij}$ for $j=i\pm 1$ mod 4  is in fact $B_{ij} \cap D_{R} \cap \partial {\bf H}^{2}_{\mathbb C}$. The annulus labeled by $B_{13}$ (and $B_{24}$ respectively)
		is in fact $B_{13} \cap D_{R}\cap \partial {\bf H}^{2}_{\mathbb C}$ (and $B_{24}\cap D_{R} \cap \partial {\bf H}^{2}_{\mathbb C}$ respectively). 
		}
	\label{figure:22infty2dimdirichletabstract}
\end{figure}
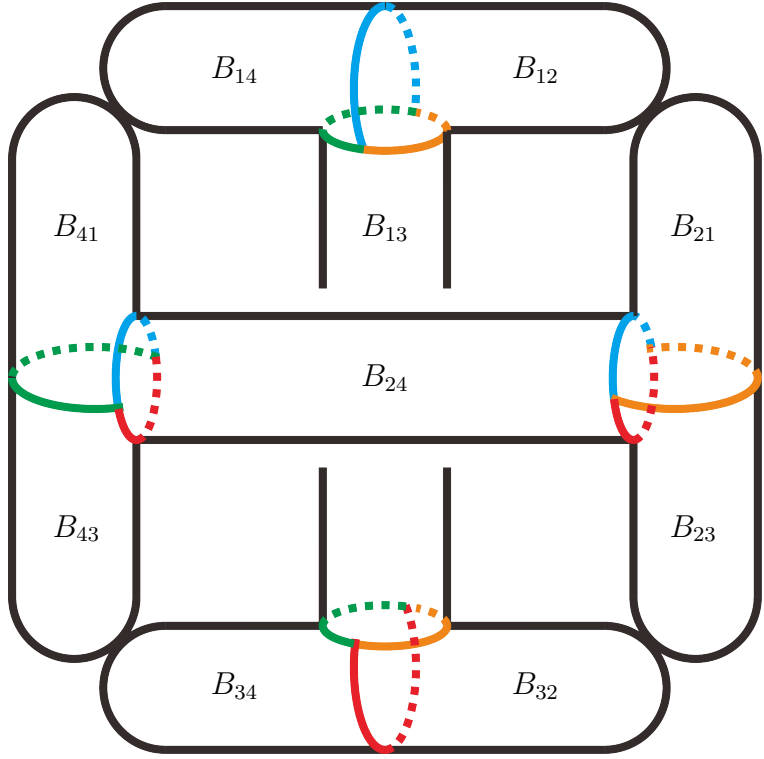
 
\subsection{Using the  Poincar\'e polyhedron theorem}\label{subsec:poincare2dim}
With the  preliminaries in Subsection \ref{subsec:2dimintersection}, we have the  side pairing maps of $D_{R}$ as follows:
\begin{itemize}
	
	\item   $I_2I_1: s_{12} \rightarrow s_{21}$;
	
		\item   $I_3I_2: s_{23} \rightarrow s_{32}$;
		
			\item   $I_4I_3: s_{34} \rightarrow s_{43}$;
			
				\item   $I_1I_4: s_{41} \rightarrow s_{14}$;
				
		\item   $I_1I_3: s_{13} \rightarrow s_{13}$;
		
			\item   $I_2I_4: s_{24} \rightarrow s_{24}$.
	
\end{itemize}

We have 

\begin{prop}\label{prop:sidepairI2I1}
	
	$I_2I_1$  is a homeomorphism from  $s_{12}$ to $s_{21}$:
	\begin{enumerate}
		
		\item \label{item:I2I11214}

		$I_2I_1$ sends the ridge $s_{12} \cap s_{14}$
		to  the ridge $s_{21} \cap s_{24}$;	
		\item   \label{item:I2I11213} $I_2I_1$ sends the ridge $s_{12} \cap s_{13}$
		to  the ridge $s_{21} \cap s_{23}$.			
	\end{enumerate}

\end{prop}
\begin{proof}

The ridge $s_{12} \cap s_{14}$	is
defined by the triple equality
$$|\langle {\bf z}, {\bf p_{0}}\rangle| =|\langle {\bf z}, p_{12}\rangle|=|\langle {\bf z},{\bf p_{14}}\rangle|$$
with ${\bf z} \in \mathbb{C}^{3,1}$.
From 	$I_2I_1$'s action  on the set 	$$\{p_0, p_{12},  p_{14}\},$$ we get the set
$$\{p_{21}, p_0,  p_{24}\}.$$
So  $I_2I_1$ maps  $s_{12} \cap s_{14}$ to  $s_{21} \cap s_{24}$.  

The proof of (\ref{item:I2I11213}) of Proposition \ref{prop:sidepairI2I1} is similar.
\end{proof}

Similarly, we have 
\begin{prop}\label{prop:I1I3}
	
	The side pairing map  $I_1I_3$    is a self-homeomorphism of $s_{13}$:
	\begin{enumerate}
		
		\item \label{item:Cridge1}  
		
			$I_1I_3$ exchanges  the ridges $s_{13} \cap s_{12}$
		and    $s_{13} \cap s_{32}$;
		\item \label{item:Cridge1} 
		$I_1I_3$ exchanges  the ridges $s_{13} \cap s_{14}$
		and    $s_{13} \cap s_{34}$.			
	\end{enumerate}

\end{prop}

\begin{prop}\label{prop:sidepairI3I2}
	
	$I_3I_2$  is a homeomorphism from  $s_{23}$ to $s_{32}$:
	\begin{enumerate}
		
		\item \label{item:I2I11214}

		$I_3I_2$ sends the ridge $s_{23} \cap s_{24}$
		to  the ridge $s_{32} \cap s_{34}$;	
		\item   \label{item:I2I11213} $I_3I_2$ sends the ridge $s_{21} \cap s_{23}$
		to  the ridge $s_{31} \cap s_{32}$.			
	\end{enumerate}

\end{prop}

\begin{prop}\label{prop:sidepairI4I3}
	
	$I_4I_3$  is a homeomorphism from  $s_{34}$ to $s_{43}$:
	\begin{enumerate}
		
		\item \label{item:I2I11214}

		$I_4I_3$ sends the ridge $s_{34} \cap s_{32}$
		to  the ridge $s_{24} \cap s_{43}$;	
		\item   \label{item:I2I11213} $I_4I_3$ sends the ridge $s_{34} \cap s_{13}$
		to the ridge $s_{43} \cap s_{41}$.			
	\end{enumerate}

\end{prop}

\begin{prop}\label{prop:sidepairI4I3}
	
	$I_1I_4$  is a homeomorphism from  $s_{41}$ to $s_{14}$:
	\begin{enumerate}
		
		\item \label{item:I2I11214}

		$I_1I_4$ sends the ridge $s_{41} \cap s_{24}$
		to  the ridge $s_{14} \cap s_{12}$;	
		\item   \label{item:I2I11213} $I_1I_4$ sends the ridge $s_{41} \cap s_{43}$
		to  the ridge $s_{14} \cap s_{13}$.			
	\end{enumerate}

\end{prop}

\begin{prop}\label{prop:I2I4}
	
	The side pairing map  $I_2I_4$    is a self-homeomorphism of $s_{24}$:
	\begin{enumerate}
		
		\item \label{item:Cridge1}  
		
		$I_2I_4$ exchanges  the ridges $s_{24} \cap s_{21}$
		and    $s_{24} \cap s_{41}$;
		\item \label{item:Cridge1} 
		$I_2I_4$ exchanges  the ridges $s_{24} \cap s_{23}$
		and    $s_{24} \cap s_{43}$.			
	\end{enumerate}

\end{prop}

{\bf Proof of Theorem  \ref{thm:2dimdirichlet}.} After above propositions,  we prove the tessellation around the sides and ridges of the partial Dirichlet domain $D_{R}$.


First,  $I_1I_3$  is a self-homeomorphism of $s_{13}$, and  $I_1I_3$
exchanges the two components  of ${\bf H}^{2}_{\mathbb C}-B_{13}$. Then 
$D_{R}$ and $I_1I_3(D_{R})$
have disjoint interiors,  and they together cover a neighborhood  of
each point in the interior of the side $s_{13}$.  The cases of the other 3-sides are similar.

Secondly,  we consider tessellations about ridges. Recall that $A_1=I_1I_2$, $A_2=I_2I_3$, $A_3=I_3I_4$ and $A_4=I_4I_1$. 

(1).  For the ridge  $s_{14}\cap s_{12}$,  the ridge circle is $$s_{14}\cap s_{12} \xrightarrow{A^{-1}_1}s_{21}\cap s_{24}\xrightarrow{(A_2A_3)^{-1}}s_{24}\cap s_{41}\xrightarrow{A^{-1}_4}s_{14}\cap s_{12}.$$
Which gives the relation $A^{-1}_4 \cdot A^{-1}_3A^{-1}_2 \cdot A^{-1}_1=id$.  By a standard argument as in \cite{ParkerWill:2017}, we have $D_{R} \cup A_1(D_{R}) \cup A^{-1}_4(D_{R})$ covers a small neighborhood of  $s_{14}\cap s_{12}$.

(2).   For the ridge  $s_{13}\cap s_{14}$,  the ridge circle is $$s_{13}\cap s_{14} \xrightarrow{A_4}s_{41}\cap s_{43}\xrightarrow{A_3}s_{34}\cap s_{31}\xrightarrow{A_1A_2}s_{13}\cap s_{14}.$$
Which gives the relation $A_1A_2 \cdot A_3\cdot A_4=id$.   By a standard argument as above $D_{R} \cup A^{-1}_4(D_{R}) \cup (A_1A_2)^{-1}(D_{R})$  covers a small neighborhood of  $s_{13}\cap s_{14}$.

(3).   For the ridge  $s_{13}\cap s_{12}$,  the ridge circle is $$s_{13}\cap s_{12} \xrightarrow{A^{-1}_1}s_{21}\cap s_{23}\xrightarrow{A^{-1}_2}s_{32}\cap s_{31}\xrightarrow{A_1A_2}s_{13}\cap s_{12}.$$
Which gives the relation $A_1A_2\cdot A^{-1}_2\cdot A^{-1}_1=id$.  We have $D_{R} \cup A_1(D_{R}) \cup A_1A_2(D_{R})$  covers a small neighborhood of  $s_{13}\cap s_{12}$.

(4).   For the ridge  $s_{24}\cap s_{23}$,  the ridge circle is $$s_{24}\cap s_{23} \xrightarrow{A^{-1}_2}s_{32}\cap s_{34}\xrightarrow{A^{-1}_3}s_{43}\cap s_{42}\xrightarrow{A_2A_3}s_{24}\cap s_{23}.$$
Which gives the relation $A_2A_3\cdot A^{-1}_3\cdot A^{-1}_2=id$. Then   $D_{R} \cup A_2(D_{R}) \cup A_2A_3(D_{R})$  covers a small neighborhood of  $s_{13}\cap s_{12}$.

In Figure 	\ref{figure:22infty2dimdirichletabstract}, we take the three ridges in the same class as in $s_{14}\cap s_{12}$ by cyan colors (more precisely, the intersections of these ridges with  $\partial {\bf H}^2_{\mathbb C}$, so we have a set of  arcs). Similarly, we color other ridge classes by orange, red and green colors.

Since the identification  around $q_{12}\in \partial {\bf H}^2_{\mathbb C}$ is given by $A_1$, which is unipotent.   So  the identification space given by the side-pairing maps is complete at the ideal point $q_{12}$. By $\langle J \rangle$-symmetry,   the identification space of $D_{R}$ by these side pairing maps is complete.

By  Poincar\'e polyhedron theorem, the partial Dirichlet domain $D_{R}$ is in fact the Dirichlet domain
of $\rho_{\frac{5 \pi}{6}}(K) <\mathbf{PU}(2,1)$.
Now we have the presentation 
$$\rho_{\frac{5 \pi}{6}}(K)=\left\langle  A_1, A_2, A_3, A_4 \Bigg| \begin{matrix}     (A_1A_2)^2=   (A_2A_3)^2 =  A_1A_2  A_3A_4=id
\end{matrix}\right\rangle.$$
Note that with the relations  $$(A_1A_2)^2= id,~~~(A_2A_3)^2 =id,~~~  A_1A_2A_3A_4=id,$$ it is trivial to get $$(A_3A_4)^2= id, ~~~(A_4A_1)^2 =id.$$ So the relations above are also $\mathbb{Z}_4$-invariant, we have a discrete and  faithful representation of $K$ into $\mathbf{PU}(2,1)$. 
This ends the proof of Theorem \ref{thm:2dimdirichlet},  and  so the proof of  Theorem \ref{thm:complex3dim} when $\theta=\frac{5 \pi}{6}$.

\section{Dirichlet domain  of $\rho_{\pi}(K) <\mathbf{PO}(3,1)$  in  ${\bf H}^3_{\mathbb R}$} \label{sec:Dirich3dimreal}

In this section, we let $\theta=\pi$. Then   $\rho_{\pi}(K)$ preserves a totally geodesic ${\bf H}^{3}_{\mathbb R} \hookrightarrow {\bf H}^{3}_{\mathbb C}$ invariant. Even through the discreteness of $\rho_{\pi}(K)$ is trivial, but we consider the Dirichlet domain  of $\rho_{\pi}(K) <\mathbf{PO}(3,1)$  in  ${\bf H}^3_{\mathbb R}$, which is the baby case  for Section \ref{sec:complex3dim}.


\subsection{$\rho_{\pi}(K)$-invariant  totally geodesic ${\bf H}^{3}_{\mathbb R} \hookrightarrow {\bf H}^{3}_{\mathbb C}$}\label{subsec:totallygeodesic}
Since the  $\rho_{\pi}(K)$-invariant  totally geodesic ${\bf H}^{3}_{\mathbb R} \hookrightarrow {\bf H}^{3}_{\mathbb C}$ is not the obvious one in ${\bf H}^{3}_{\mathbb C}$, we first describe it. 
When $\theta = \pi$, as the notations in Subsection \ref{subsec:gram},  we  have 
\begin{equation}\label{n1n2real}
n_1=\left[\begin{matrix} \frac{\sqrt{3}}{2}, \frac{1}{2}, \frac{1}{2}, \frac{1}{2} \end{matrix}\right]^{t}
,~~n_2= \left[\begin{matrix} -\frac{\sqrt{3}}{2}, \frac{\rm{i}}{2} , -\frac{\rm{i}}{2},\frac{1}{2} \end{matrix}\right]^{t},
\end{equation}
and
\begin{equation}\label{n1n2real}
n_3=\left[\begin{matrix} \frac{\sqrt{3}}{2}, -\frac{1}{2}, -\frac{1}{2}, \frac{1}{2} \end{matrix}\right]^{t}
,~~n_4= \left[\begin{matrix} -\frac{\sqrt{3}}{2}, -\frac{\rm{i}}{2}, \frac{\rm{i}}{2},\frac{1}{2} \end{matrix}\right]^{t}.
\end{equation}
We fix lifts  $\bf{n_{i}}$ of  $n_{i}$ as vectors in $\mathbb{C}^{3,1}$,  such that the entries of  $\bf{n_{i}}$ are just the same as  entries of  $n_{i}$ above. 
Let $\mathcal{L}$ be the intersection of the  real span of $\bf{ n_1},\bf{ n_2},\bf{ n_3},\bf{ n_4}$ and ${\bf H}^{3}_{\mathbb C}$, that is $\mathcal{L}$ is the set 
$$\left\{[{\bf x}=x_1 {\bf n_1}+x_2 {\bf n_2}+x_3 {\bf n_3} +x_4 {\bf n_4}] \in {\bf H}^3_{\mathbb C}~~|~~x_i \in \mathbb{R}\right\}.$$
Since $J(n_i)=n_{i+1}$ mod $4$, for 
$$x=x_1n_1+x_2n_2+x_3n_3+x_4n_4 \in \mathcal{L},$$ $$J(x)=x_4n_1+x_1n_2+x_2n_3+x_3n_4 \in \mathcal{L}.$$
So 
\begin{equation}\label{matrix:matrixrealI1}
J_{\mathcal{L}}=
\begin{pmatrix}
0& 0& 0&1\\
1& 0 & 0& 0\\
0&1&0 &0\\
0 & 0&1&0\\
\end{pmatrix}\end{equation}
is the matrix representation of $J$ on $\mathcal{L}$ with basis $\{n_1, n_2, n_3, n_4\}$.
The fix point of $J$ in $\mathcal{L}$  is $$x_{0}=n_1+n_2+n_3+n_4,$$ that is, we may take $x_1=x_2=x_3=x_4=\frac{1}{2}$.

Now
\begin{equation}
I_1=\left(\begin{array}{cccc}
\frac{1}{2}& \frac{\sqrt{3}}{2}&  \frac{\sqrt{3}}{2}&- \frac{\sqrt{3}}{2}\\[ 3 pt]
\frac{\sqrt{3}}{2}& -\frac{1}{2} & \frac{1}{2}& -\frac{1}{2}\\[ 3 pt]
\frac{\sqrt{3}}{2}&\frac{1}{2}&-\frac{1}{2} &-\frac{1}{2}\\[ 3 pt]
\frac{\sqrt{3}}{2} & \frac{1}{2}&\frac{1}{2}&-\frac{3}{2}\\
\end{array}\right)\end{equation} as given  in Subsection \ref{subsec:gram}.
Consider $$I_1(x_1n_1+x_2n_2+x_3n_3+x_4n_4)=y_1n_1+y_2n_2+y_3n_3+y_4n_4,$$ we have the matrix representation of $I_1$-action  on $\mathcal{L}$ with basis  $\{n_1, n_2, n_3, n_4\}$ is 
\begin{equation}\label{matrix:matrixrealI1}
I_{\mathcal{L},1}=\begin{pmatrix}
1& -2&  0 &-2\\
0& -1 & 0& 0\\
0&0&-1 &0\\
0& 0&0&-1\\
\end{pmatrix}.\end{equation}
Then $I_{\mathcal{L},i}=J_{\mathcal{L}} \cdot I_{\mathcal{L},i} \cdot  J^{-1}_{\mathcal{L}}$ for $i=2,3,4$ is the matrix representation of $I_{i}$-action  on $\mathcal{L}$ with base  $\{n_1, n_2, n_3, n_4\}$.

We transform   $\mathcal{L}$ into the Klein model of  ${\bf H}^{3}_{\mathbb R}$, we also transform $J_{\mathcal{L}}$ and $I_{\mathcal{L},i}$ into matrices in the Lie group $\mathbf{PO}(3,1)$.
Now the quadratic form  with respect to the basis  $\{{ \bf n_1}, { \bf n_2}, { \bf n_3}, { \bf n_4}\}$ is given by 
	\begin{equation}\label{matrix:HL}H_{\mathcal{L}}=(
\langle  { \bf n_{i}}, { \bf n_{j}}\rangle)_{1\leq i, j \leq 4}=\begin{pmatrix}
1& -1& 0&-1\\
-1& 1 & -1& 0\\
0&-1&1 &-1\\
-1 & 0&-1&1\\
\end{pmatrix}.\end{equation}
The quadratic form of the Klein model of  ${\bf H}^{3}_{\mathbb R}$  is given by  $H_{\mathcal{R}}=H=diag \{1,1,1,-1\}$.
Consider the  matrix 	\begin{equation}\label{matrix:C} C=\begin{pmatrix}
1& 0& 0&0\\
0& 0 & 1& 0\\
-\frac{1}{\sqrt{3}}&-\frac{2}{\sqrt{3}}&-\frac{1}{\sqrt{3}} &\frac{1}{\sqrt{3}}\\
1 & 1&1&0\\
\end{pmatrix}\end{equation}
with real entries. 
Then $\det( H_{\mathcal{L}})=-3$, $\det(C)=\frac{1}{\sqrt{3}}$, and  $C \cdot H_{\mathcal{L}}\cdot C^{t}=H$.

We denote by $$J_{k}=(C^{t})^{-1} \cdot J_{\mathcal{L}} \cdot C^{t},$$ and 
$$I_{k,i}=(C^{t})^{-1} \cdot I_{\mathcal{L},i} \cdot C^{t}$$
for $i=1,2,3,4$.
Then  $$J_{k} \cdot H \cdot J^{*}_{k}=H$$ and $$I_{k,i} \cdot H \cdot I^{*}_{k,i}=H$$ for $i=1,2,3,4$.
So now $J_{k}$ and $I_{k,i}$ for $i=1,2,3,4$ are the matrix presentations in  the Klein model of ${\bf H}^{3}_{\mathbb R}$ for 
$J$ and $I_{i}$ for $i=1,2,3,4$ acting on $\mathcal{L}$. It is simple to get the matrices of $J_{k}$ and  $I_{k,i}$ for $i=1,2,3,4$ from the matrix $C$, $J_{\mathcal{L}}$ and $I_{\mathcal{L},i}$, we do not write down them explicitly here.

\subsection{The partial  Dirichlet domain $D_{R}$ in ${\bf H}^{3}_{\mathbb R}$}\label{subsec:Dirichletreal}

From  the transformation in Subsection \ref{subsec:totallygeodesic}, in this subsection, we always consider  the Klein model of  ${\bf H}^{3}_{\mathbb R}$. 
Consider the partial  Dirichlet domain $D_{R}$ with center point  $$p_0=(C^{t})^{-1}([\frac{1}{2}, \frac{1}{2},\frac{1}{2},\frac{1}{2}]^{t})=[-\frac{1}{2},-\frac{1}{2},\frac{\sqrt{3}}{2},\frac{3}{2}]^{t}$$ in   ${\bf H}^{3}_{\mathbb R}$,
here $p_0$ is the fixed point of $J_{k}$  in the Klein model of ${\bf H}^{3}_{\mathbb R}$. Where $$R=\left\{(I_{k,1}I_{k,2})^{\pm 1},(I_{k,2}I_{k,3})^{\pm 1},(I_{k,3}I_{k,4})^{\pm 1},(I_{k,4}I_{k,1})^{\pm 1},I_{k,1}I_{k,3},I_{k,2}I_{k,4}\right\}$$ is a subset of the group  $$\langle I_{k,1}I_{k,2}, I_{k,2}I_{k,3}, I_{k,3}I_{k,4}, I_{k,4}I_{k,1}\rangle.$$

As usual, we denote by $p_{ij}=I_{k,i}I_{k,j}(p_0) \in {\bf H}^3_{\mathbb R}$ and $B_{ij}=B(p_0, p_{ij})$   the bisector in $ {\bf H}^3_{\mathbb R}$ with respect to the two points $p_0$ and $p_{ij}$ for certain $i,j \in \{1,2,3,4\}$. Then $$p_{12}= I_{k,1}I_{k,2}(p_0)=[\frac{3}{2},-\frac{3}{2},\frac{\sqrt{3}}{2},\frac{5}{2}]^{t},$$ \\ [-2 ex]
$$p_{21}=I_{k,2}I_{k,1}(p_0)=[-\frac{5}{2},-\frac{7}{2},\frac{\sqrt{3}}{2},\frac{9}{2}]^{t},$$ \\ [-2 ex]
$$p_{23}=I_{k,2}I_{k,3}(p_0)=[-\frac{7}{2},-\frac{5}{2},\frac{\sqrt{3}}{2},\frac{9}{2}]^{t},$$ \\ [-2 ex]
$$p_{32}= I_{k,3}I_{k,2}(p_0)=[-\frac{3}{2},\frac{3}{2},\frac{\sqrt{3}}{2},\frac{5}{2}]^{t},$$ \\ [-2 ex]
$$p_{34}= I_{k,3}I_{k,4}(p_0)=[-\frac{3}{2},\frac{3}{2},\frac{3\sqrt{3}}{2},\frac{7}{2}]^{t},$$\\ [-2 ex]
$$p_{43}= I_{k,4}I_{k,3}(p_0)=[-\frac{7}{2},-\frac{5}{2},\frac{7\sqrt{3}}{2},\frac{15}{2}]^{t},$$\\ [-2 ex]
$$p_{41}= I_{k,4}I_{k,1}(p_0)=[-\frac{5}{2},-\frac{7}{2},\frac{7\sqrt{3}}{2},\frac{15}{2}]^{t},$$\\ [-2 ex]
$$p_{14}= I_{k,1}I_{k,4}(p_0)=[\frac{3}{2},-\frac{3}{2},\frac{3\sqrt{3}}{2},\frac{7}{2}]^{t},$$\\ [-2 ex]
$$p_{13}= I_{k,1}I_{k,3}(p_0)=[\frac{1}{2},\frac{1}{2},\frac{\sqrt{3}}{2},\frac{3}{2}]^{t},$$\\ [-2 ex]
$$p_{24}= I_{k,2}I_{k,4}(p_0)=[-\frac{5}{2},-\frac{5}{2},\frac{3\sqrt{3}}{2},\frac{9}{2}]^{t}.$$

\begin{figure}
	\begin{center}
		\begin{tikzpicture}
		\node at (0,0) {\includegraphics[width=8cm,height=8cm]{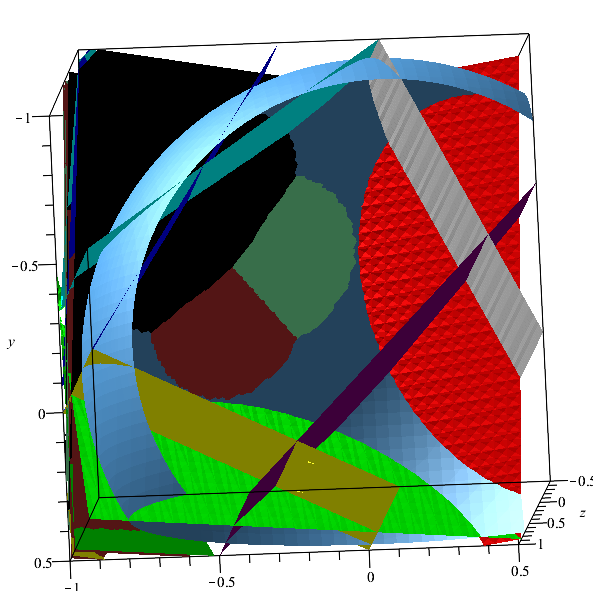}};

		\end{tikzpicture}
	\end{center}
	\caption{Part of a  realistic view of Dirichlet domain of $\rho_{\pi}(K) <\mathbf{PO}(3,1)$ from a point view inside $D_{R}$. The steelblue sphere is the ideal boundary of  ${\bf H}^{3}_{\mathbb R}$.}
	\label{figure:22inftyrealdirichlet}
\end{figure}

We fix lifts $\bf{p_0}$ and $\bf{p_{ij}}$ of $p_0$ and $p_{ij}$ as vectors in $\mathbb{R}^{3,1}$,  such that the entries of  $\bf{p_0}$ and $\bf{p_{ij}}$ are just the same as  entries of $p_0$ and $p_{ij}$ above.  We have $$\langle {\bf p_0}, {\bf p_0} \rangle =\langle {\bf p_{ij}},{\bf p_{ij}} \rangle=-1$$
for ${\bf p_{ij}}$ above.


  Now for $p=[x,y,z,1]^{t} \in {\bf H}^{3}_{\mathbb R}$, if $p$ lies in the bisector  $B_{12}$ of $I_{k,1}I_{k,2}$, 
that is, the distances of $p$ to $p_0$ and $p_{12}$ equal, then $$|{ \bf p_0}^{*} \cdot H \cdot {\bf p}|=|{\bf p_{12}}^{*} \cdot H \cdot { \bf p}|.$$
So the bisector $B_{12}$ is given by all $[x,y,z,1]^{t}$ with the conditions  $$|-\frac{1}{2}\cdot x-\frac{1}{2}\cdot y+\frac{\sqrt{3}}{2}\cdot z-\frac{3}{2} \cdot 1| = |\frac{3}{2}\cdot x-\frac{3}{2}\cdot y+\frac{\sqrt{3}}{2}\cdot z-\frac{5}{2}\cdot 1|$$
and $$x^2+y^2+z^2 <1.$$
Which is a round disk in the Klein model of ${\bf H}^{3}_{\mathbb R}$.
Similarly,  we can give all the bisectors $B_{21}$, $B_{23}$, $B_{32}$, $B_{34}$, $B_{43}$, $B_{41}$, $B_{14}$, $B_{13}$ and $B_{24}$.

Let $A_{k,i}=I_{k,i}I_{k,i+1}$ for $i=1,2,3,4$ mod 4. It is not  difficult to show the following directly (without using the over-killed methods in Section \ref{sec:complex3dim}), we omit the details. 
\begin{prop} \label{thm:3dimrealdirichlet}
	$D_{R}$ is the Dirichlet domain for  $\rho_{\pi}(K) <\mathbf{PO}(3,1)$ acting on   ${\bf H}^{3}_{\mathbb R}$ with center $p_0$. Moreover,
	the group $\rho_{\pi}(K)= \langle A_{k,1}, A_{k,2},A_{k,3},A_{k,4} \rangle$ is discrete and has a presentation 
	 $$\rho_{\pi}(K)=\left\langle A_{k,1}, A_{k,2},A_{k,3},A_{k,4}  \Bigg| \begin{array} {c}   A_{k,1}A_{k,2}A_{k,3}A_{k,4}=id,\\ [3 pt]
(A_{k,1}A_{k,2})^2=(A_{k,2}A_{k,3})^2=id,\\ [3 pt]
(A_{k,3}A_{k,4})^2=(A_{k,4}A_{k,1})^2=id
	\end{array}\right\rangle$$

\end{prop}


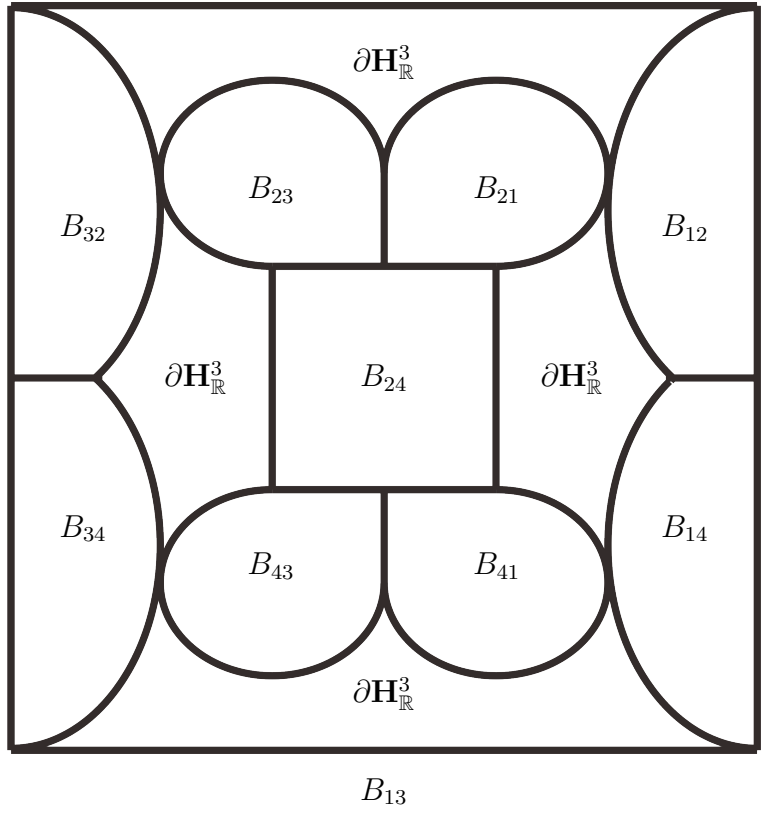
\begin{figure}
	\begin{center}
		\begin{tikzpicture}
		\node at (0,0) {\includegraphics[width=10cm,height=10cm]{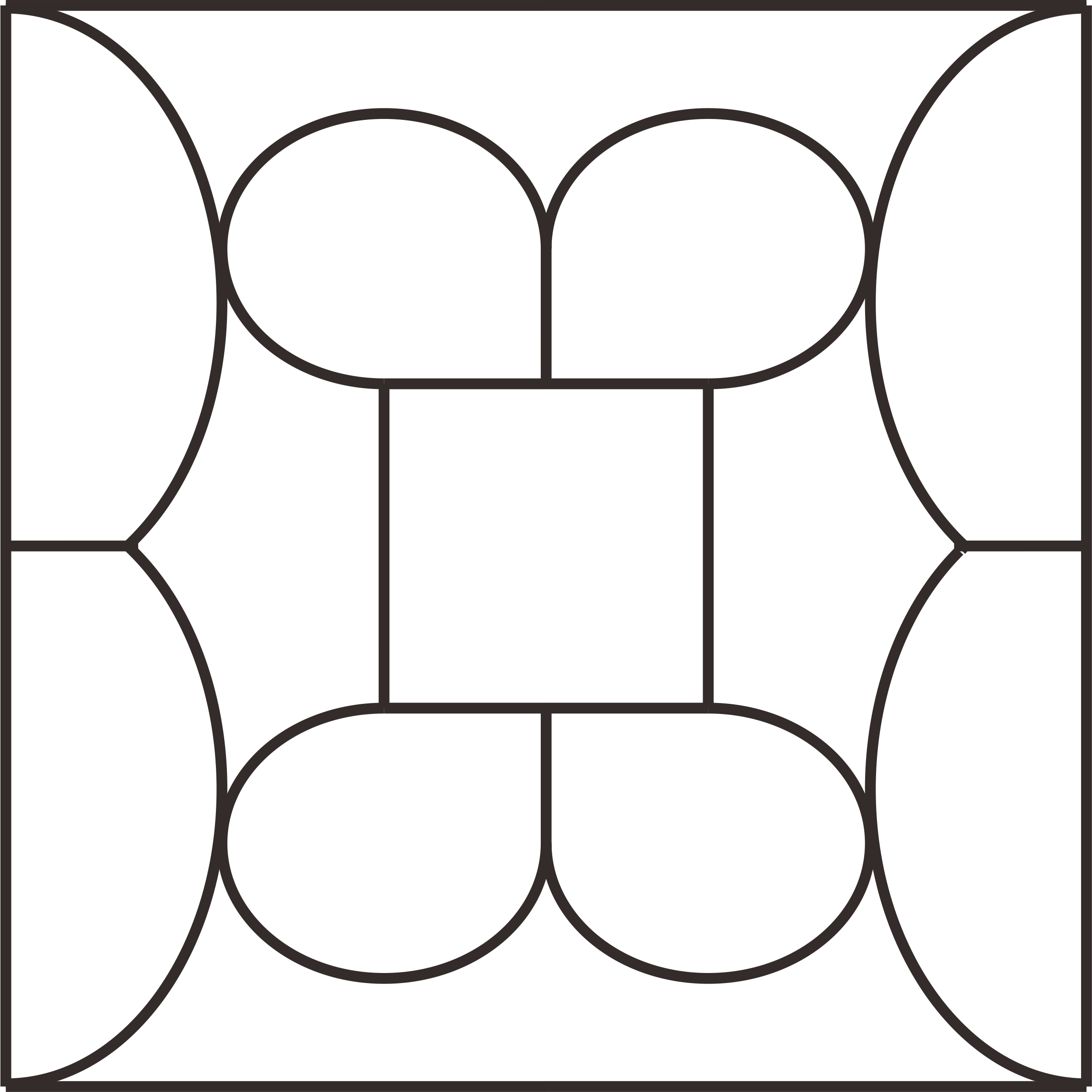}};
		
		\node at (0,4.2){\Large $\partial {\bf H}^{3}_{\mathbb R}$};
		\node at (0,-4.2){\Large $\partial {\bf H}^{3}_{\mathbb R}$};
		\node at (2.5,0.0){\Large $\partial {\bf H}^{3}_{\mathbb R}$};
		\node at (-2.5,0.0){\Large $\partial {\bf H}^{3}_{\mathbb R}$};
		
		\node at (4,2.0){\Large $B_{12}$};
		
			\node at (4,-2.0){\Large $B_{14}$};
		
			\node at (1.5,-2.5){\Large $B_{41}$};
				\node at (-1.5,-2.5){\Large $B_{43}$};
					\node at (-1.5,2.5){\Large $B_{23}$};
				\node at (1.5,2.5){\Large $B_{21}$};

					\node at (-4,-2.0){\Large $B_{34}$};
						\node at (-4,2.0){\Large $B_{32}$};

							\node at (-0,0.0){\Large $B_{24}$};
						\node at (0,-5.5){\Large $B_{13}$};

		\end{tikzpicture}
	\end{center}
	\caption{The abstract picture of the boundary of  Dirichlet domain of $\rho_{\pi}(K) <\mathbf{PO}(3,1)$ look out side from $p_0$, which is a 2-sphere with labeled disks. For example, the disk labeled by $B_{12}$ means  $B_{12}\cap D_{R}$. The out disk is $B_{13}\cap D_{ R}$. Each of the  four disks labeled by $\partial {\bf H}^{3}_{\mathbb R}$ means a disk in $\partial {\bf H}^{3}_{\mathbb R}$. }
	\label{figure:22inftyrealdirichletabstract}
\end{figure}

See Figure 	\ref{figure:22inftyrealdirichlet} for the partial  Dirichlet domain of $D_{R}<\rho_{\pi}(K)$  acting on   ${\bf H}^{3}_{\mathbb R}$. The color rule in Figure \ref{figure:22infty2dimdirichlet} is: 

\begin{itemize}
	\item 	$B_{12}$, gray;
	\item 	$B_{21}$, cyan;
	\item 	$B_{23}$, blue;
	\item 	$B_{32}$, yellow;
	\item 	$B_{34}$, green;
	\item 	$B_{43}$, brown;
	\item 	$B_{41}$, aquamarine;
	\item 	$B_{14}$, red;
	\item 	$B_{13}$, purple;
	\item 	$B_{24}$, black.
\end{itemize}

For example, the brown rectangle in Figure 	\ref{figure:22inftyrealdirichlet} is $B_{43} \cap D_{R}$, which is a part of the frontier of $D_{R}$. The aquamarine and red rectangles are $B_{41} \cap D_{R}$ and $B_{14} \cap D_{R}$ respectively.  The black hexagon is $B_{24} \cap D_{R}$. But the other $B_{ij} \cap D_{R}$  is  not so clear from this realistic  figure. 
See also Figure	\ref{figure:22inftyrealdirichletabstract} for an abstract  picture of the Dirichlet domain of $\rho_{\pi}(K) <\mathbf{PO}(3,1)$.

 Figures \ref{figure:22inftyrealdirichlet} and \ref{figure:22inftyrealdirichletabstract} in ${\bf H}^{3}_{\mathbb R}$-geometry should be compared with Figures \ref{figure:22infty2dimdirichlet} and  \ref{figure:22infty2dimdirichletabstract} in ${\bf H}^{2}_{\mathbb C}$-geometry of $\rho_{\frac{5 \pi}{6}}(K)$. In particular,  both the Dirichlet domains are given by the same set $R$. Moreover, the intersection patterns of the sides of $D_{R}$ are also the same. Which give the hint that 
the Dirichlet domain of $\rho_{\theta}(K)$ for general $\theta \in (\frac{5 \pi}{6}, \pi]$ is also  given by the set  $R$. This is what we do in Section  \ref{sec:complex3dim}.


\section{Dirichlet domain of $\rho_{\theta}(K)$ in ${\bf H}^{3}_{\mathbb C}$}\label{sec:complex3dim}

In this section, we prove Theorem  \ref{thm:complex3dim} for  all $\theta \in (\frac{5}{6},\pi]$.


\subsection{Some preliminaries}\label{subsec:3dimpreliminary}

For any $\theta \in (\frac{5}{6},\pi]$ in the moduli space $\mathcal{M}$, we denote by $p_{ij}=I_iI_j(p_0)\in {\bf H}^{3}_{\mathbb C}$,  and $B_{ij}=B(p_0,p_{ij})$  the bisector in  ${\bf H}^{3}_{\mathbb C}$  with respect to the points $p_0$ and $p_{ij}$.  We also consider the  partial Dirichlet domain $D_{R}$ for the same set $R$ as in Sections \ref{sec:Dirich2dim} and \ref{sec:Dirich3dimreal}. 
Recall $R \subset \rho_{\frac{5 \pi}{6}}(K)$  be the set of ten words in $\rho_{\theta}(K)$:
$$\left\{ (I_1I_2)^{\pm 1},~~~(I_2I_3)^{\pm 1},~~~(I_3I_4)^{\pm 1},~~~(I_4I_1)^{\pm 1},~~~ I_1I_3,~~~ I_2I_4 \right\}.$$

	\begin{remark} \label{remark:diff}The main technical  difference between the partial Dirichlet domains $D_{R}$ when  $\theta \in (\frac{5}{6},\pi]$ and $\theta= \frac{5 \pi}{6}$ is:
\begin{itemize}
 \item  $B_{12}\cap B_{34}=\emptyset$ when  $\theta$ is near to $\pi$;
 
 \item   $B_{12}\cap B_{34}\neq\emptyset$ when $\theta$ is near to $\frac{5 \pi}{6}$ (in particular, when $\theta =\frac{5 \pi}{6}$,  $B_{12}\cap B_{34}$ is non empty  as shown in Proposition \ref{prop:B12}).
 \end{itemize}  But in any case $B_{12}\cap B_{34}$ does not lie in the partial Dirichlet domain $D_{R}$, see Proposition \ref{prop:B123dim}.
 	\end{remark}

Now the fixed point of $J$ is 	$p_{0}=[0,0,0,1]^{t}$, and 

\begin{equation}\label{para:p123dim}
p_{12}=\left[
\begin{array}{c}
-\rm{e}^{-\theta\rm{i}}\sqrt{2\cos(2\theta)+1} \\  [2 ex]
\displaystyle{\frac{(-2\rm{e}^{-\theta\rm{i}}+1+\rm{i})\sqrt{-1-2\cos(\theta)+2\sin(\theta)+2\sin(2\theta)}}{2}}\\  [4 ex]
\displaystyle{\frac{(-2\rm{e}^{-\theta\rm{i}}+1-\rm{i})\sqrt{-2\sin(2\theta)-2\cos(\theta)-1-2\sin(\theta)}}{2}}\\  [4 ex]
\rm{e}^{-2\theta\rm{i}}-\rm{e}^{\theta\rm{i}}+1\\
\end{array}
\right],
\end{equation}
\begin{equation}\label{para:p213dim}
p_{21}=\left[
\begin{array}{c}
\rm{e}^{\theta\rm{i}}\sqrt{2\cos(2\theta)+1} \\  [2 ex]
\displaystyle{\frac{(-2\rm{i}\rm{e}^{\theta\rm{i}}+1+\rm{i})\sqrt{-1-2\cos(\theta)+2\sin(\theta)+2\sin(2\theta)}}{2}}\\  [4 ex]
\displaystyle{\frac{(2\rm{i}\rm{e}^{\theta\rm{i}}+1-\rm{i})\sqrt{-2\sin(2\theta)-2\cos(\theta)-1-2\sin(\theta)}}{2}}\\  [4 ex]
\rm{e}^{2\theta\rm{i}}-\rm{e}^{-\theta\rm{i}}+1\\
\end{array}
\right],
\end{equation}

\begin{equation}\label{para:p133dim}
p_{13}=\left[
\begin{array}{c}
\sqrt{2\cos(2\theta)+1} \\ [1 ex] 
0\\ 
0\\
-2\cos(\theta)\\
\end{array}
\right].
\end{equation}

From the coordinates of these points in ${\bf H}^{3}_{\mathbb C}$, it is easy to get all coordinates of $p_{i,i\pm 1}$ for $i=1,2,3,4$ mod 4 and $p_{24}$ by the $J$-action.  
We fix lifts $\bf{p_0}$ and $\bf{p_{ij}}$ of $p_0$ and $p_{ij}$ as vectors in $\mathbb{C}^{3,1}$,  such that the entries of  $\bf{p_0}$ and $\bf{p_{ij}}$ are just the same as  entries of $p_0$ and $p_{ij}$ above.  We have $$\langle {\bf p_0}, {\bf p_0} \rangle =\langle {\bf p_{ij}},{\bf p_{ij}} \rangle=-1$$
for ${\bf p_{ij}}$ above.

First we have 
\begin{lemma} \label{lemma:noncoplane} For any $\theta \in (\frac{5\pi}{6},\pi]$ in the moduli space $\mathcal{M}$,  the four vectors $\bf{p_0}$, $\bf{p_{12}}$, $\bf{p_{13}}$ and $\bf{p_{14}}$  are linearly independent  in $\mathbb{C}^{3,1}$.
\end{lemma}
\begin{proof}This is proved by a direct  calculation. Consider the matrix $S$ such that  its four columns are the vectors $p_0$, $p_{12}$, $p_{13}$ and $p_{14}$ given by (\ref{para:p123dim}), (\ref{para:p133dim}) and $J$-action respectively.
	Then $$\det(S)=(-2 \rm{i}\cos(\theta)-2\rm{i}\cos(3\theta)-2 \sin(3\theta)+\rm{i})\sqrt{4\cos^2(2\theta)-1}.$$
	Now it is easy to see $\det(S)=0$ when $\theta =\frac{5\pi}{6}$ and  $\det(S)$ is non-zero when $\theta \in (\frac{5 \pi}{6},\pi]$.
\end{proof}

Lemma  \ref{lemma:noncoplane} also explains why we should prove Theorem \ref{thm:complex3dim} when $\theta =\frac{5 \pi}{6}$  separately in Section \ref{sec:Dirich2dim}, since the proof in this section does not hold when these four vectors $\bf{p_0}$, $\bf{p_{12}}$, $\bf{p_{13}}$ and $\bf{p_{14}}$ span a three dimensional subspace of $\mathbb{C}^{3,1}$.


\begin{lemma} \label{lemma:123413coplane} For any $\theta \in (\frac{5\pi}{6},\pi]$ in the moduli space $\mathcal{M}$, the four vectors $\bf{p_0}$, $\bf{p_{12}}$, $\bf{p_{34}}$ and $\bf{p_{13}}$  are co-planar  in $\mathbb{C}^{3,1}$.
\end{lemma}
\begin{proof}This is also proved by a direct calculation. In fact we have 
	$$2 \rm{e}^{\theta\rm{i}} \cdot {\bf p_0}+{\bf p_{12}}+{\bf p_{34}}+2 \rm{e}^{-\theta\rm{i}} \cdot {\bf p_{13}}$$ is the zero vector in $\mathbb{C}^{3,1}$.
\end{proof}

\subsection{Intersection patterns of the bisectors of  the partial   Dirichlet domain $D_{R}$}\label{subsec:3dimDirichlet}

In this subsection, we will study  intersection patterns of the bisectors for $D_{R}$.  The reader may compare to  Table \ref{table:intersecton} in Section \ref{subsec:2dimDirichlet} (there is only one  difference on the intersection $B_{12} \cap B_{34}$ as noted in Remark \ref{remark:diff}).

We first consider the intersections of the bisector $B_{12}$ with other bisectors, comparing to Proposition \ref{prop:B12} in Subsection \ref{subsec:2dimintersection}, the proof here is much more involved.

\begin{prop}\label{prop:B123dim} For any $\theta \in (\frac{5\pi}{6}, \pi]$, the bisector $B_{12}$ of $I_1I_2$ have the following properties:
	
	\begin{enumerate}
		
		\item  \label{item:B12andB213dim}$B_{12}$ is tangent to   $B_{21}$;

		\item  \label{item:B12B23empty3dim}$B_{12}$ does not intersect  $B_{23}$;
			\item  \label{item:B12B32empty3dim} $B_{12}$ does not intersect   $B_{32}$;
				\item  \label{item:B12B43empty3dim}$B_{12}$ does not intersect  $B_{43}$;
					\item  \label{item:B12B41empty3dim}$B_{12}$ does not intersect   $B_{41}$;
						\item  \label{item:B12B24empty3dim}$B_{12}$ does not intersect   $B_{24}$;
		
		\item  \label{item:B12nonempty3dim} $B_{12}$ intersects $B_{34}$ in a non-empty Giraud disk  $B_{12}\cap B_{34}$ in ${\bf H}^{3}_{\mathbb C}$ when $\theta \in [\frac{5 \pi}{6}, \pi]$ but near to $\frac{5 \pi}{6}$. $B_{12}\cap B_{34}=\emptyset$  when  $\theta \in [\frac{5 \pi}{6}, \pi]$ but near to $ \pi$. When  $B_{12}\cap B_{34}$ is non-empty, $B_{12}\cap B_{34}$ lies in the component of  ${\bf H}^2_{\mathbb{C}}-B_{13}$  which does not contain the  point $p_0$. In particular, in any case  $B_{12}\cap B_{34}$ does not lie in the partial Dirichlet domain $D_{R}$.
	\end{enumerate}
	
\end{prop}	

 	\begin{proof}The proof of (\ref{item:B12andB213dim}) of Proposition \ref{prop:B123dim} runs the same line as the proof of  (\ref{item:B12andB21}) of Proposition \ref{prop:B12}.

 For the proof of 	(\ref{item:B12B23empty3dim}). We now consider the intersection $B_{12}\cap B_{23}$. It is easy to see that for any $\theta \in (\frac{5\pi}{6}, \pi]$,  the span of $\bf{p_0}$, $\bf{p_{12}}$, $\bf{p_{23}}$ is a 3-dimensional subspace  of  $\mathbb{C}^{3,1}$. The intersection of ${\bf H}^{3}_{\mathbb{C}} \subset {\bf P}^{3}_{\mathbb{C}}$ with the projection of this 3-dimensional subspace  into ${\bf P}^{3}_{\mathbb{C}}$ is denoted by $L$, then  $L$ is a totally geodesic ${\bf H}^{2}_{\mathbb{C}} \hookrightarrow {\bf H}^{3}_{\mathbb{C}}$.

 		
 		We re-denote $\bf{p_0}$, $\bf{p_{12}}$ and  $\bf{p_{23}}$  by $\bf{e_1}$, $\bf{e_2}$ and $\bf{e_3}$. Then 
 		we denote $H_{L}=H_{L}(p_0,p_{12},p_{23})$ by the matrix $({\bf e_{i}}^*\cdot H\cdot {\bf e_{j}})_{1 \leq i,j \leq 3}$, we have
 	$$H_{L}=\begin{pmatrix}
 		-1&-\rm{e}^{-2\theta\rm{i}}+\rm{e}^{\theta\rm{i}}-1&-\rm{e}^{-2\theta\rm{i}}+\rm{e}^{\theta\rm{i}}-1 \\ 
 		-\rm{e}^{2\theta\rm{i}}+\rm{e}^{-\theta\rm{i}}-1 & -1&	D\\  
 		-\rm{e}^{2\theta\rm{i}}+\rm{e}^{-\theta\rm{i}}-1& \bar{D} & -1\\
 		\end{pmatrix}. $$
 		Here $$D=2\rm{e}^{-3\theta\rm{i}}-2\rm{e}^{-2\theta\rm{i}}+2\rm{e}^{\theta\rm{i}}-2\cos(2\theta)-4,$$ and $\bar{D}$ is the  complex  conjugate of $D$.
 Now $\det(H_{L})$ is 
 $$(-4\cos(4\theta)-22\cos(2\theta)-1+16\cos(3\theta)+12\cos(\theta))\cdot (1+2\cos(\theta))^2.$$	So $H_{L}$ is the Hermitian form with signature $(2,1)$ on the subspace with the basis $\{{\bf e_1}, {\bf e_2}, {\bf e_3}\}$ when $\theta \in (\frac{5\pi}{6}, \pi]$.
 		
 	 The vector $x_1 {\bf e_1}+x_2 {\bf e_2}+x_3 {\bf e_3}$ is denoted by the vector ${\bf x}$ 	in  $\mathbb{C}^{3}$, and the vector $y_1 {\bf e_1}+y_2 {\bf e_2}+y_3 {\bf e_3}$  is denoted by the vector ${\bf y}$ 	in  $\mathbb{C}^{3}$, here 
 		\begin{equation}\nonumber
 		{\bf x}=\left(
 		\begin{array}{c}
 		x_1 \\
 		x_2 \\
 		x_3 \\
 		\end{array}
 		\right),\quad
 		{\bf y}=\left(
 		\begin{array}{c}
 		y_1\\
 		y_2\\
 		y_3 \\
 		\end{array}
 		\right)
 		\end{equation}
 		with $x_{i},y_{i} \in \mathbb{C}$.
 		So \begin{equation}\nonumber
 		{\bf E_1}=\left(
 		\begin{array}{c}
 		1 \\
 		0 \\
 		0 \\
 		\end{array}
 		\right),\quad
 		{ \bf E_2}=\left(
 		\begin{array}{c}
 		0\\
 		1\\
 		0 \\
 		\end{array}
 		\right),\quad
 		{ \bf E_3}=\left(
 		\begin{array}{c}
 		0\\
 		0\\
 		1 \\
 		\end{array}
 		\right)
 		\end{equation}
 	in  $\mathbb{C}^{3}$	present the vectors ${ \bf e_1}$, ${ \bf e_2}$ and ${ \bf e_3}$ in  $\mathbb{C}^{3,1}$.

 		Comparing to Subsection \ref{subsection:Spinalcoordinates}, we define  the   Hermitian cross-product $\boxtimes_{L}$ on the space $H_{L}$ (which is isometric to ${\bf H}^{2}_{\mathbb{C}}$) with respect to the basis $\{\bf{e_1}, \bf{e_2}, \bf{e_3}\}$  by  
 		$$\mathbf{x} \boxtimes_{L}
 		\mathbf{y}=
 		\left[\begin{matrix}
 		\mathbf{x}^*H_{L}(1,2) \cdot \mathbf{y}^*H_{L}(1,3)-\mathbf{y}^*H_{L}(1,2) \cdot \mathbf{x}^*H_{L}(1,3)\\ 
 		
 		 \mathbf{x}^*H_{L}(1,3) \cdot \mathbf{y}^*H_{L}(1,1)-\mathbf{y}^*H_{L}(1,3) \cdot \mathbf{x}^*H_{L}(1,1)\\
 		 
 		  \mathbf{x}^*H_{L}(1,1) \cdot \mathbf{y}^*H_{L}(1,2)-\mathbf{y}^*H_{L}(1,1) \cdot \mathbf{x}^*H_{L}(1,2)
 		\end{matrix}
 		\right].$$
 		Here $\mathbf{x}^*H_{L}$ is a one-by-three matrix, and $\mathbf{x}^*H_{L}(1,2)$ is the second entry of   $\mathbf{x}^*H_{L}$. The other terms have similar meanings.
 		
 		Then the intersection  $B_{12}\cap B_{23} \cap L$ is  parameterized   by $V=V(z_1,z_2)	\in \mathbb{C}^{3}$ with $\langle V,V \rangle <0$ with respect to the Hermitian form $H_{L}$.
 		Where
 		$$V={ \bf E_2}\boxtimes_{L} { \bf E_3}+z_1 \cdot { \bf E_1}\boxtimes_{L} { \bf E_3}+z_2 \cdot { \bf E_1}\boxtimes_{L} { \bf E_2}$$  and  $(z_1,z_2) =(\rm{e}^{r \rm{i}},\rm{e}^{s \rm{i}})\in \mathbb{S}^1\times \mathbb{S}^1$.
 		
 		We note that 
 		\begin{equation}\nonumber
 		E_1\boxtimes_{L}E_2=\left[
 		\begin{array}{c}
 	 D_1\\  [1 ex]
 		2\rm{i}\sin(\theta)-2\rm{i}\sin(3\theta)+2\cos(2\theta)+1+2\rm{i}\sin(2\theta)	 \\ [1 ex]
 		2\cos(3\theta)+2\cos(\theta)-2\cos(2\theta)-2 \\
 		\end{array}
 		\right],
 		\end{equation}
 		where
 		\begin{flalign} 	\nonumber &
 		D_1= -5\rm{i}\sin(2\theta)+\rm{i}\sin(3\theta)-3\rm{i}\sin(4\theta)+2\rm{i}\sin(5\theta)+4+11\cos(2\theta)
 		& \\
 		\nonumber	
 		&-3\cos(3\theta)-2\cos(5\theta)-10\cos(\theta)+3\cos(4\theta).&
 		&& \nonumber \end{flalign}
 		
 		We also have 
 		\begin{equation}\nonumber
 		E_1\boxtimes_{L}E_3=\left[
 		\begin{array}{c}
 	D_2	\\[1 ex]
 	-2\cos(3\theta)-2\cos(\theta)+2\cos(2\theta)+2 \\[1 ex]
 	2\rm{i}\sin(\theta)-2\rm{i}\sin(3\theta)+2\cos(2\theta)+1+2\rm{i}\sin(2\theta) \\
 		\end{array}
 		\right],
 		\end{equation}
 		where
 		\begin{flalign} 	\nonumber &
 		D_2= \rm{i}\sin(2\theta)+3\rm{i}\sin(3\theta)-\rm{i}\sin(4\theta)+2\rm{i}\sin(\theta)-8-7\cos(2\theta)+8\cos(\theta)
 		& \\
 		\nonumber	
 		&+7\cos(3\theta)-3\cos(4\theta),&
 		&& \nonumber \end{flalign}
 		and
 		\begin{equation}\nonumber
 		E_2\boxtimes_{L}E_3=\left[
 		\begin{array}{c}
 		4\cos(5\theta)+28\cos(3\theta)+32\cos(\theta)-14\cos(4\theta)-32\cos(2\theta)-33 \\[1 ex]
 	D_3\\[1 ex]
 	 D_4\\
 		\end{array}
 		\right],
 		\end{equation}
 	here
 	\begin{flalign} 	\nonumber &
 	D_3= \rm{i}\sin(2\theta)+3\rm{i}\sin(3\theta)-\rm{i}\sin(4\theta)+2\rm{i}\sin(\theta)+8+7\cos(2\theta)-8\cos(\theta)
 	& \\
 	\nonumber	
 	&-7\cos(3\theta)+3\cos(4\theta)&
 	&& \nonumber \end{flalign}
 and
 		\begin{flalign} 	\nonumber &
 		D_4= 5\rm{i}\sin(2\theta)-\rm{i}\sin(3\theta)+3\rm{i}\sin(4\theta)-2\rm{i}\sin(5\theta)+4+11\cos(2\theta)-3\cos(3\theta)
 		& \\
 		\nonumber	
 		&-2\cos(5\theta)-10\cos(\theta)+3\cos(4\theta).&
 		&& \nonumber \end{flalign}

 The vector 	$V \in \mathbb{C}^3$ is a three-by-one matrix,   we denote by  $V(i,1)$ the  entry in the $i$-th  row of $V$. Then  $$V(1,1) \cdot { \bf e_1}+V(2,1) \cdot { \bf e_2}+V(3,1) \cdot { \bf e_3}$$  is a vector in $\mathbb{C}^{3,1}$, we denote it  by $W$. The projection of $W$ into ${\mathbb C}{\mathbf P}^{3}$ is a point in $B_{12}\cap B_{23} \cap L$ if $W$ is a negative vector.
 		Now  $W^{*}\cdot H \cdot W$ is a very complicated term. But with Maple, $W^{*}\cdot H \cdot W$  has minimum 734.88 numerically when $r, s \in [-\pi,\pi]$ and $\theta \in [\frac{5 \pi}{6}, \pi]$ at a point when $\theta= \frac{5 \pi}{6}$. In particular,  any $W$ above is a positive vector in $\mathbb{C}^{3,1}$, so $B_{12}\cap B_{23} \cap L=\emptyset$. Then by the projection from 
 		 ${\bf H}^{3}_{\mathbb{C}}$ to $L$, we have $B_{12}\cap B_{23}$ in ${\bf H}^{3}_{\mathbb{C}}$ is the empty set. 
 		
 		We remark here the minimum of  $W^{*}\cdot H \cdot W$ when $\theta= \frac{5 \pi}{6}$ is 
 		734.88, which is different from the minimum 12.752  in  Proposition \ref{prop:B12}. The reason  is that even through the  Hermitian  from $H_{L}$ is related to $H$, but it is not the same as $H$.  This end   the proof of 	(\ref{item:B12B23empty3dim}).

  For the proof of   (\ref{item:B12B32empty3dim}), so we consider $B_{12}\cap B_{32}$.
 		For any $\theta \in (\frac{5\pi}{6}, \pi]$,  the span of ${ \bf p_0}$, ${ \bf p_{12}}$, ${ \bf p_{32}}$ is a 3-dimensional subspace  of  $\mathbb{C}^{3,1}$. The intersection of ${\bf H}^{3}_{\mathbb{C}} \subset {\bf P}^{3}_{\mathbb{C}}$ with the projection of this 3-dimensional subspace  into ${\bf P}^{3}_{\mathbb{C}}$ is denoted by $L$, then  $L$ is a totally geodesic ${\bf H}^{2}_{\mathbb{C}} \hookrightarrow {\bf H}^{3}_{\mathbb{C}}$.

 		We re-denote ${ \bf p_0}$, ${ \bf p_{12}}$ and  ${ \bf p_{32}}$  by ${ \bf e_1}$, ${ \bf e_2}$ and ${ \bf e_3}$.
 		We denote $H_{L}=H_{L}(p_0,p_{12},p_{32})$ by the matrix $({ \bf e_{i}}^* \cdot H \cdot { \bf e_{j}})_{1 \leq i,j \leq 3}$, then 
 		$$H_{L}=\begin{pmatrix}
 		-1&-\rm{e}^{-2\theta\rm{i}}+\rm{e}^{\theta\rm{i}}-1&-\rm{e}^{2\theta\rm{i}}+\rm{e}^{-\theta\rm{i}}-1\\
 		-\rm{e}^{2\theta\rm{i}}+\rm{e}^{-\theta\rm{i}}-1 & -1&	6 \cos(\theta)-4\cos(2\theta)\\
 		-\rm{e}^{-2\theta\rm{i}}+\rm{e}^{\theta\rm{i}}-1& 6 \cos(\theta)-4\cos(2\theta) & -1\\
 		\end{pmatrix}. $$
 		Now $\det(H_{L})$ is 
 		$$-(4\cos(\theta)-3)\cdot (-1+2\cos(\theta)) \cdot (4\cos(\theta)^2-2\cos(\theta)-3) \cdot (1+2\cos(\theta))^2.$$	So $H_{L}$ is the Hermitian form with signature $(2,1)$ on the subspace with the basis $\{{\bf e_1}, {\bf e_2}, {\bf e_3}\}$ when $\theta \in (\frac{5\pi}{6}, \pi]$.
 		
 	 The vector $x_1 {\bf e_1}+x_2 {\bf e_2}+x_3 {\bf e_3}$ is denoted by the vector ${\bf x}$ 	in  $\mathbb{C}^{3}$, and the vector $y_1 {\bf e_1}+y_2 {\bf e_2}+y_3 {\bf e_3}$  is denoted by the vector ${\bf y}$ 	in  $\mathbb{C}^{3}$, here 
 		\begin{equation}\nonumber
 		{\bf x}=\left(
 		\begin{array}{c}
 		x_1 \\
 		x_2 \\
 		x_3 \\
 		\end{array}
 		\right),\quad
 		{\bf y}=\left(
 		\begin{array}{c}
 		y_1\\
 		y_2\\
 		y_3 \\
 		\end{array}
 		\right)
 		\end{equation}
 		with $x_{i},y_{i} \in \mathbb{C}$.
 		So \begin{equation}\nonumber
 		{\bf E_1}=\left(
 		\begin{array}{c}
 		1 \\
 		0 \\
 		0 \\
 		\end{array}
 		\right),\quad
 		{ \bf E_2}=\left(
 		\begin{array}{c}
 		0\\
 		1\\
 		0 \\
 		\end{array}
 		\right),\quad
 		{ \bf E_3}=\left(
 		\begin{array}{c}
 		0\\
 		0\\
 		1 \\
 		\end{array}
 		\right)
 		\end{equation}
 		in  $\mathbb{C}^{3}$	present the vectors ${ \bf e_1}$, ${ \bf e_2}$ and ${ \bf e_3}$ in  $\mathbb{C}^{3,1}$.

 			%

 		Comparing to Subsection \ref{subsection:Spinalcoordinates}, we define  the   Hermitian cross-product $\boxtimes_{L}$ on the space $H_{L}={\bf H}^{2}_{\mathbb{C}}$ with respect to the basis $\{{\bf e_1}, {\bf e_2}, {\bf e_3}\}$ by
 		$$\mathbf{x} \boxtimes_{L}
 		\mathbf{y}=
 		\left[\begin{matrix}
 		\mathbf{x}^*H_{L}(1,2) \cdot \mathbf{y}^*H_{L}(1,3)-\mathbf{y}^*H_{L}(1,2) \cdot \mathbf{x}^*H_{L}(1,3)\\ \mathbf{x}^*H_{L}(1,3) \cdot \mathbf{y}^*H_{L}(1,1)-\mathbf{y}^*H_{L}(1,3) \cdot \mathbf{x}^*H_{L}(1,1)\\ \mathbf{x}^*H_{L}(1,1) \cdot \mathbf{y}^*H_{L}(1,2)-\mathbf{y}^*H_{L}(1,1) \cdot \mathbf{x}^*H_{L}(1,2)
 		\end{matrix}
 		\right].$$	
 		Then the intersection  $B_{12}\cap B_{32} \cap L$ is  parameterized by $V=V(z_1,z_2) \in \mathbb{C}^3$ with $\langle V,V \rangle <0$ with respect to the Hermitian form $H_{L}$.
 		Where
 		$$V={ \bf E_2}\boxtimes_{L} { \bf E_3}+z_1 \cdot { \bf E_1}\boxtimes_{L} { \bf E_3}+z_2 \cdot { \bf  E_1}\boxtimes_{L} {\bf E_2}$$  and  $(z_1,z_2) =(\rm{e}^{r \rm{i}},\rm{e}^{s \rm{i}})\in \mathbb{S}^1 \times \mathbb{S}^1$.
 		
 		We note that $E_1\boxtimes_{L}E_2$ is 
 		\begin{equation}\nonumber
 	\left[
 		\begin{array}{c}
 		(1+2\cos(\theta))(8\cos^2(\theta)\rm{e}^{-\theta\rm{i}}+3\rm{i}\sin(2\theta)+3\rm{i}\sin(\theta)-14\cos(\theta)^2+5\cos(\theta)) \\[1 ex]
 		3+8\cos^3(\theta)\rm{e}^{\theta\rm{i}}-10\cos^2(\theta)-\rm{i}\sin(2\theta)+2\cos(\theta)	 \\[1 ex]
 		4\cos(\theta)(\cos(\theta)-1)(1+2\cos(\theta)) \\
 		\end{array}
 		\right],
 		\end{equation}
 		 $E_1\boxtimes_{L}E_3$ is 
 		\begin{equation}\nonumber
 		\left[
 		\begin{array}{c}
 		(1+2\cos(\theta))(-8\cos^2(\theta)\rm{e}^{\theta\rm{i}}+3\rm{i}\sin(2\theta)+3\rm{i}\sin(\theta)+14\cos(\theta)^2-5\cos(\theta))\\[1 ex]
 		4\cos(\theta)(1-\cos(\theta))(1+2\cos(\theta)) \\[1 ex]
 		-3-8\cos^3(\theta)\rm{e}^{-\theta\rm{i}}+10\cos^2(\theta)-\rm{i}\sin(2\theta)-2\cos(\theta)	 \\
 		\end{array}
 		\right].
 		\end{equation}
 		and $E_2\boxtimes_{L}E_3$ is 
 		\begin{equation}\nonumber
 	=\left[
 		\begin{array}{c}
 		(5-4\cos(\theta))(1+2\cos(\theta))(8\cos^2(\theta)-6\cos(\theta)-3)
 		 \\[1 ex]
 		(1+2\cos(\theta))(8\cos^2(\theta)\rm{e}^{-\theta\rm{i}}+3\rm{i}\sin(2\theta)+3\rm{i}\sin(\theta)-14\cos(\theta)^2+5\cos(\theta))\\[1 ex]
 		(1+2\cos(\theta))(8\cos^2(\theta)\rm{e}^{\theta\rm{i}}-3\rm{i}\sin(2\theta)-3\rm{i}\sin(\theta)-14\cos(\theta)^2+5\cos(\theta)) \\
 		\end{array}
 		\right].
 		\end{equation}

 	The vector 	$V \in \mathbb{C}^3$ is a three-by-one matrix, and  $$V(1,1) \cdot { \bf e_1}+V(2,1) \cdot { \bf e_2}+V(3,1) \cdot {\bf e_3} $$  is a vector in $\mathbb{C}^{3,1}$, we  denote it  by $W$.  The projection of $W$ into ${\mathbb C}{\mathbf P}^{3}$ is a point in $B_{12}\cap B_{32} \cap L$ if $W$ is a negative vector.
 		With Maple, $W^{*}\cdot H \cdot W$ is 
 		\begin{flalign} \label{item:B12andB323dim1}
 		\left(-256\cdot (\frac{1}{2}+\cos(\theta))^3\cdot (\cos(\theta)-\frac{3}{4})\cdot (-\frac{1}{2}+\cos(\theta))\cdot (\cos^2(\theta)-\frac{\cos(\theta)}{2}-\frac{3}{4})\right) \cdot h,
 		\end{flalign} where $h$ is 
 		\begin{flalign} \nonumber &
 		21-8\cos(3\theta)+36\cos(2\theta)+4\cos(r-2\theta)+10\cos(r+2\theta)-50\cos(\theta)+14\cos(r)
 		& \\
 	\nonumber	& -14\cos(s)+2\cos(-2\theta+r-s)+2\cos(\theta+r-s)-4\cos(3\theta+r)+10\cos(-\theta+s)& \\
 	\nonumber	&+12\cos(\theta+s)-10\cos(s-2\theta)-4\cos(s+2\theta)-2\cos(-3\theta+r-s)&
 		\\
 		&-4\cos(r-s)-12\cos(-\theta+r)-10\cos(\theta+r)+4\cos(-3\theta+s).&
 		 && \nonumber \end{flalign}

\begin{figure}
\begin{center}
	\begin{tikzpicture}
	\node at (0,0) {\includegraphics[width=6cm,height=6cm]{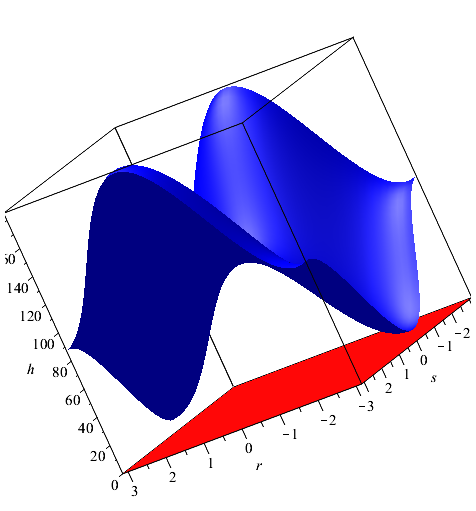}};
	\end{tikzpicture}
\end{center}
\caption{The positivity of the term $h$ in  the proof of (\ref{item:B12B32empty3dim}) of Proposition \ref{prop:B123dim}. In this figure, we take $h=h(r,s)$. The blue surface is the graph of the function $h$ when $\theta=\frac{5 \pi}{6}$. The red plane is $h=0$. This elucidates (but does not prove) $h$ is positive when $\theta=\frac{5 \pi}{6}$.  For general $\theta \in (\frac{5 \pi}{6}, \pi]$, the graph of $h=h(r,s)$ is a  surface which is similar to the surface in the figure and it is also disjoint from  the plane $h=0$.}
\label{figure:B12capB32empty}
\end{figure}

	Note that the first term of (\ref{item:B12andB323dim1}) is positive when $\theta \in [\frac{5 \pi}{6}, \pi]$. 
By Maple, $h$  
has minimum 6.5305 numerically when $r, s \in [-\pi,\pi]$ and $\theta \in [\frac{5 \pi}{6}, \pi]$ at a point when $\theta= \frac{5 \pi}{6}$. In particular, any $W$ above is a positive vector in  $\mathbb{C}^{3,1}$, so $B_{12}\cap B_{32} \cap L=\emptyset$. Then by the projection from 
${\bf H}^{3}_{\mathbb{C}}$ to $L$, we have $B_{12}\cap B_{32}$ in ${\bf H}^{3}_{\mathbb{C}}$ is the empty set. See Figure  \ref{figure:B12capB32empty} for an illustration this fact  when $\theta=\frac{5 \pi}{6}$. This ends the proof of   (\ref{item:B12B32empty3dim}).

For the proof of   (\ref{item:B12B43empty3dim}), now we consider $B_{12}\cap B_{43}$.
So for any $\theta \in (\frac{5\pi}{6}, \pi]$,  the span of ${\bf p_0}$, ${\bf p_{12}}$, ${\bf p_{43}}$ is a 3-dimensional subspace  of  $\mathbb{C}^{3,1}$. The intersection of ${\bf H}^{3}_{\mathbb{C}} \subset {\bf P}^{3}_{\mathbb{C}}$ with the projection of this 3-dimensional subspace  into ${\bf P}^{3}_{\mathbb{C}}$ is denoted by $L$, then  $L$ is a totally geodesic ${\bf H}^{2}_{\mathbb{C}} \hookrightarrow {\bf H}^{3}_{\mathbb{C}}$.

We re-denote ${\bf p_0}$, ${\bf p_{12}}$ and  ${\bf p_{43}}$  by ${\bf e_1}$, ${\bf e_2}$ and ${\bf e_3}$.
We denote $H_{L}=H(p_0,p_{12},p_{43})$ by the matrix $({\bf e_{i}}^* \cdot H \cdot {\bf e_{j}})_{1 \leq i,j \leq 3}$, then 
$$H_{L}=\begin{pmatrix}
-1&-\rm{e}^{-2\theta\rm{i}}+\rm{e}^{\theta\rm{i}}-1&-\rm{e}^{2\theta\rm{i}}+\rm{e}^{-\theta\rm{i}}-1\\
-\rm{e}^{2\theta\rm{i}}+\rm{e}^{-\theta\rm{i}}-1 & -1& D_5	\\
-\rm{e}^{-2\theta\rm{i}}+\rm{e}^{\theta\rm{i}}-1& 	\overline{D}_5 & -1\\
\end{pmatrix},$$
where 
\begin{flalign} 	\nonumber &
D_5=-32 \cos^3(\theta)\rm{e}^{\theta\rm{i}}+6\rm{i}\sin(2 \theta)+2\rm{i}\sin( \theta)+24\cos^2(\theta)+4 \cos(\theta)-3.  \end{flalign}
Now $\det(H_{L})$ is 
$$-64\cos^6(\theta)+160\cos^5(\theta)+144\cos^4(\theta)-168\cos^3(\theta)-96\cos^2(\theta)+40\cos(\theta)+20.$$	So $H_{L}$ is the Hermitian form with signature $(2,1)$ on the subspace with the basis $\{ { \bf e_1}, { \bf e_2}, {\bf e_3} \}$ when $\theta \in (\frac{5\pi}{6}, \pi]$.

The vector $x_1 {\bf e_1}+x_2 {\bf e_2}+x_3 {\bf e_3}$ is denoted by the vector ${\bf x}$ 	in  $\mathbb{C}^{3}$, and the vector $y_1 {\bf e_1}+y_2 {\bf e_2}+y_3 {\bf e_3}$  is denoted by the vector ${\bf y}$ 	in  $\mathbb{C}^{3}$, here 
\begin{equation}\nonumber
{\bf x}=\left(
\begin{array}{c}
x_1 \\
x_2 \\
x_3 \\
\end{array}
\right),\quad
{\bf y}=\left(
\begin{array}{c}
y_1\\
y_2\\
y_3 \\
\end{array}
\right)
\end{equation}
with $x_{i},y_{i} \in \mathbb{C}$.
So \begin{equation}\nonumber
{\bf E_1}=\left(
\begin{array}{c}
1 \\
0 \\
0 \\
\end{array}
\right),\quad
{ \bf E_2}=\left(
\begin{array}{c}
0\\
1\\
0 \\
\end{array}
\right),\quad
{ \bf E_3}=\left(
\begin{array}{c}
0\\
0\\
1 \\
\end{array}
\right)
\end{equation}
	in  $\mathbb{C}^{3}$	present the vectors ${ \bf e_1}$, ${ \bf e_2}$ and ${ \bf e_3}$ in  $\mathbb{C}^{3,1}$.

%

Comparing to Subsection \ref{subsection:Spinalcoordinates}, we define  the   Hermitian cross-product $\boxtimes_{L}$ on the subspace   $H_{L}$ (isometric to ${\bf H}^{2}_{\mathbb{C}}$) with respect to the basis  $\{\bf{e_1}, \bf{e_2}, \bf{e_3}\}$ by
$$\mathbf{x} \boxtimes_{L}
\mathbf{y}=
\left[\begin{matrix}
\mathbf{x}^*H_{L}(1,2) \cdot \mathbf{y}^*H_{L}(1,3)-\mathbf{y}^*H_{L}(1,2) \cdot \mathbf{x}^*H_{L}(1,3)\\ \mathbf{x}^*H_{L}(1,3) \cdot \mathbf{y}^*H_{L}(1,1)-\mathbf{y}^*H_{L}(1,3) \cdot \mathbf{x}^*H_{L}(1,1)\\ \mathbf{x}^*H_{L}(1,1) \cdot \mathbf{y}^*H_{L}(1,2)-\mathbf{y}^*H_{L}(1,1) \cdot \mathbf{x}^*H_{L}(1,2)
\end{matrix}
\right].$$
Then the intersection  $B_{12}\cap B_{43} \cap L$ is  parameterized   by $V=V(z_1,z_2) \in \mathbb{C}^3$ with $\langle V,V \rangle <0$ with respect to the Hermitian form $H_{L}$.
Where
$$V={ \bf E_2}\boxtimes_{L} { \bf E_3}+z_1 \cdot { \bf E_1}\boxtimes_{L} { \bf E_3}+z_2 \cdot { \bf E_1}\boxtimes_{L} { \bf E_2}$$  and  $(z_1,z_2) =(\rm{e}^{r \rm{i}},\rm{e}^{s \rm{i}})\in \mathbb{S}^1\times \mathbb{S}^1$.
	The vector 	$V \in \mathbb{C}^3$ is a three-by-one matrix, and  $$V(1,1) \cdot { \bf e_1}+V(2,1) \cdot { \bf e_2}+V(3,1) \cdot { \bf e_3}$$ 
is a vector in in $\mathbb{C}^{3,1}$, we  denote it  by $W$. The projection of $W$ into ${\mathbb C}{\mathbf P}^{3}$ is a point in $B_{12}\cap B_{43} \cap L$ if $W$ is a negative vector.
Now  $W^{*}\cdot H \cdot W$ is a very complicated term. But with Maple, $$W^{*}\cdot H \cdot W=-2(1+2\cos(\theta))^3 \cdot h,$$ where $h$ is 
\begin{flalign} 	\nonumber &
2213-1992\cos(3\theta)+718\cos(r)+3041\cos(2\theta)+...
& \\
\nonumber	
&-3\cos(s-7\theta)-16
\cos(s+7\theta)+\cos(r-s-6\theta).&
&& \nonumber \end{flalign}
(We omit many terms of $h$).
Note that the first term of $W^{*}\cdot H \cdot W$ is positive when $\theta \in [\frac{5 \pi}{6}, \pi]$. 
By Maple, $h$ 
has minimum 1521.583 numerically when $r, s \in [-\pi,\pi]$ and $\theta \in [\frac{5 \pi}{6}, \pi]$ at a point when $\theta= \frac{5 \pi}{6}$. In particular, $W$ above is a  positive vector in $\mathbb{C}^{3,1}$, so $B_{12}\cap B_{43} \cap L=\emptyset$. Then by the projection from 
${\bf H}^{3}_{\mathbb{C}}$ to $L$, we have $B_{12}\cap B_{32}$ in ${\bf H}^{3}_{\mathbb{C}}$ is the empty set. This ends the  proof of   (\ref{item:B12B43empty3dim}).



Note that $B_{12}\cap B_{41}=J^{-1}(B_{12}\cap B_{23})$.  Since we have proved that $B_{12}\cap B_{23}=\emptyset$, we have $B_{12}\cap B_{41}=\emptyset$. 
 This ends the proof of   (\ref{item:B12B41empty3dim}).

For the proof of   (\ref{item:B12B24empty3dim}), so we consider $B_{12}\cap B_{24}$.
For any $\theta \in (\frac{5\pi}{6}, \pi]$,  the span of ${ \bf p_0}$, ${ \bf p_{12}}$, ${ \bf p_{24}}$ is a 3-dimensional subspace  of  $\mathbb{C}^{3,1}$. The intersection of ${\bf H}^{3}_{\mathbb{C}} \subset {\bf P}^{3}_{\mathbb{C}}$ with the projection of this 3-dimensional subspace  into ${\bf P}^{3}_{\mathbb{C}}$ is denoted by $L$, then  $L$ is a totally geodesic ${\bf H}^{2}_{\mathbb{C}} \hookrightarrow {\bf H}^{3}_{\mathbb{C}}$.

We re-denote $\bf{ p_0}$, $\bf{ p_{12}}$ and  $\bf{ p_{24}}$  by $\bf{ e_1}$, $\bf{ e_2}$ and $\bf{ e_3}$.
We denote $H_{L}=H(p_0,p_{12},p_{24})$ by the matrix $({\bf e_{i}}^* \cdot H\cdot {\bf e_{j}})_{1 \leq i,j \leq 3}$, then 
$$H_{L}=\begin{pmatrix}
-1&-\rm{e}^{-2\theta\rm{i}}+\rm{e}^{\theta\rm{i}}-1& 2\cos(\theta)\\
-\rm{e}^{2\theta\rm{i}}+\rm{e}^{-\theta\rm{i}}-1 & -1&	D_6\\
2\cos(\theta)& 	\overline{D}_{6} & -1\\
\end{pmatrix}, $$
where $$D_6=-8 \cos^2(\theta)\rm{e}^{\theta\rm{i}}-\rm{e}^{\theta\rm{i}}-\rm{e}^{-2\theta\rm{i}}-1.$$
Now $\det(H_{L})$ is 
$$32\cos^5(\theta)-24\cos^3(\theta)-4\cos^2(\theta)+4\cos(\theta)+1.$$	So $H_{L}$ is the Hermitian form with signature $(2,1)$ on the subspace with the basis $\{{\bf e_1}, {\bf e_2}, {\bf e_3}\}$ when $\theta \in (\frac{5\pi}{6}, \pi]$.

The vector $x_1 {\bf e_1}+x_2 {\bf e_2}+x_3 {\bf e_3}$ is denoted by the vector ${\bf x} \in \mathbb{C}^3$, and the vector $y_1 {\bf e_1}+y_2 {\bf e_2}+y_3 {\bf e_3}$  is denoted by the vector ${\bf y}\in \mathbb{C}^3$, here 
\begin{equation}\nonumber
{\bf x}=\left(
\begin{array}{c}
x_1 \\
x_2 \\
x_3 \\
\end{array}
\right),\quad
{\bf y}=\left(
\begin{array}{c}
y_1\\
y_2\\
y_3 \\
\end{array}
\right)
\end{equation}
with $x_{i},y_{i} \in \mathbb{C}$.
So \begin{equation}\nonumber
{\bf E_1}=\left(
\begin{array}{c}
1 \\
0 \\
0 \\
\end{array}
\right),\quad
{ \bf E_2}=\left(
\begin{array}{c}
0\\
1\\
0 \\
\end{array}
\right),\quad
{ \bf E_3}=\left(
\begin{array}{c}
0\\
0\\
1 \\
\end{array}
\right)
\end{equation}
	in  $\mathbb{C}^{3}$	present the vectors ${ \bf e_1}$, ${ \bf e_2}$ and ${ \bf e_3}$ in  $\mathbb{C}^{3,1}$.

Comparing to Subsection \ref{subsection:Spinalcoordinates}, we define  the   Hermitian cross-product $\boxtimes_{L}$ on the subspace  $H_{L}$ (which is  isometric to ${\bf H}^{2}_{\mathbb{C}}$) with respect to  the basis  $\{\bf{e_1}, \bf{e_2}, \bf{e_3}\}$ by
$$\mathbf{x} \boxtimes_{L}
\mathbf{y}=
\left[\begin{matrix}
\mathbf{x}^*H_{L}(1,2) \cdot \mathbf{y}^*H_{L}(1,3)-\mathbf{y}^*H_{L}(1,2) \cdot \mathbf{x}^*H_{L}(1,3)\\ \mathbf{x}^*H_{L}(1,3) \cdot \mathbf{y}^*H_{L}(1,1)-\mathbf{y}^*H_{L}(1,3) \cdot \mathbf{x}^*H_{L}(1,1)\\ \mathbf{x}^*H_{L}(1,1) \cdot \mathbf{y}^*H_{L}(1,2)-\mathbf{y}^*H_{L}(1,1) \cdot \mathbf{x}^*H_{L}(1,2)
\end{matrix}
\right].$$

Then the intersection  $B_{12} \cap B_{24} \cap L$ is  parameterized  by $V=V(z_1,z_2) \in \mathbb{C}^3$ with $\langle V,V \rangle <0$ with respect to the Hermitian form $H_{L}$.
Where
$$V={ \bf E_2}\boxtimes_{L} { \bf E_3}+z_1 \cdot { \bf E_1}\boxtimes_{L} { \bf E_3}+z_2 \cdot { \bf E_1}\boxtimes_{L} { \bf E_2}$$  and  $(z_1,z_2) =(\rm{e}^{r \rm{i}},\rm{e}^{s \rm{i}}) \in \mathbb{S}^1 \times \mathbb{S}^1$.
The vector 	$V \in \mathbb{C}^3$	is a three-by-one matrix,   we denote by  $V(i,1)$ the  entry in the $i$-th  row of $V$. Then  $$V(1,1) \cdot { \bf e_1}+V(2,1) \cdot { \bf e_2}+V(3,1) \cdot { \bf e_3}$$   is a vector in $\mathbb{C}^{3,1}$, we also denote it  by $W$. The projection of $W$ into ${\mathbb C}{\mathbf P}^{3}$ is a point in $B_{12}\cap B_{24} \cap L$ if $W$ is a negative vector.
Now  $W^{*}\cdot H \cdot W$ is a very complicated term. But with Maple, $$W^{*}\cdot H \cdot W=-(1+2\cos(\theta))^3\cdot h,$$
where $h$ is 
\begin{flalign} 	\nonumber &
799-798\cos(3\theta)+1180\cos(2\theta)-222\cos(5\theta)-24\cos(7\theta)-86\cos(s-4\theta)...
& \\
\nonumber	
&236\cos(r-2\theta)+110\cos(r+2\theta).&
&& \nonumber \end{flalign}
(We omit many terms of $h$).
Note that the first term  of $W^{*}\cdot H \cdot W$ is positive when $\theta \in [\frac{5 \pi}{6}, \pi]$. 
By Maple, $h$ has  minimum 275.152 numerically when $r, s \in [-\pi,\pi]$ and $\theta \in [\frac{5 \pi}{6}, \pi]$ at a point when $\theta= \frac{5 \pi}{6}$. In particular, any $W$ above is a  positive vector in $\mathbb{C}^{3,1}$, so $B_{12}\cap B_{24} \cap L=\emptyset$. Then by the projection from 
${\bf H}^{3}_{\mathbb{C}}$ to $L$, we have $B_{12}\cap B_{24}$ in ${\bf H}^{3}_{\mathbb{C}}$ is the empty set. This ends the  proof of   (\ref{item:B12B24empty3dim}).

Now we consider (\ref {item:B12nonempty3dim}). Recall in Lemma  \ref{lemma:123413coplane}, we have proved that ${\bf p_0}$, ${\bf p_{12}}$, ${\bf p_{34}}$ and ${\bf p_{13}}$ are co-planar.  
So for any $\theta \in (\frac{5\pi}{6}, \pi]$,  the span of ${\bf p_0}$, ${\bf p_{12}}$, ${\bf p_{34}}$ and ${\bf p_{13}}$ is a 3-dimensional subspace  of  $\mathbb{C}^{3,1}$. The intersection of ${\bf H}^{3}_{\mathbb{C}} \subset {\bf P}^{3}_{\mathbb{C}}$ with the projection of this 3-dimensional subspace  into ${\bf P}^{3}_{\mathbb{C}}$ is denoted by $L$, then  $L$ is a totally geodesic ${\bf H}^{2}_{\mathbb{C}} \hookrightarrow {\bf H}^{3}_{\mathbb{C}}$.

By definition $p_0 \in L$. We will show if the intersection  $B_{12}\cap B_{34} \cap L$ is non-empty, then it is  a  Giraud disk  in $L$. Moreover, in this case,  $B_{12}\cap B_{34}\cap L$ lies in the half space of  $L-B_{13}$  which does not contain the fixed point $p_0$ of $J$. In particular,  $B_{12}\cap B_{34}$ does not lie in the partial Dirichlet domain $D_{R}$.

We re-denote ${\bf p_0}$, ${\bf p_{12}}$ and  ${\bf p_{34}}$  by ${\bf e_1}$, ${\bf e_2}$ and ${\bf e_3}$.
We denote $H_{L}=H(p_0,p_{12},p_{34})$ by the matrix $( {\bf e_{i}}^* \cdot H \cdot {\bf e_{j}})_{1 \leq i,j \leq 3}$, then 
$$H_{L}=\begin{pmatrix}
-1&-\rm{e}^{-2\theta\rm{i}}+\rm{e}^{\theta\rm{i}}-1& -\rm{e}^{-2\theta\rm{i}}+\rm{e}^{\theta\rm{i}}-1\\
-\rm{e}^{2\theta\rm{i}}+\rm{e}^{-\theta\rm{i}}-1 & -1&	16 \cos^3(\theta)-8\cos(\theta)-3\\
-\rm{e}^{2\theta\rm{i}}+\rm{e}^{-\theta\rm{i}}-1& 	16 \cos^3(\theta)-8\cos(\theta)-3 & -1\\
\end{pmatrix}. $$
Now $\det(H_{L})$ is 
$$128\cos^5(\theta)-96\cos^3(\theta)-16\cos^2(\theta)+16\cos(\theta)+4.$$	So $H_{L}$ is the Hermitian form with signature $(2,1)$ on the subspace with the basis  $\{\bf{e_1}, \bf{e_2}, \bf{e_3}\}$ when $\theta \in (\frac{5\pi}{6}, \pi]$.

 The vector $x_1 {\bf e_1}+x_2 {\bf e_2}+x_3 {\bf e_3}$ is denoted by the vector ${\bf x}$, and the vector $y_1 {\bf e_1}+y_2 {\bf e_2}+y_3 {\bf e_3}$  is denoted by the vector ${\bf y}$, here 
\begin{equation}\nonumber
{\bf x}=\left(
\begin{array}{c}
x_1 \\
x_2 \\
x_3 \\
\end{array}
\right),\quad
{\bf y}=\left(
\begin{array}{c}
y_1\\
y_2\\
y_3 \\
\end{array}
\right)
\end{equation}
with $x_{i},y_{i} \in \mathbb{C}$.
So \begin{equation}\nonumber
{\bf E_1}=\left(
\begin{array}{c}
1 \\
0 \\
0 \\
\end{array}
\right),\quad
{ \bf E_2}=\left(
\begin{array}{c}
0\\
1\\
0 \\
\end{array}
\right),\quad
{ \bf E_3}=\left(
\begin{array}{c}
0\\
0\\
1 \\
\end{array}
\right)
\end{equation}
	in  $\mathbb{C}^{3}$	present the vectors ${ \bf e_1}$, ${ \bf e_2}$ and ${ \bf e_3}$ in  $\mathbb{C}^{3,1}$.

Comparing to Subsection \ref{subsection:Spinalcoordinates}, we define  the   Hermitian cross-product $\boxtimes_{L}$ on the subspace  $H_{L}$ (which is isometric to ${\bf H}^{2}_{\mathbb{C}}$)  with respect to the basis   $\{\bf{e_1}, \bf{e_2}, \bf{e_3}\}$ by
$$\mathbf{x} \boxtimes_{L}
\mathbf{y}=
\left[\begin{matrix}
\mathbf{x}^*H_{L}(1,2) \cdot \mathbf{y}^*H_{L}(1,3)-\mathbf{y}^*H_{L}(1,2) \cdot \mathbf{x}^*H_{L}(1,3)\\ \mathbf{x}^*H_{L}(1,3) \cdot \mathbf{y}^*H_{L}(1,1)-\mathbf{y}^*H_{L}(1,3) \cdot \mathbf{x}^*H_{L}(1,1)\\ \mathbf{x}^*H_{L}(1,1) \cdot \mathbf{y}^*H_{L}(1,2)-\mathbf{y}^*H_{L}(1,1) \cdot \mathbf{x}^*H_{L}(1,2)
\end{matrix}
\right].$$

Then the intersection  $B_{12}\cap B_{34} \cap L$ is  parameterized  by $V=V(z_1,z_2) \in \mathbb{C}^3$ with $\langle V,V \rangle <0$ with respect to the Hermitian form $H_{L}$.
Where
$$V={ \bf E_2}\boxtimes_{L} {\bf E_3}+z_1 \cdot { \bf E_1}\boxtimes_{L} { \bf E_3}+z_2 \cdot { \bf E_1}\boxtimes_{L} { \bf E_2}$$  and  $(z_1,z_2) =(\rm{e}^{r \rm{i}},\rm{e}^{s \rm{i}})\in \mathbb{S}^1 \times \mathbb{S}^1$.
The vector 	$V \in \mathbb{C}^3$ is a three-by-one matrix, and  $$V(1,1) \cdot { \bf e_1}+V(2,1) \cdot { \bf e_2}+V(3,1) \cdot { \bf e_3}$$   is a vector  in $\mathbb{C}^{3,1}$, we also denote it  by $W$.
If $W$ is a negative vector, we consider the distances of $[W] \in {\bf H}^{3}_{\mathbb{C}}$  to $p_0$ and $p_{13}$. That is, we should consider  $|W^{*} \cdot  H \cdot \mathbf{p_0}|^2$ and  $|W^{*} \cdot  H \cdot \mathbf{p_{13}}|^2$.
Now  $W^{*}\cdot H \cdot W$ is a very complicated term. But with Maple, 
the minimum of $$|W^{*} \cdot  H \cdot \mathbf{p_0}|^2-|W^{*} \cdot  H \cdot \mathbf{p_{13}}|^2$$ with the condition 
$W^{*}\cdot H \cdot W=0$ is 99.42 when $r, s \in [-\pi,\pi]$ and $\theta \in [\frac{5 \pi}{6}, \pi]$ at a point when $\theta= \frac{5 \pi}{6}$.
And the maximum of $$|W^{*} \cdot  H \cdot \mathbf{p_0}|^2-|W^{*} \cdot  H \cdot \mathbf{p_{13}}|^2$$ with the condition 
$W^{*}\cdot H \cdot W=0$ is 921.79
 when $r, s \in [-\pi,\pi]$ and $\theta \in [\frac{5 \pi}{6}, \pi]$ at a point $\theta= 2.734$.
This implies that  $$B_{12}\cap B_{34}\cap L\cap \partial{\bf H}^{3}_{\mathbb{C}}$$ does not intersect $B_{13}$. Moreover, if the quadruple intersection  above is non-empty, then it lies in the half-space of $L - B_{13}$ which does not contain $p_0$. Then by the well-known properties of intersections of bisectors in ${\bf H}^{2}_{\mathbb{C}}$, we have  $$B_{12}\cap B_{34} \cap B_{13}\cap L=\emptyset.$$
Moreover, if  the triple  intersection $B_{12}\cap B_{34} \cap L$ is non-empty, then it lies in the half-space of $L - B_{13}$ which does not contain $p_0$.
 Then by the projection from 
${\bf H}^{3}_{\mathbb{C}}$ to $L$, we have $$B_{12}\cap B_{34} \cap B_{13}=\emptyset$$ in
 ${\bf H}^{3}_{\mathbb{C}}$.
 Moreover, if the  intersection   $B_{12}\cap B_{34} $ is non empty,  then it lies in the half-space of ${\bf H}^{3}_{\mathbb{C}} - B_{13}$ which does not contain $p_0$. So   $B_{12}\cap B_{34}$ does not lie in the partial Dirichlet domain $D_{R}$ even if it is not empty.

The author remarks that when $\theta$ is near to $\frac{5 \pi}{6}$,  for instance we can take a sample point $$\{r= \pi,~~s=0, ~~\theta=\frac{5 \pi}{6}\},$$ then $W^{*}\cdot H \cdot W$ is $-32\sqrt{3}-32$, which is negative,  and then $B_{12}\cap B_{34}\cap L$ is non-empty.
But  when $\theta=\pi$, it is not difficult to show $B_{12}\cap B_{34}\cap L=\emptyset$. 
So there is a neighborhood of $\pi$ in $[\frac{5 \pi}{6}, \pi]$, such that $B_{12}\cap B_{34}\cap L=\emptyset$ when $\theta$ lies in this neighborhood.
It seems that when  $\theta \in [2.74, \pi]$, then $W^{*}\cdot H \cdot W$ is always positive.  When $\theta=2.71$,     the graph of $W^{*}\cdot H \cdot W=h$ as a function of $(r,s)$ is the blue surface in Figure \ref{figure:B12capB343dim}.
This ends the proof of  (\ref {item:B12nonempty3dim}).

 \begin{figure}
 	\begin{center}
 		\begin{tikzpicture}
 		\node at (0,0) {\includegraphics[width=6cm,height=6cm]{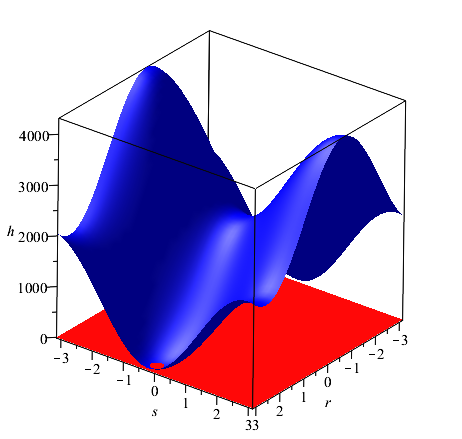}};
 		\end{tikzpicture}
 	\end{center}
 	\caption{An illustration $B_{12}\cap B_{34}\cap L=\emptyset$ when $\theta$ is near to $\pi$. 
 		The blue surface is the graph of the function  $h(r,s)=W^{*}\cdot H \cdot W$ when $\theta=2.71$. The red plane is $h=0$. They intersect in a small circle near $(r,s)=(\pi,0)$.  When increasing $\theta$ from $2.71$ to $\pi$, the graph of the function $W^{*}\cdot H \cdot W$  is also a similar blue surface   which tends to disjoint from the red plane $h=0$.}
 	\label{figure:B12capB343dim}
 \end{figure}

 \end{proof}

\begin{prop}\label{prop:B133dim} For all $\theta \in (\frac{5 \pi}{6}, \pi]$, the bisector $B_{13}$ of $I_1I_3$, we have $B_{13}$ does not intersect  $B_{21}$, $B_{23}$, $B_{43}$,  $B_{41}$ and $B_{24}$.
\end{prop}

\begin{proof}We have proved $B_{12}\cap B_{24}=\emptyset$, by the $\langle J \rangle= \mathbb{Z}_4$ symmetry, we have $$B_{13}\cap B_{23}=\emptyset, ~~~~B_{13}\cap B_{41}=\emptyset.$$
The fact 	$$B_{13}\cap B_{21}=\emptyset, ~~~B_{13}\cap B_{43}=\emptyset$$ can be proved similar to the proof of  $B_{12}\cap B_{24}=\emptyset$ in Proposition \ref{prop:B123dim}.

 Now we  consider $B_{13}\cap B_{24}$. Since $$p_{13}=[\sqrt{2 \cos(2 \theta)+1},~~0,~~0, ~~-2 \cos(\theta)]^{t},$$
	and $$p_{24}=[-\sqrt{2 \cos(2 \theta)+1},~~0,~~0, -2 \cos(\theta)]^{t}. $$ So $p_0$, $p_{13}$ and $p_{24}$ lie in the $\mathbb{C}$-line 
		$$l=\left\{	\left[z_1,~~0,~~0, ~~1\right]^{t} \in {\bf H}^3_{\mathbb C}\right\}.$$
		Now it is easy to see $B_{13}\cap B_{24} \cap l=\emptyset$. Then from the projection of $ {\bf H}^3_{\mathbb C}$ to $l$, we get  $B_{13}\cap B_{24}=\emptyset$.

\end{proof}

\begin{prop}\label{prop:B12intersection3dim} For all $\theta \in (\frac{5 \pi}{6}, \pi]$, the bisector $B_{12}$ of $I_1I_2$, we have
	\begin{enumerate}
		
		\item  \label{item:B12andB13andB143dim} the triple intersection $B_{12}\cap B_{13} \cap B_{14}$ is a non empty 3-dimensional object;
		
		\item the intersection  $B_{12}\cap B_{13}$ is a Giraud disk, which is a 4-ball in ${\bf H}^{3}_{\mathbb{C}}$;
		
			\item the intersection   $B_{12}\cap B_{14}$ is a Giraud disk, which is a 4-ball in ${\bf H}^{3}_{\mathbb{C}}$;
				\item the intersection   $B_{13}\cap B_{14}$ is a Giraud disk, which is a 4-ball in ${\bf H}^{3}_{\mathbb{C}}$.

		
		
	\end{enumerate}
\end{prop}
\begin{proof} In (\ref{equation:tripleintersection}), we take ${\bf p_0}={ \bf q_0}$, ${\bf p_{12}}={\bf q_1}$, ${\bf p_{13}}={\bf q_2}$ and  ${\bf p_{14}}={\bf q_3}$, then we can  parameterize  $B(p_0, p_{12}) \cap B(p_0, p_{13}) \cap B(p_0, p_{14})$ by 
	\begin{equation} \label{item:B12andB13abdB143dim}
	 V(z_1,z_2,z_3)={\bf p_{12}} \boxtimes {\bf  p_{13} }\boxtimes {\bf p_{14}}+z_1 \cdot {\bf p_0} \boxtimes {\bf  p_{13}} \boxtimes {\bf p_{14}}+z_2 \cdot {\bf p_0} \boxtimes {\bf p_{12}} \boxtimes {\bf p_{14}}+z_3 \cdot {\bf p_0} \boxtimes {\bf p_{12}}\boxtimes {\bf p_{13}}\end{equation}
in $\mathbb{C}^{3,1}$	with $(z_1,z_2,z_3)=(\rm{e}^{r\rm{i}},\rm{e}^{s\rm{i}},\rm{e}^{t\rm{i}}) \in \mathbb{S}^1 \times\mathbb{S}^1 \times\mathbb{S}^1$ such that  $\langle V(z_1,z_2,z_3), V(z_1,z_2,z_3)\rangle$ is negative.
	
	Now 
	$$\langle  V(z_1,z_2,z_3),V(z_1,z_2,z_3)\rangle=(16\cos(\theta)^3+8\cos(\theta)^2-4\cos(\theta)-2) \cdot h,$$ 
	with
	 \begin{flalign} \nonumber &
	h= -2+2\cos(s)+4\cos(\theta)+2\cos(r-t)-\cos(-\theta+r-t)-\cos(\theta+r-t)
	 & \\
	 \nonumber	&-\cos(-\theta+s)-\cos(\theta+s)-\cos(-2\theta-s+r)+\cos(2\theta-s+r)-\cos(-2\theta+s-t)& \\
	 \nonumber	&+\cos(2\theta+s-t)+\cos(-2\theta+r)-\cos(2\theta+r)-\cos(-2\theta+t)+\cos(2\theta+t)&
	 \\
	 \nonumber &-4\cos(2\theta)+\cos(-s+r-\theta)-\cos(-s+r+\theta)+\cos(s-t-\theta)-\cos(s-t+\theta)&
	  \\
	 \nonumber &-\cos(-\theta+r)+\cos(\theta+r)+\cos(-\theta+t)-\cos(\theta+t)+\cos(-s+r-3\theta)&
	  \\
	 &+\cos(s-t-3\theta)+\cos(r+3\theta)+\cos(t-3\theta).&
	 && \nonumber \end{flalign}
The term $$16\cos(\theta)^3+8\cos(\theta)^2-4\cos(\theta)-2$$ is always negative when $\theta \in ( \frac{5 \pi}{6}, \pi]$. 
	When $r= -\pi$, $s=0$ and $t=\pi$, the term $h$ is  $$-4\cos^3(\theta)-2\cos^2(\theta)+3\cos(\theta)+\frac{3}{2},$$
	which is always positive when $\theta \in (\frac{5 \pi}{6}, \pi]$. 
	This means that  $V=V(\rm{e}^{-\pi\rm{i}},\rm{e}^{0\rm{i}},\rm{e}^{\pi\rm{i}})$ is a point in ${\bf H}^3_{\mathbb C}$ for any $\theta \in (\frac{5 \pi}{6}, \pi]$.  So $B_{12}\cap B_{13}\cap B_{14}$ is non-empty for  any $\theta$.  Then for any  $(r,s,t)$ in  a small neighborhood (depends on $\theta$) of  $(-\pi,0,  \pi)$ in $\mathbb{R}^3$, the corresponding $V$ is also a  negative vector.  This proves (\ref{item:B12andB13andB143dim}) of Proposition \ref{prop:B12intersection3dim}.  The other terms of Proposition \ref{prop:B12intersection3dim} is trivial now. 
\end{proof}

	See Figure \ref{figure:B12B13B143dim} for the surface $\langle V(z_1,z_2,z_3), V(z_1,z_2,z_3)\rangle=0$ in the proof of Proposition \ref{prop:B133dim}  when $\theta=2.7> \frac{5 \pi}{6}$, which seems to be a 2-sphere (we do not need this fact in this paper). Which is an evidence for Question  \ref{question:3-ball}  in Section \ref{sec:intro}.
	 \begin{figure}
		\begin{center}
			\begin{tikzpicture}
			\node at (0,0) {\includegraphics[width=6cm,height=6cm]{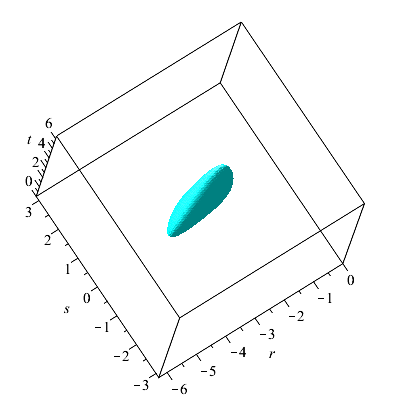}};
			\end{tikzpicture}
		\end{center}
		\caption{The cyan surface  is the locus of $V=V(z_1,z_2,z_3)$ in (\ref{item:B12andB13abdB143dim}) with the condition $\langle V, V\rangle=0$  when $\theta=2.7$, which seems to be a 2-sphere in $ \mathbb{S}^1 \times\mathbb{S}^1 \times\mathbb{S}^1$.}
		\label{figure:B12B13B143dim}
	\end{figure}

By the $\langle J\rangle =\mathbb{Z}_{4}$ symmetry, we have similar properties of  other triple intersections of bisectors.  

From Propositions \ref{prop:B123dim} and  \ref{prop:B133dim}, there is no quadruple intersection of bisectors $B_{ij}$ for $I_{ij} \in R$. 
We take the side $s_{ij}=B_{ij}\cap D_{R}$ when $j=i \pm 1$ or $j=i+2$  mod 4 and $i=1,2,3,4$. 
By taking sample points, it is easy to see $s_{ij}$ is non-empty, so each  $s_{ij}$ is a nonempty  5-dimensional object in ${\bf H}^3_{\mathbb{C}}$.  Since $$I_2I_1(s_{12})=s_{21},~~~J(s_{12})=s_{23},$$ there are only two isometric types of 5-facets, say $s_{12}$ and $s_{13}$. 
Since 
 $$J(s_{14}\cap s_{12})=s_{21}\cap s_{23},~~~~J(s_{12}\cap s_{13})=s_{23}\cap s_{24},~~~~I_2I_1(s_{12}\cap s_{13})=s_{21}\cap s_{23},$$
 there is only one   isometric types of 4-facets, say $s_{12}\cap s_{14}$. There is only one isometric type of 3-facets, say $s_{12}\cap s_{13}\cap s_{14}$. In particular, there is no quadruple intersection of bisectors.   So we do not need the precisely combinatorial structure of the facets. What we really need is the above  mentioned facets are all non-empty. Which is enough for the  Poincar\'e polyhedron theorem in our (lucky) case.
We have the following propositions.

\begin{prop}\label{prop:s123dim}	
	
 For all $\theta \in (\frac{5 \pi}{6}, \pi]$, the side  $s_{12}=B_{12}\cap D_{R}$  is a non-empty  5-dimensional object in $ {\bf H}^3_{\mathbb{C}} \cup \partial {\bf H}^3_{\mathbb{C}}$. Moreover,    \begin{itemize} 
 	\item the frontier of $s_{12} \cap  {\bf H}^3_{\mathbb{C}}$   consists of two non-empty 4-dimensional objects $s_{12}\cap s_{14}$ and $s_{12}\cap s_{13}$;
 	
 	\item  the intersection of $s_{12}\cap s_{14}$ and $s_{12}\cap s_{13}$ is the non-empty 3-dimensional object  $s_{12}\cap s_{13}\cap s_{14}$.
 	\end{itemize}
\end{prop}

\begin{prop}\label{prop:s133dim}	
	
 For all $\theta \in (\frac{5 \pi}{6}, \pi]$,	the side  $s_{13}=B_{13}\cap D_{R}$  is a non-empty  5-dimensional object in $ {\bf H}^3_{\mathbb{C}} \cup \partial {\bf H}^3_{\mathbb{C}}$. Moreover,   \begin{enumerate}  
 	\item \label{item:4dimobject}the frontier of $s_{13} \cap  {\bf H}^3_{\mathbb{C}}$  are two disjoint 4-dimensional objects;
 	
 	\item  each of the 4-dimensional objects in (\ref{item:4dimobject})   consists of two parts:
 	
 	\begin{itemize}  
 		\item  the union of $s_{13}\cap s_{12}$ and $s_{13}\cap s_{14}$ is a component of the frontier of $s_{13} \cap  {\bf H}^3_{\mathbb{C}}$. The intersection of  $s_{13}\cap s_{12}$ and $s_{13}\cap s_{14}$  is the non-empty 3-dimensional object  $s_{12}\cap s_{13}\cap s_{14}$; 
 		\item The union of $s_{13}\cap s_{32}$ and $s_{13}\cap s_{34}$ is the other  component of  the frontier of $s_{13} \cap  {\bf H}^3_{\mathbb{C}}$. The intersection of $s_{13}\cap s_{32}$ and $s_{13}\cap s_{34}$ is the non-empty 3-dimensional object  $s_{13}\cap s_{32}\cap s_{34}$.
 		\end{itemize} 
 	\end{enumerate} 
\end{prop}

\subsection{Proof of Theorem  \ref{thm:complex3dim}}\label{subsec:poincare3dim}
After above propositions, the proof of   Theorem  \ref{thm:complex3dim} for general $\theta \in (\frac{5 \pi}{6}, \pi]$ is similar to the proof when $\theta =\frac{5 \pi}{6}$.

We have the  side pairing maps of $D_{R}$ as in Subsection \ref{subsec:poincare2dim}.


For example, since $I_1I_3$  is a self-homeomorphism of $s_{13}$, $I_1I_3$
exchanges the two components  of ${\bf H}^{3}_{\mathbb C}-B_{13}$. We can see that
$D_{R}$ and $I_1I_3(D_{R})$
have disjoint interiors, and they together  cover a neighborhood  of
each point in the interior of the side $s_{13}$.  The cases of the other 5-sides are similar.

We now consider the  tessellation about ridges. Recall that $A_1=I_1I_2$, $A_2=I_2I_3$, $A_3=I_3I_4$ and $A_4=I_4I_1$. 
  For the ridge  $s_{14}\cap s_{12}$, we have proven it is a nonempty 4-dimensional object. 
   The ridge circle of it is $$s_{14}\cap s_{12} \xrightarrow{A^{-1}_1}s_{21}\cap s_{24}\xrightarrow{(A_2A_3)^{-1}}s_{24}\cap s_{41}\xrightarrow{A^{-1}_4}s_{14}\cap s_{12}.$$
Which gives the relation $A^{-1}_4 \cdot A^{-1}_3A^{-1}_2 \cdot A^{-1}_1=id$. 
What we really need (and we proved) is that $s_{14}\cap s_{12}$,  $s_{21}\cap s_{24}$ and $s_{24}\cap s_{41}$ are all non empty. Moreover $$A^{-1}_1(s_{12}\cap s_{13}\cap s_{14})=s_{21}\cap s_{23}\cap s_{24}.$$ Then the  ridge circle relation is proved by checking only the group action on the points $p_0$, $p_{14}$, $p_{12}$ and  $p_{24}$. 
A standard argument implies $D_{R} \cup A_1(D_{R}) \cup A^{-1}_4(D_{R})$ covers a small neighborhood of  $s_{14}\cap s_{12}$.

 For the ridge  $s_{13}\cap s_{14}$,  the ridge circle is $$s_{13}\cap s_{14} \xrightarrow{A_4}s_{41}\cap s_{43}\xrightarrow{A_3}s_{34}\cap s_{31}\xrightarrow{A_1A_2}s_{13}\cap s_{14}.$$
Which gives the relation $A_1A_2 \cdot A_3\cdot A_4=id$.  We have also proved that  $s_{13}\cap s_{14}$,  $s_{41}\cap s_{43}$ and $s_{34}\cap s_{31}$ are all non empty. A standard argument implies $D_{R} \cup A^{-1}_4(D_{R}) \cup (A_1A_2)^{-1}(D_{R})$ covers a small neighborhood of  $s_{13}\cap s_{14}$.

The other ridge circles can be proved similar.

Similar to  the proof when $\theta=\frac{5 \pi}{6}$,  the identification space of $D_{R}$ by these side pairing maps is complete.


By  Poincar\'e polyhedron theorem, the partial Dirichlet domain $D_{R}$ is in fact the Dirichlet domain
of $\rho_{\theta}(K) <\mathbf{PU}(3,1)$.
 Now we have the presentation 
 $$\rho_{\theta}(K)=\left\langle  A_1, A_2, A_3, A_4 \Bigg| \begin{matrix}     (A_1A_2)^2=   (A_2A_3)^2 =  A_1A_2  A_3A_4=id
 \end{matrix}\right\rangle $$
when $\theta \in (\frac{5 \pi}{6}, \pi]$.   So we have a discrete and  faithful representation of $K$. This ends the proof of  Theorem  \ref{thm:complex3dim}.


\bibliographystyle{amsplain}

\end{document}